\documentclass[a4paper,11pt]{amsart}
\usepackage[utf8]{inputenc}
\usepackage{url}
\usepackage{amsmath}
\usepackage{amsthm}
\usepackage{enumitem}
\usepackage[utf8]{inputenc}
\usepackage{cite}
\usepackage{ulem}
\usepackage{pdfpages}

\usepackage{graphicx}
\usepackage{amsmath}
\usepackage{amssymb}
\usepackage[active]{srcltx}

\usepackage[a4paper]{geometry}
\theoremstyle{plain}
\newtheorem{theorem}{Theorem}[section]

\newtheorem{lemma}[theorem]{Lemma}
\newtheorem{corollary}[theorem]{Corollary}

\newtheorem{proposition}[theorem]{Proposition}
\newtheorem{remark}[theorem]{Remark}

\theoremstyle{definition}

\newtheorem*{assumptions*}{Assumptions}
\newtheorem*{key*}{Key}
\usepackage{amssymb}
\usepackage{amsfonts}
\usepackage{tikz}
\usepackage{mathtools}
\usepackage{textcomp}

\DeclarePairedDelimiter\floor{\lfloor}{\rfloor}

\usepackage{xcolor,cancel}

\makeatletter
\newsavebox\myboxA
\newsavebox\myboxB
\newlength\mylenA

\renewcommand*\env@matrix[1][*\c@MaxMatrixCols c]{%
  \hskip -\arraycolsep
  \let\@ifnextchar\new@ifnextchar
  \array{#1}}

\newcommand*\xoverline[2][0.75]{%
    \sbox{\myboxA}{$\m@th#2$}%
    \setbox\myboxB\null
    \ht\myboxB=\ht\myboxA%
    \dp\myboxB=\dp\myboxA%
    \wd\myboxB=#1\wd\myboxA
    \sbox\myboxB{$\m@th\overline{\copy\myboxB}$}
    \setlength\mylenA{\the\wd\myboxA}
    \addtolength\mylenA{-\the\wd\myboxB}%
    \ifdim\wd\myboxB<\wd\myboxA%
       \rlap{\hskip 0.5\mylenA\usebox\myboxB}{\usebox\myboxA}%
    \else
        \hskip -0.5\mylenA\rlap{\usebox\myboxA}{\hskip 0.5\mylenA\usebox\myboxB}%
    \fi}
\makeatother

\makeatletter
\newcommand*{\ov}[1]{%
  $\m@th\overline{\mbox{#1}}$%
}
\newcommand*{\ovA}[1]{%
  $\m@th\overline{\mbox{#1}\raisebox{3mm}{}}$%
}
\newcommand*{\ovB}[1]{%
  $\m@th\overline{\mbox{#1\rule{0pt}{3mm}}}$%
}
\newcommand*{\ovC}[1]{%
  $\m@th\overline{\mbox{#1\strut}}$%
}
\newcommand*{\ovD}[1]{%
  $\m@th\overline{\mbox{#1\vphantom{\"A}}}$%
}
\newcommand*{\ovE}[1]{%
  $\m@th\overline{\raisebox{0pt}[1.2\height]{#1}}$%
}
\newcommand*{\ovF}[1]{%
  $\m@th\overline{\raisebox{0pt}[\dimexpr\height+1mm\relax]{#1}}$%
}
\newcommand*{\ovG}[1]{%
  $\m@th\overline{\raisebox{0pt}[\dimexpr\height+1mm\relax]{#1\vphantom{A}}}$%
}
\makeatother

\numberwithin{equation}{section}

\newcommand{\un}{\underline}
\newcommand{\ti}{\textit}
\newcommand{\x}{\underline{x}}
\newcommand{\y}{\underline{y}}
\newcommand{\z}{\underline{z}}
\newcommand{\h}{\underline{h}}
\newcommand{\m}{\underline{m}}
\newcommand{\p}{\underline{p}}

\newcommand{\ut}{\underline{t}}
\newcommand{\uu}{\underline{u}}
\newcommand{\uv}{\underline{v}}
\newcommand{\ua}{{\underline{a}}}
\newcommand{\mfb}{{\underline{\mathfrak{b}}}}

\newcommand{\uc}{\underline{c}}
\newcommand{\uR}{\underline{R}}
\newcommand{\F}{F_{\underline{h}}}

\newcommand{\mb}{\mathbb}
\newcommand{\ep}{\epsilon}
\newcommand{\Z}{\mathbb{Z}}
\newcommand{\Q}{\mathbb{Q}}
\newcommand{\N}{\mathbb{N}}
\newcommand{\R}{\mathbb{R}}
\newcommand{\C}{\mathbb{C}}

\newcommand{\Co}{\Big{\{}}
\newcommand{\Cc}{\Big{\}}}
\newcommand{\Bo}{\Big{(}}
\newcommand{\Bc}{\Big{)}}

\newcommand{\Lb}{\begin{lemma}}
\newcommand{\Le}{\end{lemma}}
\newcommand{\Pb}{\begin{proposition}}
\newcommand{\Pe}{\end{proposition}}
\newcommand{\Tb}{\begin{theorem}}
\newcommand{\Te}{\end{theorem}}
\newcommand{\Cb}{\begin{corollary}}
\newcommand{\Ce}{\end{corollary}}
\newcommand{\Dy}{D_P(R,t,\uR)}

\newcommand{\Smith}{\mathrm{Smith}}

\newcommand{\Diag}{\mathrm{Diag}}

\newcommand{\Vol}{\mathrm{Vol}}
\newcommand{\Char}{\mathrm{Char}}
\newcommand{\meas}{\mathrm{meas}}

\newcommand{\PP}{\mathbb{P}}
\newcommand{\QQ}{\mathbb{Q}}
\newcommand{\ZZ}{\mathbb{Z}}
\newcommand{\Sing}{\mathrm{Sing}}
\renewcommand{\AA}{\mathbb{A}}

\newcommand{\scrm}{\mathfrak{m}}
\newcommand{\ualf}{\underline{\alpha}}
\newcommand{\uz}{\underline{z}}
\newcommand{\starsum}{\sideset{}{^*}\sum}
\newcommand{\Null}{\mathrm{Null}}
\newcommand{\FF}{\mathbb{F}}
\newcommand{\vecnull}{\underline{\mathrm{0}}}

\newcommand{\ue}{\underline{e}}

\newcommand{\supp}{\mathrm{Supp}}
\newcommand{\uT}{\underline{T}}
\renewcommand{\y}{\underline{y}}
\newcommand{\e}{\underline{e}}
\renewcommand{\L}{\mathrm{L}}
\renewcommand{\h}{{\underline{h}}}

\newcommand{\cL}{\mathcal{L}}

\newcommand{\ve}{\varepsilon}
\newcommand{\utau}{\underline{\tau}}
\newcommand{\scrH}{\mathcal{H}}
\newcommand{\tH}{\tilde{\mathcal{H}}}
\newcommand{\scrL}{\mathcal{L}}
\newcommand{\Fnull}{F^{(0)}}
\newcommand{\Gnull}{G^{(0)}}
\renewcommand{\u}{\underline{u}}
\renewcommand{\v}{\underline{v}}
\newcommand{\cont}{\mathrm{Cont}}
\renewcommand{\b}{\underline{b}}
\newcommand{\lqi}{\lambda_{q,i}}
\renewcommand{\uc}{\underline{c}}
\newcommand{\rank}{\mathrm{Rank}}

\begin{document}

\title{On the Hasse Principle for Complete Intersections}
\author{Matthew J. Northey}
\email{m.j.northey@durham.ac.uk}
\author{Pankaj Vishe}
\email{pankaj.vishe@durham.ac.uk}
\begin{abstract}
We prove the Hasse principle for a smooth projective variety $X\subset \PP^{n-1}_\Q$ defined by a system of two cubic forms $F,G$ as long as $n\geq 39$. The main tool here is the development of a version of Kloosterman refinement for a smooth system of equations defined over $\Q$.
\end{abstract}
\maketitle


\pagenumbering{arabic}

\section{Introduction}
\label{C:Statement of Results}
Let $X\subseteq \PP^{n-1}_\Q$ denote a projective complete intersection variety. In particular, let $X$ correspond to the zero locus of a system of $R$ homogeneous polynomials of degree $d$ defined over $\Q$. Let
$$\sigma=\dim\Sing(X),$$
where 
\begin{align}
    \label{eq: Statement of results: sigma definition}
    \Sing(X):=\{\x\in\mb{P}^{n-1}_{\C} \: : \: F_1(\x)=\cdots=F_R(\x)=0,\; \rank(\nabla F_1(\x)\cdots \nabla F_R(\x))<R\}
\end{align}
denotes the singular locus of the variety $X$.
Furthermore, we define $\x$ to be a non-singular point of $X$ if 
\begin{align}
    \label{eq: Adelic def}
    F_1(\x)=\cdots=F_R(\x)=0,\quad \rank(\nabla F_1(\x)\cdots \nabla F_R(\x))=R.
\end{align}
A long-standing result of Birch \cite{Birch61} establishes the Hasse principle as long as 
$$n-\sigma\geq (d-1)2^{d-1}R(R+1)+R.$$
While the case of lower degree hypersurfaces $(R=1)$ has seen several breakthroughs in recent times, the case of general complete intersections has seen relatively lower success. In the case of pair of quadrics over $\Q$, Munshi introduced a version of the delta method which allows one to use Kloosterman refinement \cite{Munshi15}. He combined this with Poisson summation to verify the Hasse principle when $n\geq 11$, provided that their intersection is non-singular. Unfortunately, the techniques used fail to generalise effectively outside of the case of two quadrics.

There have been two recent notable breakthroughs. Myerson \cite{Myer},\cite{Myer1} improved the square dependence on $R$ in Birch's result to a linear one. When $d=2$ and $3$, these results improve the lower bound to $n-\sigma\geq 8R$ and $25R$ respectively. This is a significant improvement when $R$ is large. However when $R$ is small (say $2$), it fails to improve upon Birch's bounds. Typically one expects better understanding of the distribution of rational points when $d$ and $R$ are \textit{relatively small}. When $R=1$, this is facilitated by an analytic technique called \textit{Kloosterman refinement}, which allows one to use the Poisson summation formula in an effective way. A recent breakthrough was obtained in the second author's work \cite{Vishe19}, where a version of Farey dissection was developed for a system of two forms in the function field setting. Unfortunately, so far the method there doesn't extend to the $\Q$ setting.

The main purpose of this work is to provide a route to Kloosterman refinement for a system of forms over $\Q$ in the settings where Poisson summation does not work directly. In particular, the method here should improve upon the current results as long as the defining forms $F$ and $G$ of $X$ are not two quadrics or a cubic and a quadric. 

We now define the setting in this paper. Let $F(\x),G(\x)\in\Z[x_1,...,x_n]$ be two homogeneous cubic forms in $n$ variables and with integer coefficients, and let $X$ denote the smooth projective variety defined by their simultaneous zero locus. The long standing result by Birch $n\geq 49$ is yet to be improved in the current setting (a pair of cubics). In the case of a system of diagonal cubic forms, one can obtain significantly stronger results. In particular, Br\"udern and Wooley \cite{Bruedern-Wooley07}, \cite{Bruedern-Wooley16} proved that the Hasse principle is true for a smooth system of $R$ diagonal cubic forms in $n$ variables provided that $n\geq 6R+1$.

In this paper, we will use a combination of Kloosterman refinement and a two dimensional version of averaged van der Corput differencing to improve upon Birch's result. In particular, we aim to prove the following result:
\begin{theorem}
\label{thm:main thm}
Let $X:=X_{F,G}\subset \PP^{n-1}_\QQ$ be a smooth complete intersection variety defined by a system of two cubic forms $F$ and $G$.
Then $X$ satisfies the Hasse principle provided that $n\geq 39$.
\end{theorem}
To the best of the authors' knowledge, this is the first known improvement of the Birch's result in this case. As is the typical feature of the methods used here, with some more work the result can be easily extended to cover the cases of singular varieties, as long as $n-\sigma\geq 40$. However, here we will stick to the non-singular setting. The limitation of the method here is $n\geq 38$. Akin to the work \cite{Marmon_Vishe} of Marmon and the second author, saving an extra variable variable will require substantially new technical input which we will not attempt to obtain here.

For those familiar with circle method techniques, there are two key innovations here that facilitate Theorem \ref{thm:main thm}. The first improvement comes from developing a two dimensional version of averaged van der Corput differencing. This followed by Weyl differencing itself could hand us Theorem \ref{thm:main thm} when $n\geq 43$. Interestingly, averaged van der Corput followed by pointwise Poisson summation fails to improve upon this. Our second innovation comes from combining averaged van der Corput process with a version of Kloosterman refinement followed by Poisson summation. This combination saves us $4$ extra variables. If one were to combine the averaged van der Corput process here with the method of Munshi \cite{Munshi15}, our rough calculations show that one would require $n\geq 42$ variables. This is to be expected as the method in \cite{Munshi15} would result in the use of a larger \textit{total modulus} (the parameter $Q$ appearing in this paper). This is wasteful if one is dealing with forms in many variables, rendering the method less effective when dealing with complete intersections which are not defined by two quadrics. 

We now give a more detailed outline of the key ideas. From now on, we will assume that $X$ is a complete intersection of two cubics as before further containing a non-singular adelic point, i.e. that
\begin{equation}\label{eq:localobs}
X_{\textrm{ns}}(\AA_\QQ) \neq \emptyset,
\end{equation}
where given any variety $X$, let
\[X(\AA_\QQ):= X(\R) \times \prod_p X(\Q_p).\]
Given a smooth weight function $\omega\in \mathrm{C}^\infty_c(\R^n)$, and a large parameter $1\leq P$, we define the following smooth counting function:
\[N(P):=N_\omega(P):=\sum_{\substack{\underline{x}\in\mathbb{Z}^n,\\ F(\underline{x})= G(\underline{x}) = 0}} \omega(\underline{x}/P).\]

Our main tool in proving Theorem \ref{thm:main thm} is the asymptotic formula for $N(P)$ obtained in Theorem \ref{thm:main thm 1}. Before stating it, let us define the weight function $\omega$ in the following way. We will choose $\omega$ to be a smooth weight function, centred at a non-singular point $\underline{x}_0\in X(\mathbb{R})$ with the additional property that its support is a ``small" region around $\underline{x}_0$. Upon recalling \eqref{eq: Adelic def}, it is easy to see that the existence of such a point is guaranteed by our earlier assumption that $X$ has a non-singular adelic point. In particular, the point $\x_0\in X(\R)$ must have $$\rank(\nabla F(\x_0), \nabla G(\x_0))=2.$$ 

Using homogeneity of $F$ and $G$, we may further assume that $|\x_0|<1$. This condition is superficial, and only assumed to make the implied constants appearing in our argument simpler. Let
\[\gamma(\underline{x}):=\begin{cases} \prod_j e^{-1/(1-x_j)^2}\;\, \mbox{ if } |\underline{x}|<1,\\ 0 \quad\quad\quad\quad\quad\quad\; \mbox{ else,} \end{cases}\]
denote a non-negative smooth function supported in the hypercube $[-1,1]^n$. Given a parameter $0<\rho<1$ to be suitably decided later, we define
\begin{equation}\label{eq:omegadef}
\omega(\underline{x}):=\gamma(\rho^{-1}(\underline{x}-\underline{x}_0)).
\end{equation}

We are now set to state our main counting result, which directly implies Theorem \ref{thm:main thm}.
\begin{theorem}
\label{thm:main thm 1}
Let $X\subset \PP^{n-1}_\Q$ be a smooth complete intersection variety defined by a system of two cubic forms $F,G$. Then provided that $n\geq 39$ and $X_{\mathrm{ns}}(\AA_\QQ) \neq \emptyset$, there exist $C_X>0$ and some $\rho_0\in (0,1]$, such that for each $0<\rho\leq \rho_0$, there exists $\delta_0:=\delta_0(\rho)>0$ such that
\[N(P)=C_XP^{n-6}+O_{n,F,G, \rho}(P^{n-6-\delta_0}).\]
\end{theorem}
Our main tool here will be provided by the Circle method. It begins with by writing the counting function  $N(P)$ as an integral of a suitable exponential sum:
\begin{align*}
N(P):=N_\omega(P)&:=\sum_{\substack{\underline{x}\in\mathbb{Z}^n,\\ F(\underline{x})= G(\underline{x}) = 0}} \omega(\underline{x}/P)=\int_0^1 \int_0^1 S(\alpha_1,\alpha_2) d\alpha_1d\alpha_2,
\end{align*}
where 
\begin{equation}
\label{eq:Salfdef}
S(\ualf):=S(\alpha_1,\alpha_2):= \sum_{\underline{x}\in\mathbb{Z}^n} \omega(\underline{x}/P) e(\alpha_1 F(\underline{x})+\alpha_2 G(\underline{x})),
\end{equation}
denotes the corresponding exponential sum.  

In the traditional circle method, the unit square $I:=[0,1]^2$ is split into major arcs $\mathfrak{M}$ which consist of the points in $I$ which are ``close" to a rational point $\ua/q$, where $\ua=(a_1,a_2)\in\ZZ^2$ of ``small" denominator $q$, and minor arcs $\mathfrak{m}=I\backslash\mathfrak{M}$. The limitation of the process usually occurs while bounding the integral $$\int_{\scrm}S(\ualf)d\ualf .$$
When $R=1$, Kloosterman's revolutionary idea \cite{Kloosterman} was to apply Farey dissection to partition $[0,1]$ and use it to bound the minor arc contribution. This allows us to treat the minor arcs in a similar way to the major arcs. This idea essentially allows us -- upon setting $\alpha:=a/q+z$ and fixing the value of $z$ -- to consider averages of the corresponding exponential sum of the form
$$\sum\limits_{\substack{a\bmod{q}\\ (a,q)=1}}S(a/q+z). $$
The extra average over $a$ allows us to save an extra factor of size $O(q^{1/2})$, when $q$ is sufficiently large and $z$ relatively small.

When $R=2$, finding an analogue of Farey dissection which can be used to attain Kloosterman refinement over $\Q$ has proved to be major problem. In \cite{Vishe19}, the second author has managed find such an analogue in the function field setting, but so far it is not known how to use these ideas when working over $\Q$. The path to Kloosterman refinement in this paper will not focus on innovations to Farey dissection, and will instead focus on improving van der Corput differencing.

In the setting of that we will discuss (pair of two cubics), the Poisson summation formula cannot be applied directly. To be more precise, it is possible to apply Poisson summation, but the bound that it gives is trivial due to the corresponding exponential integral bound behaving badly when the degrees of our forms become too large. 

We therefore must use a differencing argument (such as van der Corput) to bound $|S(\alpha)|$ by a sum with polynomials of lower degree. To do this, one essentially starts by using Cauchy's inequality to bound 
\begin{equation}\label{eq:first estimation}
\left|\int_{\scrm} S(\ualf)d\ualf\right|\ll \left(\int_{\scrm}|S(\ualf)|^2d\ualf\right)^{1/2}.
\end{equation}
This leads us for a fixed integer $q$ and a fixed small $\uz\in I$ to consider the averages of the form 
\begin{equation}\label{eq:Sqdef}
\int_{|\uz|<q^{-1}Q^{-1/2}}\sum_{\substack{\ua\bmod{q}\\ 
(\ua,q)=1} }|S(\ua/q+\uz)|^2d\uz,
\end{equation}
where $Q$ is a suitable parameter to be fixed later. This parameter $Q$ arises from using a two dimensional version of Dirichlet approximation theorem. We further develop a two dimensional version of averaged van der Corput differencing used by in \cite{Hanselmann} , \cite{Heath-Brown07}, and \cite{Marmon_Vishe} to estimate the averages of $|S(\ua/q+\uz)|^2$ over $\uz$. This leads us to considering quadratic exponential sums for a system of differenced quadratic forms 
\begin{equation}\label{eq:fhghdef}
F_{\h}(\x):=\h\cdot \nabla F(\x),\,\,G_{\h}(\x):=\h\cdot \nabla G(\x).
\end{equation}
The extra averaging over $\ua$ in \eqref{eq:Sqdef} leads us to a saving of the size $O(q)$ in the estimation of $\sum_{\ua}|S(\ua/q+\uz)|^2$, and in the light of squaring technique used in \eqref{eq:first estimation}, it overall saves us a factor of size $O(q^{1/2})$ when $q$ is square-free.

The methods developed here are versatile and can be readily adapted to deal with general complete intersections. While dealing with averages of squares of corresponding exponential sums next to rationals of type $(a_1,...,a_R)/q$, where $q$ is square-free, we would be able to save a factor of size $O(q^{R/4})$ over the bounds coming from averaged van der Corput along with pointwise Poisson summation. To the best of the authors' knowledge, this is the first known version of Kloosterman refinement which generalises this way over $\Q$. This method could be further combined with any further versions of Kloosterman refinement in the contexts where a degree-lowering squaring technique is essential. For instance, in the function field setting, this method could potentially combined with the method in the aforementioned work by the second author \cite{Vishe19} to be able to save a factor of size $O(q^{(R-1)/4+1/2})$ instead.

\subsection{Acknowledgements} We would like to thank Tim Browning and Oscar Marmon for their help.

\section{Background on a pair of quadrics}
\label{C:Background on a pair of quadrics}
Exponential sums for a pair of quadrics will feature prominently in this work. Let $Q_1(\underline{x}), Q_2(\underline{x})$ be a pair of quadratic forms in $n$ variables with integer coefficients and consider the variety defined by
 \[V: Q_1(\underline{x})=Q_2(\underline{x})=0,\]
$\underline{x}\in\overline{\mathbb{Q}}^n$. Let $\Sing_K(V)$ to be the (projective) singular locus of $V$. When $Q_1$ and $Q_2$ intersect properly, namely, if $V$ is of projective dimension $n-3$ then we can express the singular locus of $V$ as follows:
\begin{equation}
\label{eq:singlocus}
\Sing_K(V):=\Big{\{}\underline{x}\in \mb{P}^{n-1}_{\overline{K}} \;\Big{|}\; \underline{x}\in V, \; \rank
\begin{pmatrix}
\nabla Q_1(\underline{x})\\ \nabla Q_2(\underline{x})
\end{pmatrix} <2 \Big{\}}.
\end{equation}

We say that the intersection variety of $Q_1(\underline{x})$ and $Q_2(\underline{x})$, $V$,  is non-singular if \\ $\dim\Sing_K (V)=-1$, and singular otherwise. It should be noted that \eqref{eq:singlocus} only truly encapsulates the set of singular points when $Q_1,Q_2$ have a \ti{proper} intersection over $K$ (that is, the forms $Q_1(\x)$, $Q_2(\x)$ share no common factor over K). However, $\Sing_K(V)$ is still a well defined set with a well defined dimension, even when $Q_1$ and $Q_2$ intersect improperly, and so we will also use this definition in this case.
We will now begin by noting a slight generalisation of \cite[Lemma 4.1]{Marmon_Vishe} in the context of two quadrics, which will be vital in various stages of this paper:
\begin{lemma}
\label{lem:hyperplane}
Let $Q_1,Q_2$ be a pair of quadratic forms defining a complete intersection $X=V(Q_1,Q_2)$. Let $\Pi$ be a collection of primes such that $\#\Pi=r\geq 0$ and define $\Pi_a:=\{p\in\Pi\,|\, p>a\}$ for every $a\in\mathbb{N}$. Then there exists a constant $c'=c'(n)$ and a set of primitive linearly independent vectors
\[\underline{e}_1,\cdots,\underline{e}_{n}\in\mathbb{Z}^n\]
satisfying the following property for any integer $0\leq \eta \leq n-1$, any subset $\phi\neq I\subset\{1,2\}$ and any $\upsilon\in\{\infty\}\cup\Pi_{2c'}$: The subspace $\Lambda_\eta\subset\mathbb{P}^{n-1}_{\mathbb{F}_\upsilon}$ spanned by the images of $\underline{e}_1,\cdots,\underline{e}_{n-\eta}$ is such that
\begin{equation}\label{eq:cond1}
\dim(X_I\cap\Lambda_\eta)_\upsilon=\max\{-1,\dim(X_I)_\upsilon-\eta\}
\end{equation}
and
\begin{equation}
\label{eq:cond2}
\dim \Sing((X_I\cap \Lambda_\eta)_\upsilon)=\max\{-1,\dim \Sing((X_I)_\upsilon)-\eta\}.\end{equation}
Here given $\emptyset\neq I\subset \{1,2\}$, let $X_I$ denote the complete intersection variety defined by the forms $\{F_i:i\in I\}$.
Moreover, the basis vectors $\underline{e}_i$ can be chosen so that
\begin{equation}
    \label{eq:cond4}
    L/2\leq |\underline{e}_i|\leq L
\end{equation}
for every $i=1,\cdots, n$ and 
\begin{equation}
    \label{eq:cond5}
    L^n\ll det(\underline{e}_1,\cdots,\underline{e}_n)\ll L^n
\end{equation}
for some constant $L=O_n(r+1).$
\end{lemma}
\begin{proof}

Note that the statement of this lemma is identical to that of \cite[Lemma 4.1]{Marmon_Vishe} except that in the latter there is an additional assumption that the closed subscheme $X_I\subset \mb{P}^{n-1}_{\Z}$ defined by $F_i=0$ for all $i\in I$ satisfies
\begin{equation}\label{eq:cond3}
\dim(X_I)_\upsilon=n-1-|I|.
\end{equation}
This is equivalent to the case when $X_1$ and $X_2$ intersect properly. Therefore, it is enough to consider different cases where we have an improper intersection. In each of these particular cases, somewhat softer argument works.

In the trivial case when $Q_1=Q_2=0$, any basis $\e_1,...,\e_n$ will work.

When $Q_2=\lambda Q_1$, where $\lambda\in K$ and $Q_1$ a non zero quadratic form then we may apply \cite[Lemma 4.1]{Marmon_Vishe} only to the hypersurface $X_1$ to find a basis $\e_1,...,\e_n$ which is chosen such that \eqref{eq:cond1} and \eqref{eq:cond2} hold for $I=\{1\}$. This choice will clearly work for all $I\subset\{1,2\}$.

In the remaining case when $Q_1=L_1L_2,Q_2=L_1L_3$, where $L_i=\v_i\cdot\x$ and $L_2$ is not a scalar multiple of $L_3$. In this case, it is easy to check that the singular locus of $X_1\cap X_2$ to is the hyperplane $L_1=0$. Here, we may apply \cite[Lemma 4.1]{Marmon_Vishe} to the single variety defined by the cubic form $L_1L_2L_3=0$. The basis $\Lambda$ that we get from this process will work here as well.
\end{proof}

We are now ready to prove the following generalisation of \cite[Proposition 2.1]{HeathBrown_Pierce17}. This will be particularly helpful for us when we are working with exponential sums of the form 
\[\starsum_{\ua}^q \sum_{\x}^q e_{q}(a_1Q_1(\x)+a_2Q_2(\x)+\uc\cdot\x),\]
where $q$ is square-full, in Section \ref{S621}. Here, as is standard the $*$ next to the sum denotes that the sum is over $(\ua,q)=1$ and the notation $e_q(x):=\exp (2\pi i x/q)$.
\begin{proposition}
\label{E22}
Let $\nu$ either denote a finite prime $\nu\gg_n 1$ or the infinite prime, let $\FF_\nu$ either denote the corresponding finite field or $\Q$, and let 
\begin{equation}\label{eq:mdef}
s_{\nu}(Q_1,Q_2):=\dim\Sing_{\FF_\nu}(V),
\end{equation} 
where $V$ is defined as above.  Moreover for every $(a_1,a_2)\in \FF_\nu^2\backslash (0,0)$, the rank of the matrix associated to the quadratic form $a_1Q_1+a_2Q_2$, $a_1M_1+a_2M_2$, satisfies
\begin{equation}\label{eq:geqnm2}
\rank(a_1M_1+a_2M_2)\geq n-s_{\nu}(Q_1,Q_2)-2.
\end{equation}
Moreover, there exists a set $\Gamma=\{\gamma_1,....,\gamma_k\}\subset\overline{\FF}_\nu$, such that as long as $a_1\neq \lambda_i a_2$ for some $1\leq i\leq k\leq n$, and $a_2\neq 0$, then 
\[\rank(a_1M_1+a_2M_2)\geq n-s_{\nu}(Q_1,Q_2)-1.\]

\end{proposition}

\begin{proof}
Let $M_1$ and $M_2$ denote the integer matrices defining the forms $Q_1$ and $Q_2$ respectively. We firstly note that for $s_{\nu}(Q_1,Q_2)=-1$, we recover \eqref{eq:geqnm2} from \cite[Proposition 2.1]{HeathBrown_Pierce17}. In this case may also use \cite[Proposition 2.1]{HeathBrown_Pierce17} to simultaneously diagonalise $M_1$, $M_2$, letting us instead work with 
\[Q_i'(\x):=\sum_{j=1}^n\lambda_{i,j}x_j^2, \quad M_i':= \Diag(\un{\lambda_{i}}).\]
In particular, we have
\[s_{\nu}(Q_1',Q_2')=s_{\nu}(Q_1,Q_2)=-1, \quad \rank(a_1M_1'+a_2M_2')=\rank(a_1M_1+a_2M_2),\]
for every $\ua\in\FF^2_{\nu}\backslash(0,0)$. Next, we note that $\rank(a_1M_1'+a_2M_2')<n$ if and only if there is some $j\in\{1,\cdots, n\}$ such that $a_1\lambda_{1,j}+a_2\lambda_{2,j}=0$, which imposes the desired restriction on $(a_1,a_2)$ provided that $(\lambda_{1,j},\lambda_{2,j})\neq (0,0)$. However, if $(\lambda_{1,j},\lambda_{2,j})=(0,0)$, then it is easy to see from the definition of $Q_i'(\x)$ that 
\[\nabla Q_1'(m\e_j)=\nabla Q_2'(m\e_j)=\un{0}\]
for every $m\in \overline{\FF}_{\nu}$ (provided $\nu> 2$), where $\e_j$ is the j-th vector in the standard basis. This implies that $m\e_j\in \Sing(Q_1',Q_2')$, and so $s_{\nu}(Q_1',Q_2')\geq 0$, giving us a contradiction.

If $s_{\nu}(Q_1,Q_2)\neq -1$, we invoke Lemma \ref{lem:hyperplane}. As long as $\nu\gg_n 1$, we obtain a basis $\ue_1,...,\ue_n$ of $\FF_\nu^n$ such that the system of quadrics $\tilde{Q}_1,\tilde{Q}_2$ corresponding to the restriction of $Q_1$ and $Q_2$ onto the subspace $\Lambda_{n-s_{\nu}-1}$ obeys \eqref{eq:cond1} - \eqref{eq:cond2}. This clearly defines a system of non-singular quadratic forms defined over $n-s_{\nu}-1$, whose complete intersection is non-singular over $\overline{\FF}_\nu$ as well. Now let $\tilde{M}_1$ and $\tilde{M}_2$ denote the integer matrices defining the forms $\tilde{Q}_1$, and $\tilde{Q}_2$ respectively. The Lemma now follows from noticing that $$\rank(a_1M_1+a_2M_2)\geq \rank(a_1\tilde{M}_1+a_2\tilde{M}_2),$$
for any pair $(a_1,a_2)\in\FF_\nu^2\setminus (0,0) $ and further using our analysis of the non-singular case above.
\end{proof}
%

One of the key bounds for exponential sums in this work will be provided by Weyl differencing. Typically, these bounds use a `Birch-type' singular locus $\sigma_K'$ as defined in \eqref{eq: pair of quadrics: sigma_K'(F,G): defn (F,G deg d)} instead of the singular locus \eqref{eq:singlocus} used here. A relation between the two has been studied in \cite{Browning-Heath-Brown14}. A minor modification of \cite[Lemma 1.1]{Myer} readily provides us with the following result:
\begin{lemma}
\label{P: background on a pair of quadrics: sigma'<= sigma+1}
Let $F,G$ be non-constant forms of any degree, $K$ be a field, and let
\begin{align}
\label{eq: pair of quadrics: sigma_K(F): defn (F,G deg d)}
\sigma_K(F)&:=\dim \{ \x\in \mb{P}^{n-1}_{\overline{K}} \: : \: F(\x)=0,\, \nabla F(\x)=\un{0}\},\\
\label{eq: pair of quadrics: sigma_K'(F,G): defn (F,G deg d)}
\sigma_K'(F,G)&:=\dim \{ \x\in \mb{P}^{n-1}_{\overline{K}} \: : \: \rank
\begin{pmatrix}
\nabla F(\underline{x})\\ \nabla G(\underline{x})
\end{pmatrix} <2 \}, \\
\label{eq: pair of quadrics: sigma_K(F,G): defn (F,G deg d)}
\sigma_K(F,G)&:=\dim \Sing_K(F,G).
\end{align}
Then, we have
\[\sigma_K(a_1F+a_2G)\leq \sigma_K'(F,G)\leq \sigma_K(F,G)+1,\]
for any $(a_1,a_2)\in K\backslash \{(0,0)\}$.
\end{lemma}

Our main exponential sum bound for square-full moduli $q$ will be in terms of the size of the null set
\begin{equation}\label{eq:NullMdef}
\Null_{q}(M):=\{\x\in (\Z/q\Z)^n:M\x=\vecnull\},
\end{equation}
for some matrix $M$. The following three Lemmas will be related to this set.
\begin{lemma}
\label{T1}
For every $u,v\in \N$, and every $M\in M_n(\Z)$, we have
\[\#\Null_{uv}(M)\leq \#\Null_u(M)\#\Null_v(M),\]
with equality if $(u,v)=1$.
\end{lemma}

\begin{proof}
It is easy to prove that $\#\Null_{q}(M)$ is a multiplicative function, so we will not prove that
\begin{equation}
    \label{eq: background on a pair of quadrics: Null multiplicative: gcd=1}
    \#\Null_{uv}(M)= \#\Null_u(M)\#\Null_v(M),
\end{equation} when $(u,v)=1$. We will be brief when showing the inequality, as this is a standard Hensel Lemma type of argument. If $\x\in \Null_{uv}(M)$, then we must have $\x\in\Null_u(M)$. Hence, if we write $\x:=\y+u\z$, then $\y$ must be in $\Null_u(M)$.

Now, fix $\y$ and assume that there is some $\z_1,\z_2$ (not necessarily distinct) such that $\y+u\z_i\in \Null_{uv}(M)$. Then
\[M(\y+u\z_i)\equiv \un{0} \mod uv,\]
and so 
\[M(\y+u\z_2)-M(\y+u\z_1)=uM(\z_2-\z_1)\equiv \un{0} \mod uv.\]
Therefore, upon letting $\z_2:=\z_1+\z'$ we must have
\[M\z'\equiv \un{0} \mod v.\]
Hence, there can only be at most $\#\Null_v(M)$ possible values for $\z'$ and so there can only be at most $\#\Null_v(M)$ values for $\z$ such that $\y+u\z\in \Null_{uv}(M)$ for any given $\y$. We also have that $\y$ must be in $\Null_{u}(M)$. This gives us
\[\#\Null_{uv}(M)\leq \#\Null_{u}(M)\#\Null_{v}(M),\]
as required.
\end{proof}
In both Section \ref{Sec: Exponential Sums: Initial Bounds} and \ref{sec:quadfinal}, we will need to bound $\#\Null_{p}(M)$ for matrices of the form $M(\ua):=a_1M_1+a_2M_2$, where $M_1$ and $M_2$ are symmetric matrices associated to some quadratic forms $Q_1(\x)$, $Q_2(\x)$. In Proposition \ref{E22}, we noted that for most values of $\ua$, $\rank_{p}(M(\ua))\geq n-s_{p}-1$, but there were potentially a few lines of $\ua$'s where $\rank_{p}(M(\ua))= n-s_{p}-2$. Naturally, a lower bound on the size of the rank of a matrix leads to an upper bound on the dimension of the nullspace of a matrix (due to the rank-nullity theorem), and so using $\rank_{p}(M(\ua))\geq n-s_{p}-2$ in order to bound $\#\Null_{p}(M(\ua))$ for every $\ua$ would be wasteful. This will lead us to considering averages of $\#\Null_{p}(M(\ua))$, where $\ua$ is allowed to vary. This is the topic of the next lemma.
\begin{lemma}
\label{T20}
Let $Q_1,Q_2$ be quadratic forms in $n$ variables, $q\in\N$, and 
\[d:=\prod_{i=1}^r p_i\]
be squarefree (in other words, the $p_i$'s are prime) such that $d\,|\,q$. Furthermore, let $M_1, M_2$ be integer matrices defining $Q_1$ and $Q_2$ respectively, and let $s_p=s_p(Q_1,Q_2)$ be as defined in \eqref{eq:mdef} for $K=\FF_p$, $p$ a prime. Then
\[S(d,q):=\starsum_{\ua  \bmod{q}} \#\Null_d(a_1M_1+a_2M_2)\ll_n q^2 \prod_{i=1}^r p_i^{s_{p_i}+1}.\]
\end{lemma}

\begin{proof}
We firstly note that upon setting $\ua=\b+d\uc$,
\begin{align}
    \label{eq: pair of quadrics: S(d,q) bound: 1}
    S(d,q)&=\starsum_{\ua \bmod{q}} \#\Null_d(a_1M_1+a_2M_2)\leq \sum_{\substack{\ua \bmod{q}\\ (a_1,a_2,d)=1}} \#\Null_d(a_1M_1+a_2M_2)\nonumber\\
        &= \sum_{\substack{\b \bmod{d}\\ (b_1,b_2,d)=1}} \#\Null_d(b_1M_1+b_2M_2) \sum_{\uc \bmod{q/d}} 1= \Bo\frac{q}{d}\Bc^2\starsum_{\b \bmod{d}} \#\Null_d(b_1M_1+b_2M_2)\nonumber\\
        &= \Bo\frac{q}{d}\Bc^2 S(d,d).
\end{align}
For convenience, define 
\begin{align}
    \label{eq: pair of quadrics: T(d) defn}
    T(d):=S(d,d).
\end{align}
Using the Chinese remainder theorem, it is easy to see that $T(d)$ is a multiplicative function. In particular, we have 
\begin{align}
    \label{eq: pair of quadrics: T(d) prime decomp}
    T(d)=\prod_{\substack{i=1\\ p_i\mid d \textrm{ where } p_i \textrm{ prime }}}^r T(p_i).
\end{align}It is therefore sufficient to consider
\begin{align}\label{eq:spk}
T(p)=\starsum_{\ua \bmod{p}} \#\{\x\bmod{p}:(a_1M_1+a_2M_2)\x\equiv \vecnull\bmod{p}\},
\end{align}
where p is a prime. When $p\ll_n 1$, the right hand side is trivially $ O(p^{2})$. It is therefore enough to consider the case $p\gg_n 1$, where the implied constant is chosen as in the statement in Proposition \ref{E22}.
Proposition \ref{E22} now implies that except for $O_n(p)$ different exceptional pairs $(a_1,a_2)$, $\rank(a_1M_1+a_2M_2)\geq n-s_p-1$. Moreover, for the exceptional pairs we still have $\rank(a_1M_1+a_2M_2)= n-s_p-2$. Finally, we note that if $M$ is an integer matrix rank $k$ over $\FF_p$, it is easy to see that
$$\#\{\x\in\FF_p^n:M\x=\vecnull\}\ll p^{n-k}.$$
Applying these results to \eqref{eq:spk} gives us
\begin{align*}
T(p)&\ll \starsum_{\substack{\ua \bmod{p}\\ \rank(a_1M_1+a_2M_2)\geq n-s_p-1}} p^{s_p+1} +  \starsum_{\substack{\ua \bmod{p}\\ \rank(a_1M_1+a_2M_2)= n-s_p-2}} p^{s_p+2}\\
    &\ll p^2\times p^{s_p+1}+p\times p^{s_p+2}\\
    &\ll p^{2+s_p+1},
\end{align*}
and so 
\begin{align*}
    T(d)\ll\prod_{i=1}^r p_i^{2+s_{p_i}+1} = d^2 \prod_{i=1}^r p_i^{s_{p_i}+1}
\end{align*}
by \eqref{eq: pair of quadrics: T(d) prime decomp}. Hence, by \eqref{eq: pair of quadrics: S(d,q) bound: 1} - \eqref{eq: pair of quadrics: T(d) defn}, we have
\begin{align*}
S(d,q)&\leq \Bo\frac{q}{d}\Bc^2 T(d)\ll q^2 \prod_{i=1}^r p_i^{s_{p_i}+1},
\end{align*}
as required.
\end{proof}
During the process of bounding quadratic exponential sums, we will need to bound the size of the set 
\begin{align}
    \label{eq: pair of quadrics: N_b,q defn}
    N_{\un{b},q}(M):=\{\x\in (\Z/q\Z)^n \: : \: M\x \equiv \frac{q}{2}\,\un{b} \pmod{q}\}.
\end{align}
The next lemma will help us to do this by letting us relate $N_{\un{b},q}(M)$ to $\Null_q(M)$.
\Lb
\label{lem:N_b=y_b+ Null}
Let $q\in \N$ be even, $M\in M_n(\Z/q\Z)$, and let $N_{\un{b},q}(M)$ be defined as in \eqref{eq: pair of quadrics: N_b,q defn}. Then for every $\un{b}\in \{0,1\}^n$, either $N_{\un{b},q}(M)=\emptyset$ or there exists some $\y_{\un{b}}\in (\Z/q\Z)^n$ such that
\[N_{\un{b},q}(M)= \y_{\un{b}} + \Null_q(M).\]
\Le

\begin{proof}
If we assume that $N_{\un{b},q}(M)\neq \emptyset$, then there must be some $\y\in N_{\b,q}(M)$. By the definition of $\Null_q(M)$, if $\y_0\in \Null_q(M)$, then $\y+\y_0\in N_{\b,q}(M)$. Hence 
\begin{equation}
    \label{eq:Null subset N_(b,q)}
    \y+\Null_q(M)\subset N_{\b,q}(M),
\end{equation}
and so $\#N_{\b,q}(M)\geq \#\Null_q(M)$. 

Likewise, we note that if $\y_1,\y_2\in N_{\b,q}(M)$, then $\y_1-\y_2\in \Null_q(M)$, and so $\#N_{\b,q}(M)\leq \#\Null_q(M)$. Therefore 
\[\#N_{\b,q}(M)=\#\Null_q(M)=\#(\y+\Null_q(M)).\]
Combining this with \eqref{eq:Null subset N_(b,q)} gives us the result we desire.
\end{proof}

\section{Initial setup}
\label{C: Initial setup}
In this section we will start with some initial considerations which will help us to properly set up the circle method and state our main results which will be used to prove Theorem \ref{thm:main thm 1}. As stated before, the Hardy Littlewood circle method transforms the task of answering Theorem \ref{thm:main thm 1} to proving  an asymptotic formula:
\begin{equation}
    \label{E01}
    \int_0^1 \int_0^1 S(\alpha_1,\alpha_2) d\alpha_1d\alpha_2= C_{X}P^{n-6}+o(P^{n-6}).
\end{equation}
Here $S(\ualf)$ is the exponential sum as defined in \eqref{eq:Salfdef}, and $C_X$ denotes a product of local densities.  We will start by splitting the box $[0,1]^2$ into a set of major arcs and minor arcs as follows:

For any pair $(\alpha_1,\alpha_2)$, we can use a two dimensional version of Dirichlet's approximation theorem to find a simultaneous approximation $(a_1/q,a_2/q)$. In particular upon taking $Q=\floor*{P^{3/2}}$, there exists $\underline{a}=(a_1,a_2)\in\mathbb{Z}^2$ and $q\in\mathbb{N}$ s.t. $(a_1,a_2,q)=1$, $q\leq Q$, and 
\begin{equation}
\label{E11}
    \Big{|}\alpha_1-\frac{a_1}{q}\Big{|}\leq \frac{1}{qQ^{1/2}}, \quad\quad \Big{|}\alpha_2-\frac{a_2}{q}\Big{|}\leq \frac{1}{qQ^{1/2}}.
\end{equation}
We can therefore write
\begin{equation}
\label{eq:uzdef}
\alpha_1=\frac{a_1}{q}+z_1, \quad\quad \alpha_2=\frac{a_2}{q}+z_2,
\end{equation}
for some $|\z|:=\max\{|z_1|,|z_2|\}\leq 1/qQ^{1/2}$. The choice $Q=\floor*{P^{3/2}}$ arises from our final optimisation of various bounds. We explain this in detail in Section \ref{S: Explaining Q}.

Now let $0<\Delta<1$ be some small parameter also to be chosen later, and define
\[\mathfrak{M}_{q,\underline{a}}(\Delta):=\Big{\{}(\alpha_1,\alpha_2) \mod 1 \;:\; \Big{|}\alpha_i-\frac{a_i}{q}\Big{|}\leq P^{-3+\Delta}, i=1,2\Big{\}}.\]
We then define the set of major arcs to be
\begin{align}
    \label{Initial set up: major arcs definition}
    \mathfrak{M}=\mathfrak{M}(\Delta):=\bigcup_{q\leq P^\Delta}\bigcup_{\substack{\underline{a}\bmod{q}\\ (\underline{a},q)=1}} \mathfrak{M}_{q,\underline{a}}(\Delta).
\end{align}
This union of sets is disjoint if $P^{-2\Delta}\geq 2 P^{-3+\Delta}$, namely when $\Delta<1$ and when $P$ is sufficiently large. Moreover, it is easy to check that $P^{-3+\Delta}<1/qQ^{1/2}$ for any $q\leq Q$, provided that $Q<P^{3-\Delta}$. This is certainly true for our final choice $Q=P^{3/2}$ since we assumed $\Delta<1$, and so we have that each set $\mathfrak{M}_{q,\underline{a}}$ is contained in the corresponding range from \eqref{E11}. Therefore, the major arcs give the following contribution to the integral in \eqref{E01}:
\begin{equation}
\label{E02}
    S_{\mathfrak{M}}:=\sum_{1\leq q\leq P^{\Delta}} \,\,\starsum_{\underline{a}\bmod{q}}\int_{|\underline{z}|\leq P^{-3+\Delta}} S_{\underline{a}}(q,\underline{z})d\underline{z},
\end{equation}
where 
\begin{equation}
\label{eq:saqdef}
S_{\underline{a}}(q,\underline{z}):=S(\ua/q+\uz).
\end{equation}
We then define the minor arcs to be $\mathfrak{m}=[0,1]^2\backslash\mathfrak{M}$. By construction of $\mathfrak{M}$, the individual minor arcs must therefore either have 
\begin{equation}
\label{eq:minorar}
P^\Delta<q\leq Q \textrm{ and } |\uz|<(qQ^{1/2})^{-1},\textrm{ or } 1\leq q\leq P^\Delta \textrm{ and } P^{-3+\Delta}<|\uz|<(qQ^{1/2})^{-1}.
\end{equation}
Hence, we can bound the minor arcs contribution, upon further bringing the average over $\ua$ inside the integral in \eqref{E01}, by
\begin{align}
\label{E03}
    S_{\mathfrak{m}}=\sum_{1\leq q\leq P^{\Delta}} \sum_{\underline{a}}^q{}^*\int_{P^{-3+\Delta} \leq |\underline{z}|\leq 1/qQ^{1/2}} S(q,\underline{z}) d\underline{z} + \sum_{P^{\Delta}\leq q\leq Q } \sum_{\underline{a}}^q{}^*\int_{|\underline{z}|\leq 1/qQ^{1/2}} S(q,\underline{z}) d\underline{z}.
\end{align}
Here
\begin{equation}\label{eq:sqdef}
S(q,\z):=\starsum_{\ua\bmod{q}}|S_{\ua}(q,\z)|.
\end{equation}

Our techniques for dealing with the major arcs contribution are standard. Let 
\begin{equation}
\begin{split}
\mathfrak{S}(R)&:=\sum_{q=1}^R q^{-n}\starsum_{\underline{a}\bmod{q}} \sum_{\x  \bmod{q}} e_q(a_1F(\x)+a_2G(\x)),\\ \mathfrak{J}(R)&:=\int_{|\underline{z}|<R}\int_{\mathbb{R}^n}\omega(\underline{x})e(z_1F(\underline{x})+z_2G(\underline{x}))\: d\underline{x}d\underline{z},
\end{split}
\end{equation}
and let
\begin{equation}
\mathfrak{S}:=\lim_{R\rightarrow\infty} \mathfrak{S}(R),\quad \mathfrak{J}=\lim_{R\rightarrow \infty} \mathfrak{J}(R),
\end{equation}denote the singular series and the corresponding singular integral, provided the limits exist. Our main major arcs estimate is the following Lemma:
\begin{lemma}
\label{T3001}
Assume that $n-\sigma(F,G)\geq 34$, where $\sigma(F,G):=\sigma(X_{F,G})$ as defined in \eqref{eq: Statement of results: sigma definition}, and assume that $\mathfrak{S}$ is absolutely convergent, satisfying
\[\mathfrak{S}(R)=\mathfrak{S}+O_{\phi}(R^{-\phi}).\]
Then provided that we have $\Delta\in (0,1/7)$, 
\[S_{\mathfrak{M}}=\mathfrak{S}\mathfrak{J}P^{n-6}+O_{\phi}(P^{n-6-\delta}).\]
\end{lemma}
The proof of this lemma, along with the proof of convergence of the singular series will be established in Section \ref{MA}.

The majority of our effort will be spent in bounding the minor arcs contribution. In order to state the Proposition we aim to prove for the minor arcs, we need to further specify our choice of weight function and the point which it will centred on. Let $\x_0$ be a fixed point satisfying $|\x_0|<1$ and
\begin{equation}
\label{E12}
\rank \begin{pmatrix}
\nabla  F(\underline{x}_0)\\
\nabla  G(\underline{x}_0)
\end{pmatrix}=2.
\end{equation}
Without loss of generality, we may assume that
\begin{equation}\label{eq:cross}
|\nabla  F(\x_0)\cdot \nabla  G(\x_0)|\leq C'\|\nabla  F(\x_0)\|_{\L^2}\|\nabla  G(\x_0)\|_{\L^2},
\end{equation}
for some $0<C'<1$, possibly depending on $\x_0$.
We will also slightly expand our definition of the test function $\omega$ to assume it to be supported in a box $\x_0+(-\rho,\rho)^n$, for a small parameter $\rho>0$ to be chosen in due course. 
 Moreover, we ask that the following bound to be true on its derivatives:
\begin{equation}
\label{eq:omegaderi}
\max\Bigg{\{}\Bigg{|}\frac{\partial^{j_1+\cdots + j_n}}{\partial x_1^{j_1}\cdots \partial x_n^{j_n}} \omega(\underline{x})\Bigg{|} \mid \underline{x}\in\mathbb{R}, j_1+\cdots +j_n=j\Bigg{\}}\ll_{j,n} 1,
\end{equation}
for every $j\geq 0$. A satisfactory bound for the minor arcs will be produced by the following proposition, which we aim to prove:

\begin{proposition}
\label{P02}
Let $F,G$ be a system of two cubic forms with a smooth intersection satisfying $n\geq 39$, and let $\omega\in \C^\infty_c(\x_0+(-\rho,\rho)^n)$ satisfy \eqref{eq:omegaderi}, where $\x_0$ satisfies \eqref{eq:cross}. Then there exists some $\delta=\delta(\Delta)>0$ and some $\rho_0>0$, such that for any $0<\Delta< 1/7$ and for any $0<\rho<\rho_0$, we have
\[S_{\mathfrak{m}}=O_{n,\rho,\Delta, ||F||,||G||}(P^{n-6-\delta}).\] 
Here, given a polynomial $F$, let $\|F\|$ denote the maximum of all its coefficients.
\end{proposition}
A major part of the rest of this work will be dedicated to proving Proposition \ref{P02}, which will ultimately be achieved in Section \ref{sec:minor}.
Before we move on, it will be desirable to obtain a consequence of our choice of $\omega$ and $\x_0$, akin to the conditions \cite[(2.15)-(2.16)]{Marmon_Vishe}. This will be our aim in Lemma \ref{lem:restatement} below, which will be useful in setting up a two dimensional van der Corput differencing argument in Section \ref{vdc} and in particular, in the proof of Lemma \ref{T42}. We choose the vectors $\e_1'$ and $\e_2'$ to be a basis for the span of the two dimensional vector space $\{\nabla F(\x_0),\nabla G(\x_0)\}$, chosen in the following way:
\begin{equation}
\label{eq:e1'e2'}
\underline{e}_1':=\frac{\nabla F(\underline{x}_0)}{||\nabla F(\underline{x}_0)||},\quad \underline{e}_2':=\frac{\nabla G(\underline{x}_0)-\gamma\e_1'}{\gamma_1},
\end{equation}
where $\gamma=\nabla G(\underline{x}_0)\cdot \e_1'$,  and $\gamma_1=\|\nabla G(\underline{x}_0)-\gamma\e_1'\|$
 is a non-zero constant by \eqref{eq:cross}. $\e_2'$ is chosen so that the Gram-Schmidt procedure works.
\begin{lemma}\label{lem:restatement}
Let $F$ and $G$ be cubic forms and $\omega$ be a compactly supported function supported in $\x_0+(-\rho,\rho)^n$ satisfying \eqref{eq:omegaderi}, where $\x_0$ satisfies \eqref{eq:cross}. Then there exist constants $M_1,M_2>0$ and  there exists some $0<\rho_0\leq 1$ such that if $\rho\leq \rho_0$ then
\begin{equation}
\label{E14}
    \min_{\underline{x}\in \supp(P\omega)}|\nabla F(\underline{x})\cdot \e_1'|\geq M_1P^2, \quad \min_{\underline{x}\in \supp(P\omega)}|\nabla G(\underline{x})\cdot \e_2'|\geq M_1P^2,
\end{equation}
\begin{equation}
\label{E15}
\begin{split}
\max_{\x\in\supp (P\omega)}\{|\nabla F(\x)\cdot \e_2'|
\}\leq \rho M_2P^2, \quad \max_{\x\in\supp (P\omega)}\{|\nabla G(\x)\cdot \e_1'|\}\leq M_2P^2.
\end{split}
\end{equation}
Furthermore, $M_1$ and $M_2$ depend only on $F$, $G$, and our choice of $\x_0$ (in particular $M_1$ and $M_2$ do not depend on $\rho$).
\end{lemma}
\begin{proof}
A key in the proof here will be the following bound, which is an easy consequence of the Mean Value Theorem: Given any $\x\in\supp(P\omega)$, we have
\begin{equation}
\label{eq:MVT}
\begin{split}
\|\nabla F(\x)-\nabla F(P\x_0)\|\ll_{||F||} \rho P^2 \textrm{ and } \|\nabla G(\x)-\nabla G(P\x_0)\|\ll_{||G||} \rho P^2.
\end{split}
\end{equation}
Let us first prove that the conditions for $\nabla F(\x)$ in \eqref{E14} - \eqref{E15} are met. The key here are the conditions \eqref{E12} and \eqref{eq:cross}.  Clearly, using \eqref{eq:MVT} we have
\begin{align*}
\nabla F(\underline{x})\cdot \e_1'&=(\nabla F(\x)-\nabla F(P\x_0))\cdot \e_1'+\nabla F(P\x_0)\cdot \e_1'\\
&=(\nabla F(\x)-\nabla F(P\x_0))\cdot \e_1'+P^2\nabla F(\x_0)\cdot \nabla F(\x_0)/\|\nabla F(\x_0)\|\\
&=(\nabla F(\x)-\nabla F(P\x_0))\cdot \e_1'+P^2\|\nabla F(\x_0)\|\\
&\geq (1-O(\rho))P^2\|\nabla F(\x_0)\|\\
&\geq M_{F,1} P^2
\end{align*}
for some $M_{F,1}>0$ which is independent of $\rho$, provided that $\rho$ is chosen to be small enough. Similarly, we may also assure that
\begin{equation}\label{eq:g0bound}
|\nabla G(\underline{x})\cdot \nabla  G(\x_0)|
\geq (1-O(\rho))P^2\|\nabla G(\x_0)\|^2.
\end{equation}
In both of these equations, the implied constants only depend on $\| F\|,|| G||$ and $n$. This will be a feature of all implied constants appearing in this proof. On the other hand, since $\nabla F(\x_0)=\|\nabla F(\x_0)\|\,\e_1'$ is orthogonal to $\un{e}_2'$, we have
\begin{align}
\label{eq:2}
|\nabla F(\x)\cdot \e_2'|=|(\nabla F(\x)-P^2\nabla F(\x_0))\cdot \e_2'|\leq \|(\nabla F(\x)-\nabla F(P\x_0))\|\ll_{\|F\|} \rho P^2
\end{align}
by \eqref{eq:MVT}. In other words, there is some $M_{F,2}>0$ independent of $\rho$ such that 
\[|\nabla F(\x)\cdot \e_2'|\leq M_{F,2} \,\rho P^2.\]
To deal with the inequalities concerning $G$, we use \eqref{eq:cross}, which hands us a constant $0<C'<1$ satisfying
\begin{align}\label{eq:gammabound}
\gamma\|\nabla F(\x_0)\|=|\nabla F(\x_0)\cdot \nabla G(\x_0)|&\leq C' \|\nabla F(\x_0)\|_{L_2}\|\nabla  G(\x_0)\|_{L_2}\nonumber\\
                                                             &\leq C' \|\nabla F(\x_0)\|\|\nabla  G(\x_0)\|.
\end{align}
Therefore, for any $\x\in\supp(P\omega)$, by \eqref{eq:MVT} and \eqref{eq:gammabound}, we have that
\begin{equation*}
\begin{split}
|\nabla F(\x_0)\cdot \nabla G(\x)|&\leq |\nabla F(\x_0)\cdot \nabla G(P\x_0)|+ |\nabla  F(\x_0)\cdot (\nabla  G(\x)-\nabla  G(P\x_0))|  \\
&\leq C' P^2\|\nabla G(\x_0)\|\|\nabla  F(\x_0)\|+O_{\|G\|}(\rho) P^2\|\nabla F(\x_0)\|
\end{split}
\end{equation*}
Hence (since $\|\nabla G(\x_0)\|>0$ is a constant), provided that the support $\rho $ is sufficiently small, we may choose some $0<C''<1$ independent of $\rho$ such that 
\begin{align}\label{eq:gammabound1}
|\nabla F(\x_0)\cdot \nabla G(\x)|\leq C''P^2\|\nabla F(\x_0)\|\|\nabla  G(\x_0)\|.
\end{align}

Thus, for any $\x\in\supp(P\omega)$,
\begin{equation*}
\begin{split}
|\nabla  G(\x)\cdot(\nabla G(\x_0)-\gamma \e_1')|&=|\nabla G(\x)\cdot \nabla G(\x_0)- \gamma\|\nabla  F(\x_0)\|^{-1}\nabla G(\x)\cdot \nabla F(\x_0)|\\
&\geq (1-O(\rho)-C'C'')P^2\|\nabla  G(\x_0)\|^2,
\end{split}
\end{equation*}
where we have used \eqref{eq:gammabound} to bound $\gamma$ by $C'\|\nabla G(\x_0)\|$, as well as \eqref{eq:gammabound1} and \eqref{eq:g0bound}. Hence provided that the support $\rho $ is chosen to be sufficiently small, there is some $M_{G,1}>0$ such that
\[|\nabla  G(\x)\cdot \e_2'|=\gamma_1^{-1}|\nabla  G(\x)\cdot(\nabla G(\x_0)-\gamma \e_1')|\geq M_{G,1} P^2.\]
Hence, upon taking 
\[M_1:=\min\{M_{F,1}, M_{G,1}\},\]
we conclude that \eqref{E14} is true. Finally, \eqref{eq:gammabound1} also hands us:
\begin{equation}
|\nabla  G(\x)\cdot \e_1'|=\|\nabla F(\x_0)\|^{-1}|\nabla F(\x_0)\cdot \nabla G(\x)|\leq C'' P^2\| G (\x_0)\|,
\end{equation}
for any $\x\in\supp(P\omega)$. Therefore, upon setting $M_{2,G}:=C''\| G (\x_0)\|$, and taking
\[M_2:=\max\{M_{F,2}, M_{G,2}\},\]
we are now able to verify \eqref{E15}. Furthermore, there is some $\rho_0>1$, such that $M_1$ and $M_2$ are independent of $\rho$ provided that $\rho\leq \rho_0$. This concludes the proof of the lemma. 
\end{proof}

\section{Van der Corput differencing}
\label{vdc}
In this Section, we will use van der Corput differencing to bound $S_{\ua}(q,\z)$  by a quadratic exponential sum. We will introduce the topic by beginning with the simpler \ti{pointwise} van der Corput differencing before attempting to generalise the differencing arguments used in \cite{Vishe19} to attain a bound which also takes advantage of averaging over the both $\z$ integrals. In both cases, we will innovate on the standard differencing approach in order to introduce a path to attaining Kloosterman refinement.

\subsection{Pointwise van der Corput}
\label{S: pvdc}
For convenience, we will set 
\begin{equation}
\label{eq:hatfdef}
\hat{F}_{\underline{a},q,\z}(\underline{x}):=(a_1/q+z_1)F(\underline{x})+(a_2/q+z_2)G(\underline{x}),
\end{equation}
where $F$ and $G$ are cubic forms. Since $\underline{x}$ is summed over all of $\mathbb{Z}^n$, we can replace $\underline{x}$ with $\underline{x}+\underline{h}$, for any $\underline{h}\in\mathbb{Z}^n$, giving 
\begin{equation}
    \label{eq: van der corput: 1}
    S(q,\underline{z})=\sum_{\ua}{}^*\Big{|}\sum_{\underline{x}\in\mathbb{Z}^n}\omega((\underline{x}+\underline{h})/P)e(\hat{F}_{\underline{a},q,\z}(\underline{x}+\underline{h}))\Big{|},
\end{equation}
where $S(q,\z)$ is as defined in \eqref{eq:sqdef}. Let $\mathcal{H}\subset \mathbb{Z}^n$ be a set of lattice points (which we may choose freely). In the case of pointwise van der Corput differencing, we can just take $\mathcal{H}$ to be the set of lattice points $\h$ such that $|\h|<H$, for some $1\leq H\ll P$ which we may choose freely. However, we will not specify this in the arguments that follow since we will need a different choice of $\mathcal{H}$ when we come to averaged van der Corput differencing later. Applying the Cauchy-Schwarz inequality to \eqref{eq: van der corput: 1} gives the following
\begin{align*}
  \#\mathcal{H}S(q,\underline{z})&= \sum_{\ua}{}^*\Big{|}\sum_{\underline{h}\in\mathcal{H}}\sum_{\underline{x}\in\mathbb{Z}^n} \omega((\underline{x}+\underline{h})/P)e(\hat{F}_{\underline{a},q,\z}(\underline{x}+\underline{h}))\Big{|}\\
                                               &\leq \sum_{\ua}{}^*\sum_{\underline{x}\in\mathbb{Z}^n} \Big{|}\sum_{\underline{h}\in\mathcal{H}}\omega((\underline{x}+\underline{h})/P)e(\hat{F}_{\underline{a},q,\z}(\underline{x}+\underline{h}))\Big{|}\\
                                               &\leq \Big{(} \sum_{\ua}{}^*\sum_{|\underline{x}|<2P} 1 \Big{)}^{1/2} \Big{(}\sum_{\ua}{}^*\sum_{\underline{x}\in\mathbb{Z}^n}\Big{|}\sum_{\underline{h}\in\mathcal{H}} \omega((\underline{x}+\underline{h})/P)e(\hat{F}_{\underline{a},q,\z}(\underline{x}+\underline{h}))\Big{|}^2\Big{)}^{1/2}\\
                                               &\ll  q P^{n/2} \Big{(}\sum_{\ua}{}^*\sum_{\underline{x}\in\mathbb{Z}^n} \sum_{\underline{h}_1,\underline{h}_2\in\mathcal{H}} \omega((\underline{x}+\underline{h}_1)/P)\overline{\omega((\underline{x}+\underline{h}_2)/P)}\\ 
                                               &\qquad\qquad\qquad\qquad\qquad\qquad\qquad\;\;\;\; e(\hat{F}_{\underline{a},q,\z}(\underline{x}+\underline{h}_1)) \overline{e(\hat{F}_{\underline{a},q,\z}(\underline{x}+\underline{h}_2))}\Big{)}^{1/2}.
\end{align*}
The key difference between this and the standard van der Corput differencing process is the introduction of the $\ua$ sum in the Cauchy-Schwarz step. In particular, this enables us to bring the $\ua$ sum inside of the bracket in the final step which in turn gives us a path to Kloosterman refinement. We still need to write $S(q,\z)$ in terms of a quadratic exponential sum however, so we will come back to Kloosterman refinement later.\\
Set $\underline{y}:=\underline{x}+\underline{h}_2$, $\underline{h}=\underline{h}_1-\underline{h}_2$ and recall that we defined $\omega$ to be a real weight function. Therefore, after setting 
\begin{equation}
\label{eq:Nhwomegahdef}
N(\underline{h}):=\#\{\underline{h}_2-\underline{h}_1=\underline{h}:\underline{h}_1,\underline{h}_2\in\mathcal{H}\},\textrm{  and  }\omega_{\underline{h}}(\underline{x}):=\omega(\underline{x}+P^{-1}\underline{h})\omega(\underline{x}),
\end{equation}
 we get
\[|S(q,\underline{z})|^2\ll \#\mathcal{H}^{-2}q^2P^{n}\sum_{\ua}{}^*\sum_{\underline{y}\in\mathbb{Z}^n} \sum_{\underline{h}\in\mathcal{H}} N(\underline{h})\omega_{\underline{h}}(\underline{y}/P) e(\hat{F}_{\underline{a},q,\z}(\underline{y}+\underline{h})-\hat{F}_{\underline{a},q,\z}(\underline{y})).\]
Recall that $\hat{F}_{\underline{a},q,\z}(\underline{x})=(a_1/q+z_1)F(\underline{x})+(a_2/q+z_2)G(\underline{x})$. Therefore if we set $F_{\h}$ and $G_{\h}$ be the differenced polynomials
\begin{equation*}
F_{\underline{h}}(\underline{y}):=F(\underline{y}+\underline{h})-F(\underline{y}),\quad G_{\underline{h}}(\underline{y}):=G(\underline{y}+\underline{h})-G(\underline{y}),
\end{equation*}
we have
\[\hat{F}_{\underline{a},q,\z}(\underline{y}+\underline{h})-\hat{F}_{\underline{a},q,\z}(\underline{y})=(a_1/q+z_1)F_{\h}(\underline{y})+(a_2/q+z_2)G_{\h}(\underline{y}).\]
Hence
\begin{equation}
    \label{E430}
|S(q,\underline{z})|^2\ll \#\mathcal{H}^{-2}P^{n}q^2 \sum_{\underline{h}\in \mathcal{H}} N(\underline{h})T_{\underline{h}}(q,\underline{z}),
\end{equation}
where
\begin{equation}
    \label{eq:Thmaindef0}
    T_{\underline{h}}(q,\underline{z}):= \starsum_{\ua\bmod{q}}\sum_{\underline{y}\in\mathbb{Z}^n} \omega_{\underline{h}}(\underline{y}/P)e((a_1/q+z_1)F_{\h}(\underline{y})+(a_2/q+z_2)G_{\h}(\underline{y}))
\end{equation}
denote the corresponding exponential sum for the system of quadratic polynomials $F_{\h}$ and $G_{\h}$. Note that the top form of $F_{\h}$, $F_{\h}^{(0)}$, is precisely \eqref{eq:fhghdef}. Finally, by noting that $N(\h)\leq \#\mathcal{H}=H^n$, we arrive at the following:
\begin{lemma}
\label{T46}
For any $1\leq H\ll P$, for any fixed choice of $\uz\in[0,1]^2$, we have
\[|S(q,\z)|\ll H^{-n/2}P^{n/2}q\Bo\sum_{\underline{h}\ll H}|T_{\underline{h}}(q,\underline{z})|\Bc^{1/2}.\]
\end{lemma}
This bound will be useful to us when $t:=|\z|$ is small, say of size $P^{-3-\Delta}$, since it is wasteful to use \ti{averaged} van der Corput differencing in this case. We will now set up averaged van der Corput differencing, which will be a key in proving Proposition \ref{P02}.

\subsection{Averaged van der Corput}
Throughout this section, $\x_0$ will denote a fixed point satisfying $|\x_0|<1$ in $\x_0\in\supp(\omega)$, where $\supp(\omega)$ is contained in the set $\x_0+(-\rho,\rho)^n$. Likewise, $F$ and $G$ will be cubic polynomials whose leading forms satisfy \eqref{E14} and \eqref{E15} for a fixed orthonormal set of vectors $\e_1',\e_2'$ (see \eqref{eq:e1'e2'}). Let
\begin{equation}
\label{eq:extbasis}
\{\e_1',...,\e_n'\},
\end{equation}
denote an extended orthonormal basis of $\R^n$. We will begin our effort to bound the sum
\begin{equation}
\label{E41}
    \sum_{P^{\Delta}\leq q\leq Q }\int_{P^{-3-\Delta}\leq|\underline{z}|\leq 1/qQ^{1/2}} S(q,\underline{z})\: d\underline{z},
\end{equation}
where $S(q,\z)=\starsum_{\ua\bmod{q}} |S_{\ua}(q,\z)|$ is as defined in \eqref{eq:sqdef}. As in the previous section, let $1\leq H\ll P$ be a parameter to be chosen later. Typically, $H$ will be chosen as a small power of $P$, so it is safe to further assume $H\log P\ll P$. Also, let $\ve>0$ be an arbitrarily small absolute constant to be chosen at the end. Note that the implied constants will be allowed to depend on the choice of $\ve$ after it is introduced into our bounds. As is standard (\cite{Vishe19} for example), we start by splitting the integral over $\z$ above as a sum over $O(P^\ve)$ \ti{dyadic intervals} of the form $[t,2t]$ where $P^{-3+\Delta}\leq t\leq 1/(qQ^{1/2}) $. For convenience, given $t\in \R_{>0}^2$, we will set 
\[I(q,t):=\int_{t\leq |\uz|\leq 2t} S(q,\underline{z})\: d\underline{z}.\]
Analogous to \cite{Hanselmann} and \cite[Sec. 3]{Marmon_Vishe}, for a fixed value of $P^{-3-\Delta}<t<1/qQ^{1/2}$ we choose two sets $T_1$, $T_2$, each of cardinality $O(1+tHP^2)$ such that 
\begin{align}
\label{eq:z}
\{\z: t\leq |\uz|\leq 2t\}&\subseteq \bigcup_{\un{\tau} \in T_1\times T_2} \big{[}\tau_1-(HP^2)^{-1},\tau_1+(HP^2)^{-1}\big{]}\times \big{[}\tau_2-(HP^2)^{-1},\tau_2+(HP^2)^{-1}\big{]}\nonumber\\
&\subseteq \{\z: t-(HP^2)^{-1}\leq|\uz|\leq 2t+(HP^2)^{-1}\}.
\end{align}
Thus, an application of Cauchy-Schwarz further gives
\begin{equation}\label{eq:Iqbound1}
I(q,t)\ll ((HP^2)^{-1}+t)\sum_{\underline{\tau}\in \uT} \mathcal{M}_q(\underline{\tau},H)^{1/2},
\end{equation}
where
\begin{align}
    \label{E42}
    \mathcal{M}_q(\underline{\tau},H):&=\int_{\underline{\tau}-(HP^2)^{-1}}^{\underline{\tau}+(HP^2)^{-1}} |S(q,\underline{z})|^2 \:d\underline{z}\nonumber\\
                                        &\ll \int_{\mathbb{R}^2} \exp(-H^2P^4[(\tau_1-z_1)^2+(\tau_2-z_2)^2])|S(q,\underline{z})|^2 \:d\underline{z}.
\end{align}
Here we have used $\uT:=T_1\times T_2$, and $\int_{\underline{\tau}-(HP^2)^{-1}}^{\underline{\tau}+(HP^2)^{-1}}$ to denote the integral $$\int_{(\tau_1-(HP^2)^{-1}, \tau_1+(HP^2)^{-1})\times (\tau_2-(HP^2)^{-1}, \tau_2+(HP^2)^{-1})} $$ in order to simplify the notation.
After an inspection of the right hand side of \eqref{eq:z}, it is easy to see that
\begin{equation*}
\int_{P^{-3-\Delta}\leq |\underline{z}|\leq 1/qQ^{1/2}} S(q,\underline{z})\: d\underline{z}\ll \sum_{t}((HP^2)^{-1}+t)\sum_{\underline{\tau}\in \uT} \mathcal{M}_q(\underline{\tau},H)^{1/2},
\end{equation*}
where the sum over $t$ runs over $O_\ve(P^\ve)$ choices satisfying
\begin{equation}
\label{eq:tbound}
P^{-3-\Delta}\leq t\leq 1/(qQ).
\end{equation}
Note that the choice of the parameter $H$ will ultimately depend on $t$. For now, we will assume $t$ to be fixed. 

We are therefore first led to find a bound for $|S(q,\underline{z})|^2$ using van der Corput differencing. We may now use the same arguments as those from Section \ref{S: pvdc} to arrive at the following: 
\begin{equation}
    \label{E43}
|S(q,\underline{z})|^2\ll \#\mathcal{H}^{-2}P^{n}q^2 \sum_{\underline{h}\in \mathcal{H}} N(\underline{h})T_{\underline{h}}(q,\underline{z}),
\end{equation}
where $\mathcal{H}\subset \Z^n$ is a set of lattice points to be chosen later, and 
\begin{equation}
    \label{eq:Thmaindef}
    T_{\underline{h}}(q,\underline{z}):= \starsum_{\ua\bmod{q}}\sum_{\underline{y}\in\mathbb{Z}^n} \omega_{\underline{h}}(\underline{y}/P)e((a_1/q+z_1)F_{\h}(\underline{y})+(a_2/q+z_2)G_{\h}(\underline{y})),
\end{equation}
denotes the corresponding exponential sum for the system of quadratic polynomials $F_{\h}$ and $G_{\h}$ (this is a restating of \eqref{E430} and \eqref{eq:Thmaindef0}).

Therefore by \eqref{eq:Iqbound1}, \eqref{E42}, and \eqref{E43}, we have shown the following:
\begin{lemma}
\label{T41}
Let 
\begin{align}
    \label{eq: vdc: U(q,z) def}
    U(q,\z):=\sum_{\un{\tau}\in \un{T}}\Big{(}\sum_{\underline{h}\in \mathcal{H}} N(\underline{h})\int_{\mathbb{R}^2} \exp(-H^2P^4[(\tau_1-z_1)^2+(\tau_2-z_2)^2])T_{\underline{h}}(q,\underline{z})\; d\underline{z}\Big{)}^{1/2}.
\end{align}
Then, for any $1\leq H\leq P$, $\mathcal{H}\subset \mathbb{Z}^n$, and $t$ satisfying \eqref{eq:tbound} we have
\begin{equation}\label{eq:Iqtb}
\begin{split}
I(q,t)\ll &((HP^2)^{-1}+t)\#\mathcal{H}^{-1}P^{n/2}q\, U(q,\z).
\end{split}
\end{equation}
\end{lemma}
Since we intend to develop a two dimensional version of averaged van der Corput differencing, we intend to choose $\mathcal{H}$ to be a set of size $O(P^2H^{n-2})$ and then use averaging over $ z_1$ and $z_2$ to show that for all but $O((H\log(P))^n)$ of $\underline{h}\in\mathcal{H}$, the value of the averaged integral $\mathcal{M}_q(\underline{\tau},H)$ defined in \eqref{E42} is negligible. This will enable us to `win' an extra factor of $P/H$ in our final estimate for \eqref{E41} when compared to pointwise van der Corput differencing.

Our choice of $\mathcal{H}$ will be informed by the following lemma:
\begin{lemma}
\label{T42}
For any $\underline{h}\in\mathbb{R}^n$, any $1\leq H\leq P$, any fixed $\utau$ and any $N>0$,
\[\int_{-\infty}^{\infty}\int_{-\infty}^{\infty}\exp(-H^2P^4[(\tau_1-z_1)^2+(\tau_2-z_2)^2])T_{\underline{h}}(q,\underline{z})\: d\underline{z}\ll_NP^{-N},\]
provided that $\underline{h}=\sum_{i=1}^n h_i'\underline{e}_i'$ satisfies the following condition: 
\begin{equation}\label{eq:condition}
H\mathcal{L}\ll |h_1'|\ll P \quad \mbox{ or }\quad H\mathcal{L}\ll |h_2'|\ll P,\quad|h_i'|<H \mbox{ for } i\in\{3,\cdots,n\},
\end{equation}
where $\cL=\log(P)$, $\{\e_1',...,\e_n'\}$ denote the basis chosen in \eqref{eq:extbasis} and the implied constants only depend on $n,\|F\|$ and $\|G\|$.
\end{lemma}

\begin{proof}
We start by rewriting:
\begin{align*}
    \int_{\R^2}\exp(-H^2P^4[(\tau_1-z_1)^2+(\tau_2-&z_2)^2])T_{\underline{h}}(q,\underline{z})\: d\underline{z}\: \\
    &=\sum_{\underline{y}\in\mathbb{Z}^n}\sum_{\ua}{}^*\omega_{\underline{h}}(\underline{y}/P)e_q(a_1F_{\h}(\y)+a_2G_{\h}(\y))J(\underline{h},\underline{y}),
\end{align*}\
where 
\begin{align}
    \label{eq: vdc : J(h,y) definition}
     J(\underline{h},\underline{y})=\int_{\mathbb{R}^2}\exp(-H^2P^4[(\tau_1-z_1)^2+(\tau_2-z_2)^2])e(z_1F_{\underline{h}}(\underline{y})+z_2G_{\underline{h}}(\underline{y})) d\underline{z},
\end{align}
and $e_q(x):=e^{2\pi i x/q}$. We may separate the two integrals over $\z$ and integrate them to get
\begin{align}
     J(\underline{h},\underline{y})= \frac{\pi}{H^2P^4} \exp\Big{(}-\frac{\pi^2}{H^2P^4}\big{(}|F_{\underline{h}}(\underline{y})|^2+|G_{\underline{h}}(\underline{y})|^2\big{)}\Big{)}&e(-\tau_1F_{\underline{h}}(\underline{y})-\tau_2G_{\underline{h}}(\underline{y}))\nonumber.
\end{align}
We note that if either $|F_{\underline{h}}(\underline{y})|$ or $|G_{\underline{h}}(\underline{y})|$ are $\gg HP^2\mathcal{L}$, then trivially bounding everything in $J$ from above gives:
\begin{align*}
\sum_{\underline{y}\in\mathbb{Z}^n}\starsum_{\ua\bmod{q}}\omega_{\underline{h}}(\underline{y}/P)e_q(a_1F_{\h}(\y)+a_2G_{\h}(\y))J(\underline{h},\underline{y})&\ll P^nq^2\frac{1}{H^2P^4}\exp(-m\mathcal{L}^2)\\&\ll_N P^{-N},
\end{align*}
for some constant $m>0$. Therefore it is sufficient to show that there exist constants $0<c_1,c_2<1$ such that for every $\underline{h}\in\mathbb{R}^n$ with
\begin{align}
\label{E44}
H\mathcal{L}\ll |h_1'|<c_1P \quad \mbox{ or }\quad H\mathcal{L}\ll |h_2'|<c_2P,\quad|h_i'|<H \mbox{ for } i\in\{3,\cdots,n\},
\end{align}
we have 
\begin{equation}
    \label{E45}
    |F_{\underline{h}}(\underline{y})|\gg HP^2\mathcal{L}\quad \mbox{or}\quad |G_{\underline{h}}(\underline{y})|\gg HP^2\mathcal{L},
\end{equation}
where $\underline{h}'=(h_1',\cdots,h_n')$ is defined by
\begin{equation}
    \label{E46}
    \underline{h}=\sum_{i=1}^n h_i\underline{e}_i=\sum_{i=1}^n h_i'\underline{e}_i'.
\end{equation}

We will rewrite $F_{\underline{h}}$ as follows:
\[F_{\underline{h}}(\underline{y})=\nabla F(\underline{y})\cdot \underline{h}+\underline{h}^t\mathcal{H}_F(\underline{y})\underline{h}+F_{\underline{h}}^{(2)}\]
where $F_{\underline{h}}^{(2)}$ is the constant part of $F_{\underline{h}}$ and $\mathcal{H}_F(\underline{y})$ is the Hessian of $F$ evaluated at $\underline{y}$. Now for $\underline{h}$ satisfying \eqref{E44}, we have
\begin{align}
\nonumber F_{\underline{h}}(\underline{y})&=\nabla F(\underline{y})\cdot \underline{h}+\Big{(}\sum h_i'\underline{e}_i'\Big{)}^t\mathcal{H}_F(\underline{y})\Big{(}\sum h_i'\underline{e}_i'\Big{)}+F_{\underline{h}}^{(2)}\\
                           \label{eq:fhb1}&=\nabla F(\y)\cdot \underline{h}+F_{\underline{h}}^{(2)}+O(|h_1'|^2P)+O(|h_2'|^2P)+O(HP^2),
\end{align}
where $F_{\h}^{(2)}$ is a cubic polynomial in $\h$, and the implied constants depend only on $\|F\|$, $\|G\|$ and $n$. Note that
\begin{align*}
     F_{\underline{h}}^{(2)}&= O(|h_1'|^3)+O(|h_2'|^3)+O(H^3),
\end{align*}
and so we may simplify \eqref{eq:fhb1} to
\begin{align}
\label{eq: vdc differencing: f_h(y) equality 1}
F_{\underline{h}}(\underline{y})&=\nabla F(\underline{y})\cdot \underline{h} +O(|h_1'|^2P)+O(|h_2'|^2P)+O(HP^2),
\end{align}
since $H, |h_1'|, |h_2'|<P$. We also write $\h=h_1'\e_1'+...+h_n'\e_n'$ and invoke \eqref{E14} and \eqref{E15} to further get that for all $\y\in\supp(P\omega)$ we have
\begin{align*}
|\nabla F(\underline{y})\cdot \underline{h}|\geq |h_1'| M_1P^2+O(\rho |h_2'| P^2)+O(HP^2),
\end{align*}
and so we get
\begin{align}
\label{eq: vdc differencing: 1}
|F_{\underline{h}}(\underline{y})|\geq M_1|h_1'| P^2+O(\rho |h_2'| P^2)+O(|h_1'|^2P)+O(|h_2'|^2P)+O(HP^2),
\end{align}
by \eqref{eq: vdc differencing: f_h(y) equality 1}. For now, let us focus on the case $|h_2'|\ll\rho^{-1/2}| h_1'|$. In this case, we must have that $h_1'$ satisfies \eqref{E44}. Furthermore, upon choosing $c_1\leq \rho^2$ and by \eqref{E44}, we have
\begin{align*}
    \rho |h_2'| P^2&\ll \rho^{1/2} |h_1'|P^2, 
    &|h_1'|^2P\leq c_1 |h_1'|P^2 \leq \rho^2 |h_1'|P^2,\\
    |h_2'|^2P&\ll \rho^{-1}|h_1'|^2P\leq \rho^{-1}c_1|h_1'| P^2 \leq \rho |h_1'|P^2,
    &HP^2\ll |h_1'|P^2 \mathcal{L}^{-1}\ll \rho |h_1'|P^2.
\end{align*}
Hence, we may simplify \eqref{eq: vdc differencing: 1} to obtain
\begin{align*}
    |F_{\underline{h}}(\underline{y})|\geq M_1|h_1'| P^2+O(\rho^{1/2} |h_1'| P^2)\gg |h_1'|P^2\gg HP^2\mathcal{L},    
\end{align*}
provided that $\rho$ is chosen to be sufficiently small with respect to $M_1$.

It now remains to study the case $|h_1'|\ll \rho^{1/2}|h_2'|$. In this case, we instead have that $h_2'$ must satisfy the bound in \eqref{E44}. We now apply the same process used to obtain \eqref{eq: vdc differencing: f_h(y) equality 1} to $G_{\underline{h}}(\underline{y})$ to obtain
\begin{equation}
\label{E47}
    G_{\underline{h}}(\underline{y})=\nabla G(\underline{y})\cdot \underline{h}+O(|h_1'|^2P)+O(|h_2'|^2P)+O(HP^2),
\end{equation}
where the implied constants again depend only on $n$, $\|F\|$ and $\|G\|$. Note again that 
\begin{align*}
    \nabla G(\y)\cdot \h=h_1'\nabla G(\y)\cdot \e_1'+h_2'\nabla G(\y)\cdot \e_2'+O(HP^2).
\end{align*} 
Combining this with \eqref{E47}, and applying \eqref{E14} - \eqref{E15} gives
\begin{align}
\label{eq: vdc differencing: 2}
|G_{\underline{h}}(\underline{y})|\geq M_1|h_2'| P^2+O(|h_1'| P^2)+O(|h_1'|^2P)+O(|h_2'|^2P)+O(HP^2).
\end{align}
We now aim to simplify \eqref{eq: vdc differencing: 2}. Using the assumption that $|h_1'|\ll \rho^{1/2}|h_2'|$, the fact that $|h_2'|$ must obey \eqref{E44} in this case, and setting $c_2\leq \rho$ we have
\begin{align*}
    |h_1'| P^2&\ll \rho^{1/2} |h_2'|P^2, 
    &|h_1'|^2P\ll \rho |h_2'|^2P\leq \rho c_2|h_2'| P^2 \leq \rho^2 |h_2'|P^2,\\
    |h_2'|^2P&\leq c_2 |h_2'|P^2 \leq \rho |h_1'|P^2,&
    HP^2\ll |h_1'|P^2 \mathcal{L}^{-1}\ll \rho |h_1'|P^2.
\end{align*}
Hence
\begin{align*}
|G_{\underline{h}}(\underline{y})|\geq M_1|h_2'|P^2+O(\rho^{1/2} |h_2'| P^2)\gg |h_2'|P^2\gg HP^2\mathcal{L},
\end{align*}
as long as $\rho$ is chosen small enough.

\end{proof}

The lemma above leads to the following natural choice for $\mathcal{H}$:
\begin{equation}
\label{E49}
\mathcal{H}:=\{\underline{h}\in\mathbb{Z}^n:0\leq h_1'<c_1P,\; 0\leq h_2'<c_2P,\;0\leq h_i'<H \mbox{ for } i\in\{3,\cdots,n\}\},
\end{equation}
where $c_1$ and $c_2$ are the implied constants arising in \eqref{eq:condition}.
Essentially, $\mathcal{H}$ is chosen to be the collection of lattice points inside of a fixed $n$ dimensional cuboid, $B_P$, centred at the origin, with volume $\Vol(B_P)=c_1c_2P^2H^{n-2}$. The sides of the cuboid are in the direction of the basis vectors $\{\underline{e}_1',\cdots,\underline{e}_n'\}$. We now claim that 
\begin{equation}\label{eq:Hcount}
P^2H^{n-2}\ll\#\mathcal{H}\ll P^2H^{n-2}.
\end{equation} This follows very easily from the following asymptotic formula for a general cuboid $B$ with side lengths $l_1,\cdots,l_n$. It is easy to see that
\[\#\{\Z^n\cap B\}=\Vol(B)+\sum_{i=1}^n O(\prod_{j\neq i} l_j).\]
The error comes from estimating the $n-1$ dimensional boundary of $B$. In our case $l_1=c_1P$,$l_2=c_2P$, $l_i=H$ for $i\geq 3$, which leads to \eqref{eq:Hcount}. Note that $\mathcal{H}$ is chosen as in \eqref{E49} so that we can use the bound Lemma \ref{T42}. In particular, we can now show the following:
\begin{lemma}
\label{T43}
Let $1\leq H\leq P$ and let
\[\tilde{\mathcal{H}}:=\{\underline{h}\in\mathbb{Z}^n: |\h|\ll H\mathcal{L}\}.\]
Then for any $1\leq H\leq P$, any $1\leq N$, and any $t>0$ such that \eqref{eq:tbound} holds, we have
\begin{equation*}
\begin{split}
I(q,t)\ll H^{-n/2+1}(\log P)^{1/2}P^{n/2-1}q((HP^2)^{-1}+t)^2\Big{(}\sum_{\underline{h}\in \tilde{\mathcal{H}}}\max_{\uz}|T_{\underline{h}}(q,\underline{z})|\; &\Big{)}^{1/2}+O_N(P^{-N}),
\end{split}
\end{equation*}
where the maximum over $\uz$ is taken over the set 
\begin{equation}
\label{eq:uzregion}
t-(HP^2)^{-1}\leq |\uz|\leq 2t+(HP^2)^{-1}.
\end{equation}
\end{lemma}
\begin{proof}
Let $\scrH$ be as in \eqref{E41}. Then we use the decomposition $\mathcal{H}=\tilde{\mathcal{H}}\cup \mathcal{H}\backslash\tilde{\mathcal{H}}$. 
By construction, 
\begin{align*}
    \mathcal{H}\backslash\tilde{\mathcal{H}}=\{\underline{h}\in\mathbb{Z}^n:H\mathcal{L}\ll |h_1'|<c_1P \mbox{ or }\; &H\mathcal{L}\ll |h_2'|<c_2P,\: |h_i'|<H, \mbox{ for } i\in\{3,\cdots,n\}\}.
\end{align*}
Furthermore, note that for any fixed $\h$, $N(\h)$ as defined in \eqref{eq:Nhwomegahdef} satisfies the bound
\begin{equation}
\label{eq:Hboun}
N(\h)\ll \#\scrH\ll P^2H^{n-2}.
\end{equation}
 
Therefore by Lemma \ref{T42},
\begin{align*}
\#\scrH^{-1}\left(\sum_{\underline{h}\in \mathcal{H}\setminus\tilde{\scrH}} N(\underline{h})\int_{\mathbb{R}^2} \exp(-H^2P^4[(\tau_1-z_1)^2+(\tau_2-z_2)^2])T_{\underline{h}}(q,\underline{z}) d\underline{z}\right)^{1/2}
           \ll P^{-N}.
\end{align*}
Further combining with the bounds $q\leq Q\leq P^{3/2} $ and $\#\uT\ll (1+tHP^2)^2\ll P^6$, which arises from using crude bounds $t\leq 1$ and $1\leq H\leq P$, we may bound the contribution from the sum over $\h\in\scrH\setminus \tilde{\scrH}$ in \eqref{eq:Iqtb}, $U(q,\z)$, as follows:
\begin{align*}
U(q,\z)&\ll((HP^2)^{-1}+t)P^{n/2}q\#\mathcal{H}^{-1}\times\\
&\qquad\qquad\qquad\quad\sum_{\un{\tau}\in \un{T}}\Big{(}\sum_{\underline{h}\in \mathcal{H}} N(\underline{h})\int_{\mathbb{R}^2} \exp(-H^2P^4[(\tau_1-z_1)^2+(\tau_2-z_2)^2])T_{\underline{h}}(q,\underline{z})\; d\underline{z}\Big{)}^{1/2}\\
&\ll_N P^{-2+n/2+3/2-N}\ll_N P^{(n-1)/2-N}\ll_{n,N} P^{-N},
\end{align*}
as $N$ is allowed to be arbitrarily large.
Therefore, combining this with Lemma \ref{T41}, we get
\begin{align}\label{eq:12}
I(q,\ut)\ll (H&P^2)^{-1}\#\mathcal{H}^{-1/2}P^{n/2}q\:\times\nonumber\\
&\sum_{\tau\in T}\Big{(}\sum_{\underline{h}\in \tH}\int_{\mathbb{R}^2} \exp(-H^2P^4[(\tau_1-z_1)^2+(\tau_2-z_2)^2])T_{\underline{h}}(q,\underline{z})\; d\underline{z}\Big{)}^{1/2}+O_{n,N}(P^{-N}).
\end{align}
Further note that for a fixed $\tau$ and for any $z$ satisfying $|z-\tau|\geq HP^2\scrL$ we have the following decay of the function in the integrand:
\begin{equation}\label{eq:integrand}
\exp(-H^2P^4(\tau-z)^2)\ll \frac{\exp(-\cL^2/2)}{|z-\tau|^2+1}\ll_N \frac{P^{-N}}{|z-\tau|^2+1}.
\end{equation} 
Thus, in the same vein as before, using bound \eqref{eq:integrand} in \eqref{eq:12} we may obtain
\begin{align*}
I(q,\ut)\ll ((HP^2)^{-1}+t)\#\mathcal{H}^{-1/2}P^{n/2}q\sum_{\utau\in \uT}\Big{(}\sum_{\underline{h}\in \tH}\int_{\utau-((HP^2)^{-1}+t)\cL}^{\utau+(HP^2)^{-1}\cL}|T_{\underline{h}}(q,\underline{z})|\; d&\underline{z}\Big{)}^{1/2}+O_{n,N}(P^{-N}).
\end{align*}
The lemma now follows after using \eqref{eq:Hcount} to estimate $\#\scrH$, using the estimate $\#\uT=O((1+tHP^2)^2)$, and \eqref{eq:z} which allows us to take the maximum over all possible $\z$ appearing in the expression.
\end{proof}
 Since $H$ is arbitrary, we may re-label $H\mathcal{L}$ as $H$ at the expense of a factor of size at most $O_\ve(P^\ve))$ we can now conclude the following
\begin{lemma}
\label{T45}
For any $1\leq H\ll P$, any $0<\ve<1$, any $\ut$ satisfying \eqref{eq:tbound} and any $N\geq 1$ we have
\[I(q,t)\ll_{\ve,n,N} H^{-n/2+1}P^{n/2-1+\ve}q((HP^2)^{-1}+t)^2\Big{(}\max_{|\uz|}\sum_{|\h|\ll H}|T_{\underline{h}}(q,\underline{z})|\; \Big{)}^{1/2}+P^{-N},\]
where the maximum over $\uz$ is taken over the set \begin{equation}
\label{eq:uzregion1}
t-P^\ve(HP^2)^{-1}\leq |\uz|\leq 2t+P^\ve(HP^2)^{-1}.
\end{equation}
\end{lemma}

\section{Quadratic Exponential Sums: Initial Consideration}
\label{Sec: Exponential Sums: Initial Bounds}
The differencing technique used in Section \ref{vdc} leads us to consider quadratic exponential sums $T_\h(q,\uz)$ (see \eqref{eq:Thmaindef}) for a family of differenced quadratic forms $F_{\h}$ and $G_\h$. Throughout this section, let $q$ denote an arbitrary but fixed integer.  Our main goal here is to estimate quadratic sums corresponding to a general system of quadratic polynomials $F,G$ defined as
\begin{equation}
\label{eq:Thdef1}
T(q,\z):=\starsum_{\ua}^q\sum_{\underline{y}\in\mathbb{Z}^n} \omega(\underline{y}/P)e((a_1/q+z_1)F(\underline{y})+(a_2/q+z_2)G(\underline{y})).
\end{equation}
Here $F$ and $G$ denote a system of quadratic polynomials with integer coefficients and $\omega$ denotes a compactly supported function on $\R^n$.  Let us denote leading quadratic parts of $F$ and $G$ by $F^{(0)}$ and $G^{(0)}$ respectively.  We further assume that the quadratic forms $\Fnull$ and $\Gnull$ are defined by integer matrices $M_1$ and $M_2$ respectively. We will later apply the estimates in this section by setting $F=F_\h$ and $G=G_\h$.

Given a (finite or infinite) prime $p$, by $s_p$ we denote 
\begin{equation}
\label{eq:mpdef}
s_p:=s_p(\Fnull,\Gnull),
\end{equation}
where further, given a set of forms $F_1,...,F_R$,  
$s_p(F_1,...,F_R)$ denotes the dimension of singular locus of the projective complete intersection variety defined by the simultaneous zero locus of the forms $F_1,...,F_R$. That is:
\begin{align*}
    s_p(F_1,...,F_R):=\dim \{\x\in \mb{P}_{\overline{\FF}_p}^n \: : \: F_1(\x)=\cdots&=F_R(\x)=0,\rank_p (\nabla F_1(\x), \cdots, \nabla F_R(\x))<2\}.
\end{align*}
When $n\geq 2$, given an integer $q$, we define $D(q)$ by
\begin{equation}
\label{eq:Ddef}
D(q):=\prod_{\substack{p\mid q\\ p \textrm{ prime }}} p^{s_p+1}.
\end{equation}
On the other hand, when $n=1$, we define $D(q)$ as
\begin{equation}
\label{eq:Ddef1}
D(q):=(q,\cont(F^{(0)}),\cont(G^{(0)})),
\end{equation}
where, given a polynomial $F$, $\cont(F)$ is the gcd of all its coefficients.
 
As is standard, we begin by applying Poisson summation to $T(q,\z)$. This will allow us separate the sum over $\ua$ and the integral over $\z$, into an exponential sum and an exponential integral respectively. In particular, applying Poisson summation gives us the following:
\begin{lemma}
\label{T61}
We have
\[T(q,\z)=q^{-n} \sum_{\underline{m}\in\mathbb{Z}} S(q;\underline{m})I(\z;q^{-1}\underline{m})\]
where
\begin{equation}\label{eq:Sexpsumdef}
S(q; \underline{m},F, G)=S(q; \underline{m}):=\starsum_{\ua}^q\sum_{\underline{u} \bmod{q}} e_q(a_1F(\underline{u})+a_2G(\underline{u})+\underline{m}\cdot\underline{u}),
\end{equation}
and
\begin{equation}\label{eq:Iexpintdef}
I(\underline{\gamma};\underline{k}):=\int_{\mathbb{R}^n}\omega(\underline{x}/P)e(\gamma_1F(\underline{x})+\gamma_2G(\underline{x})-\underline{k}\cdot \underline{x}) \;d\underline{x}.
\end{equation}
\end{lemma}

\begin{proof}
The proof of Lemma \ref{T61} is standard and can be obtained by slightly modifying \cite[Lemma 8]{Browning-Heath-Brown09}: Let $\x=\uu+q\uv$. Then
\begin{align*}
    T(q,\z)&=\starsum_{\ua}^q\sum_{\underline{u} \bmod{q}}\sum_{\uv\in\Z^n} \omega((\uu+q\uv)/P)e([a_1/q+z_1]F(\underline{u}+q\uv)+[a_2/q+z_2]G(\underline{u}+q\uv))\\
    &=\starsum_{\ua}^q\sum_{\underline{u} \bmod{q}} e_q(a_1F(\uu)+a_2G(\uu))\sum_{\uv\in\Z^n}\omega((\uu+q\uv)/P)e(z_1F(\underline{u}+q\uv)+z_2G(\underline{u}+q\uv)).
\end{align*}
We now apply Poisson summation on the second sum (and use the substitution $\x=\uu+q\uv$) to get
\begin{align*}
T(q,\z)&=\starsum_{\ua}^q\sum_{\underline{u} \bmod{q}} e_q(a_1F(\uu)+a_2G(\uu)) \times \\
    &\quad\quad\quad\quad\sum_{\m\in\Z^n}\int_{\R^n} \omega((\uu+q\uv)/P)e(z_1F(\underline{u}+q\uv)+z_2G(\underline{u}+q\uv)-\m\cdot \uv)\, d\uv\\
    &=q^{-n}\sum_{\m\in\Z^n}\starsum_{\ua}^q\sum_{\underline{u} \bmod{q}} e_q(a_1F(\uu)+a_2G(\uu)+\m\cdot \uu)\times\\
    &\quad\quad\quad\quad\quad\quad\quad\quad\quad\quad\int_{\R^n} \omega(\x/P)e(z_1F(\x)+z_2G(\x)-q^{-1}\m\cdot \x),\, d\x\\
\end{align*}
as required.
\end{proof}
As a result, we trivially have the following pointwise bound
\begin{equation}
    \label{eq: T_h(q,z): Poisson Ptwise bound: 1 }
    |T(q,\z)|\leq q^{-n}\sum_{\underline{m}\in\mathbb{Z}} |S(q;\underline{m})|\cdot |I(\z;q^{-1}\underline{m})|.
\end{equation}
The treatment of the exponential integral is standard. In particular, we can use the following lemma to bound $I(\z;q^{-1}\underline{m})$:
\begin{lemma}
\label{E620}
Let $F,G$ be quadratic polynomials such that $\max\{\|F\|_P,\|G\|_P\}\ll H$, where
\begin{align}
    \|F\|_P:=\|P^{-\deg(F)}F(Px_1,\cdots,Px_n)\|.
\end{align}
Let $V:=1+q P^{\ve-1}\max\{1,HP^2|\z|\}^{1/2}$, $\ve>0$, and $N\in\mathbb{N}$. Then 
\begin{align*}
    I(\z;q^{-1}\underline{m})\ll_N P^{-N}+ \meas(\{\underline{y}\in P\,\supp(\omega_{\underline{h}}):\: |\nabla \hat{F}_{\z}(\underline{y})-\underline{m}|\leq V\}),
\end{align*}
where
\[\hat{F}_{\z}(\x):=qP^{-1}z_1F(\x)+qP^{-1}z_2G(\x).\]
Furthermore, if $|\underline{m}|\geq q P^{\epsilon-1}\max\{1,HP^2|\z|\}$, then we have
\[I(\z;q^{-1}\underline{m})\ll_N P^{-N}|\underline{m}|^{-N}.\]
\end{lemma}
The proof of this is almost identical to the proofs of \cite[Lemma 6.5-6.6]{Browning-Dietmann-Heath-Brown}, and so we will not detail it here. In particular, the only thing in the proofs that needs to be tweaked in order to verify Lemma \ref{E620} is that $\Theta$ in \cite[equation (6.11)]{Browning-Dietmann-Heath-Brown} must be replaced with 
\[\Theta':=1+|z_1|HP^2+|z_2|HP^2.\]
We also note that we use $|\nabla \hat{F}_{\z}(\underline{y})-\underline{m}|\leq V$ instead of $Pq^{-1}|\nabla \hat{F}_{\z}(\underline{y})-\underline{m}|\leq Pq^{-1}V$ since we are using slightly different notation.   

The latter bound enables us to handle the tail of the sum over $\underline{m}$. Let
\\$\hat{V}:=q P^{\epsilon-1}\max\{1,HP^2|\z|\}$. By trivially bounding $|S(q;\underline{m})|$ by $q^n$, and setting $N\geq n+2$, it is easy to show that
\[q^{-n}\sum_{|\underline{m}|\gg \hat{V}} |S(q;\underline{m})|\cdot |I(\z;q^{-1}\underline{m})|\ll 1,\]
by the second half of Lemma \ref{E620}. Hence,
\[\implies |T_{\underline{h}}(q,\z)|\ll 1+q^{-n}\sum_{|\underline{m}|\ll \hat{V}} |S(q;\underline{m})|\cdot |I(\z;q^{-1}\underline{m})|.\]
Now by the first half of Lemma \ref{E620} (setting $N\geq n+4$), we have 
\begin{align*}
    |T_{\underline{h}}(q,\z)|&\ll 1+q^{-n}\sum_{|\underline{m}|\ll \hat{V}} |S(q;\underline{m})|\cdot \meas(\{\underline{y}\in P\,\supp(\omega):\: |\nabla \hat{F}_{\z}(\underline{y})-\underline{m}|\leq V\}\\
    &= 1+q^{-n}\sum_{|\underline{m}|\ll \hat{V}} |S(q;\underline{m})|\int_{\underline{y}\in P\,\supp(\omega)} \Char_G(\underline{m},\underline{y})\; d\underline{y},
\end{align*}
where 
\[\Char_G(\underline{m},\underline{y})=\begin{cases}1 \;\mbox{ if }\; |\nabla \hat{F}_{\z}(\underline{y})-\underline{m}|\leq V\\ 0 \;\mbox{ else.}\end{cases}\]
\begin{align*}
    \implies |T_{\underline{h}}(q,\z)|&\ll 1+q^{-n}\int_{\underline{y}\in P\,\supp(\omega)}\sum_{\substack{|\underline{m}|\ll \hat{V}\\|\nabla \hat{F}_{\z}(\underline{y})-\underline{m}|\leq V}}|S(q;\underline{m})| \;d\underline{y}\\
    &\ll 1+q^{-n}\int_{\underline{y}\in P\,\supp(\omega)}\sum_{|\underline{m}-\underline{m}_0(\underline{y})|\leq V}|S(q;\underline{m})| \;d\underline{y}.
\end{align*}
where $\m_0(\underline{y}):=\nabla \hat{F}_{\z}(\underline{y})$. Hence, we have the following:
\Pb
\label{P: T_h(q,z): bound in terms of V sum}
Let $|\z|=\max\{|z_1|,|z_2|\}$. Then for any $q\in \N$,
\begin{equation*}
        |T(q,\z)|\ll 1+q^{-n}\max_{\underline{y}\in P\,\supp(\omega)}\Big{\{}\sum_{|\underline{m}-\underline{m}_0(\y)|\leq V}|S(q;\underline{m})| \Big{\}}.
\end{equation*}
for some $\m_0(\y)$, where
\begin{equation*}
    \label{eq: V defn}
    V:=1+q P^{-1+\ve}\max\{1,HP^2|\z|\}^{1/2}.
\end{equation*}
\Pe

Our attention now turns to finding a suitable bound for $|S(q;\m)|$. As is standard when dealing with exponential sum bounds, we will take advantage of the multiplicative property of $S(q;\m)$ and decompose $q$ into its square-free, square, and cube-full components so that we can use better bounds in the former two cases (in particular, we will make use of the $\ua$ sum to improve our bounds in the former cases). Indeed, we may use a Lemma of Hooley \cite[Lemma 3.2]{Hooley78} to get the following result. 

\begin{lemma}
\label{T63}
Let $\underline{a}\in\mathbb{Z}^2$ s.t. $(q,\underline{a})=1$, $q=rs$ where $(r,s)=1$ and $\underline{m}\in\mathbb{Z}^n$. Then
\begin{equation}
\label{eq:multiplicativity}
S(rs; \underline{m})=S(r; \xoverline{s}\underline{m})S(s;\xoverline{r}\underline{m}),
\end{equation}
where $r\xoverline{r}+s\xoverline{s}=1$.
\end{lemma}
The above lemma is proved using a very standard argument akin to \cite[Lemma 10]{Browning-Heath-Brown09} and \cite[Lemma 4.5]{Marmon_Vishe}, and therefore we will skip its proof here.
Our treatment of bounds for the quadratic exponential sums will vary depending on whether $q$ is square-free, a square or cube-full. Since the exponential sums satisfy the mutliplicativity relation \eqref{eq:multiplicativity}, it is natural to set $q=b_1b_2q_3$ where 
\begin{equation}
    \label{E61}
    b_1:=\prod_{p||q}p,\quad b_2:=\prod_{p^2||q}p^2,\quad q_3:=\prod_{\substack{p^e||q\\e>2}}p^e.
\end{equation}
Then by Lemma \ref{T63}, we have that 
\begin{equation}
    \label{E62}
    S(q; \underline{m})=S(b_1; c_1\underline{m})S(b_2; c_2\underline{m})S(q_3; c_3\underline{m}),
\end{equation}
for some constants $c_1,c_2,c_3$ such that $(b_1,c_1)=(b_2,c_2)=(q_3,c_3)=1$. Finding suitable bounds for the size of these three exponential sums will be the topic of the rest of this section. 

\subsection{Square-free Exponential Sums}
\label{S: sq. free exp. sums}
In this section, we will briefly consider the quadratic exponential sums  $S(b_1;\m)$ when $q=b_1$ is square-free. This case is extensively studied in \cite[Section 5]{Marmon_Vishe}, where bounds are obtained for exponential sums for a general system of polynomials $F$ and $G$. Using the multiplicativity of the exponential sum in \eqref{eq:multiplicativity}, it is enough to consider the sums $S(p,\m)$ where $p$ is a prime. We may rewrite
\begin{equation}\label{eq:Spm}
S(p,\m)=\Sigma_1-\Sigma_4,
\end{equation}
  where 
\begin{equation}
\label{eq:Sigma14def}
\Sigma_1:=\sum_{a_1=1}^p\sum_{a_2=1}^p \sum_{\underline{u} \bmod{q}} e_p(a_1F(\underline{u})+a_2G(\underline{u})+\underline{m}\cdot\underline{u})\quad\textrm{ and }\quad\Sigma_4:=\sum_{\underline{u} \bmod{q}} e_p(\m\cdot\u).
\end{equation}
Here the notation $\Sigma_1$ and $\Sigma_4$ is used to correspond to the corresponding sums in \cite[Section 5]{Marmon_Vishe}. Note that the argument in \cite[Section 5]{Marmon_Vishe} does not depend on the degree of the forms $F$ and $G$. In fact our exponential sums are more ``natural" than the ones which appear in \cite{Marmon_Vishe} and as a result, only sums $\Sigma_1$ and $\Sigma_4$ appear in our analysis. We may now use the results in \cite[Section 5]{Marmon_Vishe} directly here as they do indeed bound the sums $\Sigma_1$ and $\Sigma_4$ as well, but only in the case where $F$ and $G$ intersect properly over $\overline{\FF}_p$. When $n\geq 2$, we may use \cite[Prop 5.2, Lemma 5.4]{Marmon_Vishe} to get
\Pb
\label{P:sq free:n>1}
Let $F,G\in \Z[x_1,\cdots,x_n]$ be quadratic polynomials such that\\
$s_{\infty}(\Fnull,\Gnull)=-1$. Let $b_1$ be a square-free number where 
\[(b_1, \cont(\Fnull))=(b_1, \cont(\Gnull))=1,\]
If $n>1$, then there exists some $\Phi_{F,G}=\Phi\in \Z[x_1,\cdots,x_n]$ such that
\[S(b_1,\un{m})\ll_n b_1^{1+n/2+\ve}D(b_1)(b_1,\Phi(\m))^{1/2}\]
for every $\un{m}\in\Z^n$. Furthermore $\Phi$ has the following properties:
\begin{enumerate}
    \item $\Phi$ is homogeneous.
    \item $\deg(\Phi)\ll_n 1$.
    \item $\log ||\Phi||\ll_n \log ||F|| + \log ||G||$.
    \item $\cont(\Phi)=1$.
\end{enumerate}
\Pe
\begin{proof}
To begin, since $s_{\infty}(F,G)=-1$, we may use an explicit description of the dual variety as seen by a $\QQ$-version of \cite[Lemma 4.2]{Vishe19} to see that the polynomial defining the dual variety of the intersection variety of $F,G$ satisfies the four conditions for $\Phi$ in the statement of this proposition. Hence, we will may let $\Phi$ be this polynomial. This allows us to improve the first assertion of \cite[Proposition 5.2]{Marmon_Vishe}: Indeed, if we let $\delta_p(\uv):=s_p(F,G,L_{\uv})$, where $L_{\uv}$ is the hyperplane defined by $\uv$, then $\delta_p(\uv)\leq s_p(F,G)$ whenever $p\nmid \Phi(\uv)$. We automatically get this since 
\[s_p(F,G,L_{\uv})\leq s_p(F,G)\]
for every $\uv$ not on the dual variety of $F,G$ (over $\overline{\FF}_p$), and we must have $p\mid \Phi(\uv)$ when $\uv$ is on the dual variety by our choice of $\Phi$. Hence, in the case when $F,G$ intersect properly over $\overline{\FF}_p$, \cite[Lemma 5.4]{Marmon_Vishe} (and our improvement to \cite[Proposition 5.2]{Marmon_Vishe}) hands us
\begin{equation}
    \label{eq: sq free: proper intersections bound}
    S(p,\un{m})\ll_n p^{1+n/2+\ve}p^{(s_p(F,G)+1)/2}(p,\Phi(\m))^{1/2}=p^{1+n/2+\ve}D(p)^{1/2}(p,\Phi(\m))^{1/2}.
\end{equation}
A major difference here with \cite{Marmon_Vishe} is that both of our forms $F$, $G$ vary as $\h$ varies, and this forces us to consider the case when $F$ and $G$ intersect improperly in greater detail. We will firstly show that 
\[|\Sigma_1|\ll p^2D(p)\]
in this case. To do this, we start by noting that
\begin{align}
    \label{eq: square free: sigma_1 rewritten}
    |\Sigma_1|&=\big{|}p^2\sum_{\substack{\x\bmod{p}\\F(\x)\equiv G(\x)\equiv 0 \bmod{p}}} e_p(\m\cdot\x)\Big{|}\nonumber\\
              &\leq p^2 \#\{\x\in \FF_p^n \: : \: F(\x)=G(\x)=0\}.
\end{align}
In the case where $n>1$ and $F,G$ intersect improperly over $\overline{\FF}_p$, there are three cases to consider: $F\equiv 0$ (or $G\equiv 0)$, $F\equiv \lambda G$ for some $\lambda\in \overline{\FF}$, and $F\equiv L_1L_2$, $G\equiv L_1L_3$ for some hyperplanes $L_i$ such that $L_i\not \equiv 0$, $L_2\not \equiv \lambda L_3$. For the first two cases, we note that by definition of $\Sing_p(F,G)$ (see \eqref{eq:singlocus}), we have
\[\Sing_p(F,G)=\{\x\in \PP_{\overline{\FF}_p}^{n-1} \: : \: F(\x)=G(\x)=0\}\]
since $\rank(\nabla F(\x), \nabla G(\x))<2$ is automatically true when $F\equiv 0$, $G\equiv 0$, or $F\equiv \lambda G$. In particular, we see that $\{\x\in \FF_p^n \: : \: F(\x)=G(\x)=0\}$ is a subset of the affine singular locus of $F,G$ (over $\overline{\FF}_p$). Therefore
\[\#\{\x\in \FF_p^n \: : \: F(\x)=G(\x)=0\}\ll p^{\dim\Sing_p(F,G)+1}= p^{s_p(F,G)+1}=D(p).\]
Similarly when $F\equiv L_1L_2$, $G\equiv L_1L_3$,
\begin{align*}
\Sing_p(F,G)=\{L_1(\x)=0\}\cup \{L_2(\x)=L_3(\x)=0, \quad \rank(L_1(\x)\uv_2,\: L_1(\x) \uv_3)<2\}
\end{align*}
where $L_i(\x)=\uv_i\cdot \x$. However since $L_2\not \equiv \lambda L_3$, $\rank(L_1(\x)\uv_2,\: L_1(\x) \uv_3)<2$ if and only if $L_1(\x)=0$, which implies that
\[\Sing_p(F,G)=\{\x\in \PP_{\overline{\FF}_p}^{n-1} \: : \: L_1(\x)=0\}.\]
We therefore see that
\begin{align*}
\#\{\x\in \FF_p^n \: : \: F(\x)=G(\x)=0\}&= \#\{\x\in \overline{\FF}_p^n \: : \: L_1(\x)=0\} + \#\{\x\in \overline{\FF}_p^n \: : \: L_2(\x)=L_3(\x)=0\}\\
    &\leq 2\, \#\{\x\in \overline{\FF}_p^n \: : \: L_1(\x)=0\}\\
&\ll p^{s_p(F,G)+1}\\
&=D(p).
\end{align*}
Hence by \eqref{eq: square free: sigma_1 rewritten}, when $F$, $G$ intersect improperly, we have
\[|\Sigma_1|\ll p^2D(p)\leq p^{1+n/2}D(p),\]
provided that $n\geq 2$, as required. Therefore, we may conclude that for a general $p$ (irrespective of whether or not the intersection is proper)
\[S(p,\un{m})\leq C(n) p^{1+n/2+\ve}D(p)(p,\Phi(\m))^{1/2},\]
where $C$ is some constant. Finally by Lemma \ref{T63}, we have
\begin{align*}
    S(b_1,\un{m})&=\prod_{p\mid b_1} S(p,c_p\m)\\
    &\leq C(n)^{d(b_1)} b_1^{1+n/2+\ve}D(b_1)\prod_{p\mid b_1}(p,\Phi(c_p\m))^{1/2}\\
    &=C(n)^{d(b_1)} b_1^{1+n/2+\ve}D(b_1)(b_1,\Phi(\m))^{1/2},
\end{align*}
where $d(b_1):=\#\{p\mid b_1\}$ is the divisor function of $b_1$. We could replace $(p,\Phi(c_p\m))$ with $(b_1,\Phi(\m))$ because $\Phi$ is homogeneous and $(p,c_p)=1$. All that is left to do is show that $C(n)^{d(b_1)}$ does not contribute more than $O(P^{\ve})$. To see this, we note that $d(b_1)\ll \log(b_1)/\log\log(b_1)$. Hence there is some constant $d$ such that
\begin{align*}
    C(n)^{d(b_1)}&\leq C(n)^{d\log(b_1)/\log\log(b_1)}\ll b_1^{d\log(C(n))/\log\log(b_1)}\ll b_1^{\ve}
\end{align*}
provided that $b_1\gg_{\ve} 1$. We automatically have $d(b_1)\ll 1$ if $b_1\not\gg 1$, so we get $c^{d(b_1)}\ll 1\ll b_1^{\ve}$ in that case. Hence, we may conclude that Proposition \ref{P:sq free:n>1} is true. We will bound the $C(n)$ term in future lemmas by $b_1^{\epsilon}$ without further comment.
\end{proof}

We also must consider when $n=1$. In this case, it is sufficient for us to use a weaker bound than \cite[Lemma 5.5]{Marmon_Vishe}. We will show the following:

\Pb
\label{P:sq free:n=1}
Let $F,G\in \Z[x]$ be quadratic polynomials and let $b_1$ be a square-free integer. Then
\[S(b_1,m)\ll b_1^{2+\ve}D(b_1).\]
\Pe
\begin{proof}
The proof of Proposition \ref{P:sq free:n=1} is almost trivial. We start by applying Lemma \ref{T63} so that we may consider $S(p; cm)$ for some $p\nmid c$. We note that
\[|\Sigma_1|= p^2 \#\{x \mod p \: :\: F(x)\equiv G(x) \equiv 0 \mod p\}\ll p^2(p,\cont(F),\cont(G)),\]
and we trivially have $|\Sigma_4|\leq p$. Hence, by \eqref{eq:Ddef1} and noting that\\
$(p,\cont(F),\cont(G))\leq (p,\cont(F^{(0)}),\cont(G^{(0)}))$:
\[|S(p;cm)|\leq |\Sigma_1|+|\Sigma_4|\ll p^2D(p),\]
and so
\[|S(b_1;m)|\ll b_1^{2+\ve} D(b_1)\]
for any $m\in \Z$.
\end{proof}

\subsection{Square-full Bound}\label{sec:sqfullexp}
In this section, we will derive the bound which will be used when $q$ is square-full. When $q$ is square-full, we give up on saving $q$ over the $\ua$ sum, and instead start with the bound
\begin{equation}
    \label{eq: S(q,m): bound w/ S(a,q,m)}
    |S(q;\m)|\leq \starsum_{\ua}^q |S(\ua,q;\m)|
\end{equation}
where $F,G$ are quadric polynomials, and
\[S(\un{a}, q; \underline{m}):=\sum_{\underline{x} \bmod{q}} e_{q}(a_1F(\underline{x})+a_2G(\underline{x})+\underline{m}\cdot\underline{x}).\]
For a fixed value of $\ua$, the exponential sum $S(\un{a}, q; \underline{m})$ is a standard quadratic exponential sum with leading quadratic part defined by the matrix
\begin{equation}
\label{eq:Mdef}
M(\ua):=M:=a_1M_1+a_2M_2.
\end{equation}
A standard squaring argument as obtained in \cite[Lemma 2.5]{Vishe19} for example readily hands us a bound
\begin{equation}
\label{eq:sqfullb1}
|S(\un{a}, q; \underline{m})|\ll q^{n/2}\#\Null_q(M)^{1/2},
\end{equation}
where $\#\Null_q(M)$ denotes the number of solutions of the equation $M\x\equiv\vecnull\bmod{q}$ as defined in \eqref{eq:NullMdef}. To estimate this, we will resort to using a Smith normal form of the matrix $M$. The Smith normal form of $M$ hands us invertible integer matrices $S$ and $T$ be with determinant $\pm1$ such that 
\begin{equation}
\label{eq:SmithMdef}
SMT=\Smith(M)=\begin{pmatrix} \lambda_1 & 0 & 0 &\cdots&0\\
0 & \lambda_2 & 0 & \cdots&0 \\
0 & 0 & \ddots & & \vdots\\
\vdots &\vdots & &\ddots\\ 
0&0 &\cdots & & \lambda_n \end{pmatrix}\in M_n(\mathbb{Z}),
\end{equation}
where $\lambda_1\mid \lambda_2\mid\cdots\mid \lambda_n$. Since the forms $\Fnull$ and $\Gnull$ are assumed to be arbitrary for now, it is easy to conclude that
\begin{equation}
\label{eq:sqfullb2}
|S(\un{a}, q; \underline{m})|\ll q^{n/2}\prod_{i=1}^n \lambda_{q,i}^{1/2},
\end{equation}
where 
\begin{equation}
\label{eq:lambdaqidef}
\lambda_{q,i}:=(q,\lambda_i).
\end{equation}
\begin{remark}Recall that we aim to finally substitute $F=F_\h$ and $G=G_\h$. Note that the extra factor appearing on the right hand side of \eqref{eq:sqfullb2} is a generalisation of the factor $D(b_1)^{1/2}$ appearing in Proposition \ref{P:sq free:n>1}. This is a drawback of van der Corput differencing that although one starts with a {\em nice} pair of forms $F$ and $G$, one ends up with exponential sums of differenced polynomials $F_\h$ and $G_\h$, which can be highly singular modulo $q$. If $q=p^\ell$ for some prime $p$, if the singular locus $s_p$ as defined in \eqref{eq:mpdef} is large, then this gives restrictions on the vector $\h\bmod{p}$. When $\ell$ is small,
 the extra factors appearing can be compensated from the corresponding bounds on the $\h$ sum. However, in the case at hand, when $q=p^\ell$ for a large $\ell$, we can not rule out the possibility that for many $\h$, there may exist a large $q$ such that the factor $\prod_{i=1}^n \lqi^{1/2}$ is as large as $q^{n/2}$. This complication arises partly due to the simplicity of the quadratic exponential sums appearing. However, later we would need to average the sums over various $|\m-\m_0|\leq V$. We will aim to salvage some of this loss by gaining a congruence condition on $\m$ instead and saving from the sum over $\m$. This idea partly has already featured in Vishe's work \cite[Lemma 6.4]{Vishe19}. However, in \cite{Vishe19}, the authors are dealing with fixed $F$ and $G$, which is not the case here.
\end{remark}

Our main goal here is to prove the following result:

\Pb
\label{T600}
Let $\ua\in \Z^2$ and $q\in N$ be such that $(\ua,q)=1$, let $\m\in\Z^n$, and let $F,G$ be quadratic polynomials. Let
\begin{equation}
\label{eq:a_1F+a_2G}
(a_1F_1+a_2F_2)(\x)=\x^t M\x+\un{\mathfrak{b}}\cdot\x+\mathfrak{c}.
\end{equation}
(We use $\un{\mathfrak{b}}$ instead of $\b$ to avoid confusion since we have already defined $b_1$, $b_2$, $b_3$). Then 
    \[|S(\un{a}, q; \underline{m})|\leq 2^{n/2}q^{n/2}\#\Null_q(M)^{1/2}\Delta_{q}(\un{m}+\un{\mathfrak{b}})\]
where 
\begin{equation}
    \label{E6001}
    \Delta_q(\m):=\Delta_{T,q}(\un{m}):=\begin{cases} 1  \;\;\mbox{ if } \lambda_{q,i} \mid (T^t\un{m})_i\textrm{ for } 1\leq i \leq n\\ 0 \;\;\mbox{  else. }\end{cases}
\end{equation}
Here,  $T$ be the matrix appearing in the Smith normal form of $M$ in \eqref{eq:SmithMdef}, $\lambda_{q,i}$ be as in \eqref{eq:lambdaqidef} and given a vector $\v$, let $(\v)_i$ denote its $i$-th component.
\Pe
\begin{proof}
To estimate $|S(\ua,q;\m)|$, we begin by working with its square: 
\begin{align*}
|S(\un{a},q,\underline{m})|^2&= \sum_{\underline{x},\underline{y} \bmod{q}} e_{q}((a_1F_1+a_2F_2)(\x)+\underline{m}\cdot\underline{x})\overline{e_{q}((a_1F_1+a_2F_2)(\y)+\underline{m}\cdot\underline{y})}\\
                                   &= \sum_{\underline{x},\underline{y} \bmod{q}} e_{q}(\x^t M\x-\y^t M\y+(\underline{m}+\un{\mathfrak{b}})\cdot(\underline{x}-\underline{y})).
\end{align*}
We will now change order of summation by setting $\underline{x}=\underline{y}+\underline{z}$. Then
\begin{align*}
 |S(\un{a},q,\underline{m})|^2&= \sum_{\underline{y},\underline{z} \bmod{q}} e_{q}(\z^t M\z+(\underline{m}+\un{\mathfrak{b}})\cdot\z +2\y^t M \z)\\
                                               &=\sum_{\underline{z} \bmod{q}} e_{q}(\z^t M\z+\underline{m}'\cdot\underline{z})\sum_{\underline{y} \bmod{q}} e_{q}(\underline{y}\cdot 2M\z).\\
\end{align*}
where $\un{m}'=\un{m}+\un{\mathfrak{b}}$. Therefore
\begin{equation}
    \label{eq: s(a,q,m): delta M_F : 2}
    |S(\un{a},q,\underline{m})|^2=q^{n}\sum_{\underline{z} \bmod{q}} e_{q}(\z^t M \z+\underline{m}'\cdot\underline{z})\delta_{2M}(\z),
\end{equation}
where
\begin{equation}
    \label{eq: delta M_F: defn}
    \delta_{M}(\z):=\begin{cases} 1 \;\;\mbox{ if } M\underline{z}\equiv 0\mod q\\ 0 \;\;\mbox{  else. }\end{cases}
\end{equation}
The ``2" appearing in $\delta_{2M}(\z)$ gives rise to some minor technical difficulties in the case when $q$ is even. Therefore, we will start by considering the case when $q$ is odd first. 

\subsubsection{Case: $q$ odd} In this case, $\delta_{2M}(\z)=1$ if and only if $M\z\equiv \un{0} \mod{q}$, and so we may replace $\delta_{2M}(\z)$ in \eqref{eq: s(a,q,m): delta M_F : 2} by $\delta_M(\z)$. Furthermore, we note that $M\z\equiv \un{0} \mod{q}$ implies that $\z^t M\z\equiv \un{0} \mod{q}$. Hence \eqref{eq: s(a,q,m): delta M_F : 2} simplifies as:

\begin{equation}
    \label{eq: s(a,q,m): delta M_F simplified: 2.5}
    |S(\un{a},q,\underline{m})|^2=q^{n}\sum_{\underline{z} \bmod{q}} e_{q}(\underline{m}'\cdot\underline{z})\delta_{M}(\z).
\end{equation}

Now, $M$ has a Smith Normal form over $\mb{Z}$ as in \eqref{eq:SmithMdef}, $\Smith(M):=SMT$, for some matrices $S,T\in SL_n(\mb{Z})$. In particular, matrices $S$ and $T$ are invertible over $\mb{Z}/q\mb{Z}$, for any $q\in\mathbb{N}$. We will now rewrite our sum in terms of the $\Smith(M)$, Firstly, we note that 
\[\delta_M=\delta_{SM}.\]
Therefore, on using the substitution $\z\mapsto T^{-1}\z$, \eqref{eq: s(a,q,m): delta M_F : 2} becomes
\begin{equation}
    \label{eq: s(a,q,m): delta Smithed : 3}
    |S(\un{a},q,\underline{m})|^2=q^{n}\sum_{\underline{z} \bmod{q}} e_{q}(\underline{m}'\cdot T\underline{z})\delta_{SMT}(\z),
\end{equation}
since $\delta_{SM}(T\z)=\delta_{SMT}(\z)$ by \eqref{eq: delta M_F: defn}.
We will now work towards determining which $\z$ make $\delta_{SMT}(\z)$ non-zero. By definition, $\delta_{SMT}(\z)\neq 0$ if and only if
\[SMT\z\equiv \un{0} \mod q,\]
or equivalently
\[\z\in \Null_q(SMT):=\{\x\in \big{(}\mb{Z}/q\mb{Z}\big{)}^n\:\mid \: SMT\x \equiv \un{0} \mod q\}.\] 
Hence, we may simplify \eqref{eq: s(a,q,m): delta Smithed : 3} as follows:
\begin{align}
    |S(\un{a},q,\underline{m})|^2&=q^{n}\sum_{\z\in \Null_q(SMT)} e_{q}(\underline{m}'\cdot T\underline{z})\nonumber\\
                                \label{eq: s(a,q,m): z in Null(SMT) : 4}
                                 &=q^{n}\sum_{\z\in \Null_q(SMT)} e_{q}(\underline{z}\cdot T^t\underline{m}')
\end{align}
where $T^t$ is the transpose of $T$. This is true because 
\[\un{m}'\cdot T\z=(T \z)^t\m'=\z^t T^t\m'=\un{z}\cdot T^t\un{m}'.\]
We now turn our attention to structure of the $\Null_q(SMT)$. Since $S$ and $T$ are defined to be the unique matrices (up to units) such that $SMT=\Smith(M)$, it is quite easy to determine precisely when $\uz\in \Null_q(SMT)$. Therefore $SMT\z\equiv \un{0} \mod q$ if and only if 
\begin{equation}
    \label{eq: z in Null(SMT)}
    \frac{q}{\lambda_{q,i}}\:\Big{|}\: z_i
\end{equation}
for every $i\in\{1,\cdots, n\}$. Therefore
\begin{equation}
\label{eq: Null = lambda product}
    \#\Null_q(SMT)=\prod_{i=1}^n \lambda_{q,i}.
\end{equation}
Hence by \eqref{eq:lambdaqidef}, and \eqref{eq: s(a,q,m): z in Null(SMT) : 4}-\eqref{eq: z in Null(SMT)}, we have the following:

\begin{align}
    |S(\un{a},q,\underline{m})|^2&=q^{n}\prod_{i=1}^n\sum_{q/\lambda_{q,i}\mid z_i}e_q(z_i(T^t\un{m}')_i)=q^{n}\prod_{i=1}^n\sum_{x_i=1}^{\lambda_{q,i}}e_{\lambda_{q,i}}(x_i(T^t\un{m}')_i)
                                 \label{eq: s(a,q,m): prod lambda delta: 5}
                                 =q^{n}\prod_{i=1}^n \lambda_{q,i} \delta_{q,i}(\un{m}'),
\end{align}
where
\begin{equation}
    \label{E6000}
    \delta_{q,i}(\un{u}):=\begin{cases} 1  \;\;\mbox{ if } \lambda_{q,i} \mid (T^t\un{u})_i\\ 0 \;\;\mbox{  else }\end{cases},
\end{equation}
and $(\un{v})_i$ is the i-th component of vector $\un{v}$. Therefore, by \eqref{eq: Null = lambda product} and \eqref{eq: s(a,q,m): prod lambda delta: 5}:
\[|S(\un{a},q,\underline{m})|^2=q^{n}\#\Null_q(SMT)\prod_{i=1}^n\delta_{q,i}(\un{m}').\]
Finally it is easy to check that
\[\#\Null_q(SMT)=\#\Null_q(M)\]
since $S$ and $T$ are both invertible over $\mb{Z}/q\mb{Z}$ and therefore in this case we establish:
  \[|S(\un{a}, q; \underline{m})|= q^{n/2}\#\Null_q(M)^{1/2}\Delta_{q}(\un{m}+\un{\mathfrak{b}}),\]
 which clearly suffices.
\subsubsection{Case: $q$ even} We now turn to the case where $q$ is even. In this case, the above argument needs to be modified due to not being able to directly replace the condition $\delta_{2M}(\z)$ with $\delta_M(\z)$ in \eqref{eq: s(a,q,m): delta M_F : 2}. Instead we note that $\delta_{2M}(\z)\neq 0$ if and only if $M\z\equiv \un{0} \mod q/2$. In particular, there must be some $\uc\in \{0,1\}^n$ such that 
\[M\z \equiv \frac{q}{2}\uc \mod{q}.\]
Therefore, if we let 
\[N_{\uc,q}(M):=\{\x \mod q \: : \: M\x \equiv \frac{q}{2}\uc \mod{q}\},\]
then $\delta_{2M(\z)}\neq 0$ if and only if $\z\in N_{\uc,q}$ for some $\uc$. Hence, we may rewrite \eqref{eq: s(a,q,m): delta M_F : 2} as follows:
\begin{align}
    \label{eq: s(a,q,m) q even: 1}
    |S(\un{a}, q; \underline{m})|^2= q^n\sum_{\uc\in \{0,1\}^n} \sum_{\z\in N_{\uc,q}(M)} e_q(\z^t M \z + \m'\cdot \z).
\end{align}
We now wish to write $N_{\uc,q}$ in terms of $\Null_q(M)$ as this will enable us to use the arguments discussed in the odd case. To do this, we invoke Lemma \ref{lem:N_b=y_b+ Null} to see that either $N_{\uc,q}=\emptyset$ or there exists some $\y_{\uc}\in (\Z/q\Z)^n$ such that
\[N_{\uc,q}= \y_{\uc}+ \Null_q(M).\]
Hence 
\begin{align}
    \label{eq: s(a,q,m) q even: 2}
    |S(\un{a}, q; \underline{m})|^2&= q^n\sum_{\substack{\uc\in \{0,1\}^n\\ N_{\uc,q}(M)\neq \emptyset}} \sum_{\z\in \y_{\uc} + \Null_{q}(M)} e_q(\z^t M \z + \m'\cdot \z)\nonumber\\
    &=q^n\sum_{\substack{\uc\in \{0,1\}^n\\ N_{\uc,q}(M)\neq \emptyset}} \sum_{\z\in \Null_{q}(M)} e_q([\y_{\uc}+\z]^t M [\y_{\uc}+\z] + \m'\cdot [\y_{\uc}+\z])\nonumber\\
    &= q^n\sum_{\substack{\uc\in \{0,1\}^n\\ N_{\uc,q}(M)\neq \emptyset}} e_q(\y_{\uc}^t M \y_{\uc} + \m'\cdot \y_{\uc}) \sum_{\z\in \Null_{q}(M)} e_q((\z+2\y_{\uc})^t M \z + \m'\cdot \z)\nonumber\\
    &\leq q^n\sum_{\uc\in \{0,1\}^n} \Big{|} \sum_{\z\in \Null_{q}(M)} e_q((\z+2\y_{\uc})^t M \z + \m'\cdot \z)\Big{|}.
\end{align}
Finally, we note that $M\z\equiv \un{0} \mod q$ since $\z\in \Null_q(M)$, and so by \eqref{eq: s(a,q,m) q even: 2}, we have the following:

\begin{align}
    |S(\un{a}, q; \underline{m})|^2 &\leq q^n\sum_{\uc\in \{0,1\}^n} \Big{|} \sum_{\z\in \Null_{q}(M)} e_q(\m'\cdot \z)\Big{|}\nonumber\\
    &=2^n q^n \Big{|} \sum_{\z  \bmod{q}} e_q(\m'\cdot \z) \delta_{M}(\z)\Big{|}.\nonumber
\end{align}
This is precisely \eqref{eq: s(a,q,m): delta M_F simplified: 2.5} with an extra factor of $2^n$ and some absolute value signs around the sum (which are irrelevant). We may therefore repeat the arguments in the $q$ odd case which follow from \eqref{eq: s(a,q,m): delta M_F simplified: 2.5}  to establish Proposition \ref{T600}.
\end{proof}

\subsubsection{Special Case: $n=1$.}
We will now briefly consider the case when $n=1$, as we will need to deal with this case separately later. The arguments used above are still valid in this case, but the bound that we get is simpler due to the matrix, $M$, becoming an integer. In particular, Proposition \ref{T600} becomes

\Pb
\label{P:n=1 exponential sum bound: No Kloosterman}
Let $\ua\in \Z^2$ and $q\in N$ be such that $(\ua,q)=1$, let $m\in\Z$, and let $F,G\in \Z[x]$ be quadratic polynomials. Let
\begin{equation}
\label{eq: n=1 case: a_1F=a_2G}
(a_1F_1+a_2F_2)(x)=Mx^2+bx+c.
\end{equation}
 Then 
    \[|S(\un{a}, q; \underline{m})|\leq 2^{1/2}q^{1/2}(q,M)^{1/2}\Delta_{q}'(m+b)\]
where 
\begin{equation}
    \label{eq: n=1 case: Delta' definition}
    \Delta_{q}'(m):=\begin{cases} 1  \;\;\mbox{ if } (q,M) \mid m\\ 0 \;\;\mbox{  else. }\end{cases}
\end{equation}
\Pe

We will use Propositions \ref{T600} and \ref{P:n=1 exponential sum bound: No Kloosterman} directly in our future treatment of the cube-full part of $S(q_3,\m)$ (see \eqref{E62}) in order to get additional saving over the $\m$ sum. For the \ti{perfect square} part -- $b_2$ -- however, we will derive a slightly weaker bound from this which will be used to get saving over the $\h$ sum later on in the argument. 

\subsection{Cube-free Square Exponential Sums}
\label{S621}

In this section, we will assume that $q=b_2$, or equivalently that $q$ is a cube-free square. In this case, we will give up on the potential saving we could attain via the $\m$ sum from the $\Delta_q(\m')$ term in Proposition \ref{T600}, and bound $\#\Null_q(M(\ua))^{1/2}$ in terms of the singular locus of $F,G$, where $M(\ua)$ is defined as in \eqref{eq:a_1F+a_2G}. In this special case, we will need to obtain a pointwise saving over the $\ua$ sum in order for our bound to be useful. We will start with the case when $n\geq 2$. Upon letting $b_2=c^2$, and by Proposition \ref{T600}, Lemmas \ref{T1} - \ref{T20}, and \eqref{eq: S(q,m): bound w/ S(a,q,m)} we have

\begin{align}
    \label{eq: cube-free square case n>1: 1}
    |S(b_2,\m)|&\leq \starsum_{\ua}^{b_2}|S(\un{a}, b_2; \underline{m})|\leq b_2^{n/2}\starsum_{\ua}^{b_2}\#\Null_{c^2}(M(\ua))^{1/2}\nonumber\leq b_2^{n/2}\starsum_{\ua}^{b_2}\#\Null_{c}(M(\ua))\nonumber\\
                                                       &\ll b_2^{2+n/2}c^{s_p+1}=b_2^{2+n/2}\prod_{\substack{p^2\mid q\\ p \textrm{ prime }}} p^{s_p+1}
                                             =b_2^{2+n/2}D(b_2).
\end{align}
When $n=1$, we have $M(\ua)=a_1d_F+a_2d_G$ for some constants $d_F,d_G$. the same type of argument applies. By Proposition \ref{P:n=1 exponential sum bound: No Kloosterman},
\begin{align}
    \label{eq: cube-free square case n=1: 1}
    |S(p^2,\m)|&\leq p\starsum_{\ua\bmod{p^2}}(p^2,M(\ua))^{1/2}\leq p\starsum_{\ua\bmod{p^2}}(p,a_1d_F+a_2d_G)\nonumber\\&
    =  p \Bo\: \starsum_{\substack{\ua\bmod{p^2}\\p | a_1d_F+a_2d_G}} p + \starsum_{\substack{\ua\bmod{p^2}\\p \nmid a_1d_F+a_2d_G}} 1 \Bc\nonumber\\
               &\leq \begin{cases} 2p^{5} \mbox{ if } (d_F, d_G, p)=1 \\ p^6 \mbox{ else. }\end{cases}
\end{align}
Hence, upon recalling \eqref{eq:Ddef1}, we may bound \eqref{eq: cube-free square case n=1: 1} by 
\begin{align}
    \label{eq: cube-free square case n=1: 2}
    |S(p^2,\m)|\ll p^{5}D(p)\nonumber.
\end{align}
We may then use the multiplicativity relation in Lemma \ref{T63} to get 
\[|S(b_2,\m)|\ll b_2^{2+1/2+\ve}D(b_2).\]
Combining this with \eqref{eq: cube-free square case n>1: 1} gives us the following:
\Pb
\label{T602}
Let $b_2\in \N$ be a cube-free square. Then
\[S(b_2,\m)\ll b_2^{2+n/2+\ve}D(b_2).\]
\Pe

\section{Quadratic Exponential Sums: Finalisation}
\label{sec:quadfinal}
In this section, we will combine all of the bounds we have found in Section \ref{Sec: Exponential Sums: Initial Bounds} to reach our final estimate for $T(q,\z)$. Recall that Proposition \ref{P: T_h(q,z): bound in terms of V sum} hands us
\begin{align}
    \label{E7000}
    |T(q,\z)|\ll 1+q^{-n}\max_{\underline{y}\in P\,\supp(\omega)}\Big{\{}\sum_{|\underline{m}-\underline{m}_0(\underline{y})|\leq V}|S(q;\underline{m})|\Big{\}}.
\end{align}
In the last section, we focused on getting bounds for individual exponential sums $|S(q;\underline{m})|$. We begin by considering averages of exponential sums. Throughout, let $\m_0$ be an arbitrary but fixed vector in $\Z^n$ and let $\mfb(\ua)=\mfb$ be defined as in \eqref{eq:a_1F+a_2G}. For $n\geq 2$: By Lemma \ref{T63} and Propositions \ref{P:sq free:n>1}, \ref{T600}, and \ref{T602}, there are some constants $c_1,c_2,c_3$ such that $(b_1,c_1)=(b_2,c_2)=(q_3,c_3)=1$, and
\begin{align}
    \sum_{|\underline{m}-\underline{m}_0|\leq V}|S(q;\underline{m})|&\leq \sum_{|\underline{m}-\underline{m}_0|\leq V}|S(b_1;c_1\underline{m})|\cdot|S(b_2;c_1\underline{m})|\cdot|S(q_3;c_1\underline{m})|\nonumber\\
    &\ll q^{n/2+\ve} b_1b_2^2 D(b_1b_2)\sum_{|\underline{m}-\underline{m}_0(\underline{m}_0)|\leq V} (\Phi(c_1\un{m}),b_1)^{1/2}\times\nonumber\\
    &\quad\quad\quad\quad\quad\quad\;\;\starsum_{\ua}^{q_3}\#\Null_{q_3}(a_1M_1+a_2M_2)^{1/2}\Delta_{T,q_3}(c_3\m+\mfb) \nonumber \\ 
    \label{E701}
    &:= q^{n/2+\ve} b_1b_2^2 D(b_1b_2)\starsum_{\ua}^{q_3}\#\Null_{q_3}(M(\ua))^{1/2}B(b_1,q_3, V; \m_0),
\end{align}
where $M(\ua)$ be as in \eqref{eq:Mdef}
and
\begin{align}
    \label{eq: B(b_1,q_3,V,y) defn}
    B(b_1,q_3, V; \m_0):&= \sum_{|\underline{m}-\underline{m}_0|\leq V} (\Phi(\un{m}),b_1)^{1/2}\cdot \Delta_{T,q_3}(\un{m}+\mfb'),
\end{align}
where $\mfb'\equiv c_3^{-1}\mfb \mod q_3$. We used $(\Phi(\un{m}),b_1)$ instead of $(\Phi(c_1\un{m}),b_1)$ in the definition of $B(b_1,q_3,V;\m_0)$ because $\Phi$ is homogeneous and $(b_1,c_1)=1$. Likewise, by inspecting the definition of $\Delta$, we can use $\Delta_{T,q_3}(\un{m}+\mfb')$ in the definition of $B(b_1,q_3,V;\m_0)$ instead of $\Delta_{T,q_3}(c_3\un{m}+\mfb)$  since we can ``divide through" by $c_3$, as $(c_3,q_3)=1$ (in particular $(c_3,\lambda)=1$ for any divisor, $\lambda$, of $q_3$).

The first and most difficult task for this section is to bound $B(b_1,q_3, V; \m_0)$. This will be quite a delicate task since we need to save over the $\m$ sum in two different ways, simultaneously. The following Lemma will provide our main estimate for this sum:

\Lb
\label{L: B(b_1,q_3, V,y): main bound}
Let $b_1,q_3, V\in \N$ and $\m_0\in \Z^n$. Furthermore, let $c$ and $q_3$ be defined as follows: 
\begin{align}
    \label{eq: exp sum finalisation: tilde(q)_3 defn}
    \hat{q}_3:=\prod_{\substack{p^e||q_3\\ 2\nmid e}} p, \quad q_3=c^2\hat{q}_3
\end{align}
Then
\[B(b_1,q_3, V; \m_0) \ll b_1^{\ve}\Bo b_1^{1/2}c^{n/2}+ V^{n-1}b_1^{1/2}c^{1/2}+ V^n \Bc \#\Null_c(M(\ua))^{-1}.\]
\Le

\begin{proof}
We begin by noting that 
\[(T^t\x)_i\equiv 0 \mod (q_3,\lambda_i) \;\implies\; (T^t\x)_i\equiv 0 \mod (c,\lambda_i),\]
and so by the definition of $\Delta_{T, q_3}$ \eqref{E6001} we clearly have that
\[\Delta_{T, q_3}(\x)=1 \;\implies\; \Delta_{T, c}(\x)=1,\]
for any $\x\in\Z^n$. Therefore -- since we are looking for an upper bound of \\$B(b_1,q_3, V; \m_0)$ -- we may replace $\Delta_{T, q_3}(\m+\mfb')$ in \eqref{eq: B(b_1,q_3,V,y) defn} with $\Delta_{T, c}(\m+\mfb')$. Furthermore, since all elements of our sum are non-negative, we may extend the sum in \eqref{eq: B(b_1,q_3,V,y) defn} if we wish. In particular, the following bound must be true:
\begin{align}
    \label{E7005}
    B(b_1,q_3, V; \m_0)&\leq \sum_{|\underline{m}-\underline{m}_0|\leq \hat{V}} (\Phi(\un{m}),b_1)^{1/2}\cdot \Delta_{T,c}(\un{m}+\mfb'),
\end{align}
where 
\begin{align}
    \label{eq: exp sum finalisation: hat(V) defn}
    \hat{V}:=\max\{V,c\}.
\end{align}
We have extended the sum up to $\hat{V}$ so that we can consider complete sums modulo $c$, as this will make it easier to acquire saving from $\Delta_{T,c}$ later. To this end, let $\m:=\m_0+\m_1+c\m_2$, where $\m_1\in(\Z/c\Z)^n$ and $|\m_2|\leq \hat{V}/c$.
Applying this decomposition on the right-hand side of \eqref{E7005} gives
\begin{align}
    \label{E7007}
    B(b_1,q_3, V; \m_0)&\leq \sum_{\m_1 \bmod{c}}\:\sum_{|\m_2|\leq \hat{V}/c} (\Phi(\m_0+\m_1+c\m_2),b_1)^{1/2}\Delta_{T,c}(\m_0+\m_1 +c\m_2+\mfb')\nonumber\\
    &=\sum_{\m_1 \bmod{c}}\Delta_{T,c}(\m_0+\m_1+\mfb')\sum_{\m_2\in U(\m_1)} (\Phi(\m_0+\m_1+c\m_2),b_1)^{1/2}.
\end{align}
The upshot of reordering our sum in this way is that we have managed to separate $\Delta_{T,c}(\m_0+\m_1+\mfb')$ and $(\Phi(\m_0+\m_1+c\m_2),b_1)^{1/2}$. In particular, we can treat $\m_1$ as fixed for now, and since $\m_0$ and $c$ are also fixed, we may focus on acquiring saving in the $\m_2$ sum via $(\Phi_{c,\m_1}(\m_2),b_1)^{1/2}$, where
\begin{align*}
    \Phi_{c,\m_0,\m_1}(\m_2):=\Phi(\m_0+\m_1+c\m_2).
\end{align*}

We observe that $(\Phi_{c,\m_0,\m_1}(\m_2),b_1)$ must be equal to some divisor of $b_1$, so we will decompose the $\m_2$ sum as follows:
\begin{align}
    \label{E7008}
    \sum_{|\m_2|\leq \hat{V}/c} (\Phi(\m_0+\m_1+c\m_2)&,b_1)^{1/2}=\sum_{d\mid b_1} d^{1/2} \#\{|\x|\leq \hat{V}/c : \Phi(\m_0+\m_1+c\x)\equiv 0 \mod d\}.
\end{align}
We now aim to use \cite[Lemma 4]{Browning-Heath-Brown09} to bound the right hand side. Since $\Phi$ is homogeneous with $\cont(\Phi)=1$ by Proposition \ref{P:sq free:n>1}, and since $c$ and $d$ are co-prime, we have that
\[(\cont(\Phi_{c,\m_0,\m_1}),d)\leq (\cont(\Phi_{c,\m_0,\m_1}^{(0)}),d)=(c^{\deg(\Phi)} \cont(\Phi), d)=(c^{\deg(\Phi)},d)=1.\]
Hence for every prime $p$ dividing $d$, $\Phi(\m_0+\m_1+c\x)$ is a non-trivial polynomial and therefore the corresponding variety is of dimension $n-1$. Therefore, we may now use \cite[Lemma 4]{Browning-Heath-Brown09} to conclude that
\begin{align*}
   \#\{|\x|\leq \hat{V}/c: \Phi(\m_0+\m_1+c\x)\equiv 0 \mod d\}&\ll 1+\Bo \frac{V}{c}\Bc^{n-1}+ \Bo \frac{V}{c}\Bc^n d^{-1}.
\end{align*} 
Substituting this back into \eqref{E7008} gives the following:
\begin{align}
    \label{E7011}
  \sum_{|\m_2|\leq \hat{V}/c} (\Phi(\m_0+\m_1+c\m_2),b_1)^{1/2}&\ll \sum_{d\mid b_1} d^{1/2}+\Bo \frac{V}{c}\Bc^{n-1}d^{1/2}+ \Bo \frac{V}{c}\Bc^n d^{-1/2}\nonumber\\
    &\ll b_1^{\ve}\Bo b_1^{1/2}+ \Bo \frac{V}{c}\Bc^{n-1}b_1^{1/2}+ \Bo \frac{V}{c}\Bc^n \Bc.
\end{align}
This in turn will enable us to find a suitable bound for $B(b_1,q_2,V;\m_0)$. By \eqref{E7007} and \eqref{E7011}, we have 
\begin{align}
    \label{E7012}
    B(b_1,q_3, V; \m_0)&\ll b_1^{\ve}\Bo b_1^{1/2}+ \Bo \frac{V}{c}\Bc^{n-1}b_1^{1/2}+ \Bo \frac{V}{c}\Bc^n \Bc \sum_{\x\bmod{c}}\Delta_{T,c}(\x+\m_0+\mfb').
\end{align}
In order to find the bound we desire for $B(b_1,q_3, V; \m_0)$, we will need to turn our attention to the sum of type
\[\sum_{\x\bmod{c}}\Delta_{T,c}(\x+\un{l}),\]
for some fixed $\un{l}\in\Z^n$. Our bound here will be independent of the choice of the vector $\un{l}$. This sum is much easier to handle since we have a complete sum at hand. It is easy to check from the definition of $\Delta_{T,c}(\x+\un{l})$ (and the fact that $\det(T^t)=1$) that
\begin{align*}
    \sum_{\x\bmod{c}}\Delta_{T,c}(\x+\un{l})&=\#\{\x \bmod{c} \: : \: (T^t\x)_i\equiv -(\un{l})_i \bmod{\lambda_{c,i}}, \; i\in\{1,\cdots, n\}\}\\
    &\leq \#\{\x \bmod{c} \: : \: (T^t\x)_i\equiv 0 \bmod{\lambda_{c,i}}, \; i\in\{1,\cdots, n\}\}\\
    &= \#\{\x \bmod{c} \: : \: x_i\equiv 0 \bmod{\lambda_{c,i}}, \; i\in\{1,\cdots, n\}\}\\
    &=\frac{c^n}{\prod_i \lambda_{c,i}}=c^n \#\Null_c(M(\ua))^{-1}.
\end{align*}
Therefore, by \eqref{E7012}, we have
\begin{align}
    \label{E7014}
    B(b_1,q_3, V; \m_0)
    &\leq b_1^{\ve}\Bo b_1^{1/2}c^n+ V^{n-1}b_1^{1/2}c+ V^n \Bc \#\Null_c(M(\ua))^{-1},
\end{align}
as required.
\end{proof}
 We are now ready to obtain a final bound for $\sum_{|\m-\m_0|\leq V} |S(q;\m)|$. Before substituting \eqref{E7014} back into \eqref{E701}, we will perform some simplifications. Firstly, we note that by \eqref{eq: exp sum finalisation: tilde(q)_3 defn} and Lemma \ref{T1}, we have
\begin{align}
    \label{E7015}
    \#\Null_{q_3}(M(\ua))^{1/2}\leq \#\Null_{c}(M(\ua))\#\Null_{\hat{q}_3}(M(\ua))^{1/2}.
\end{align}
Furthermore, by Proposition \ref{T20}, we have
\begin{align}
    \label{E7016}
    \starsum_{\ua}^{q_3}\#\Null_{\hat{q}_3}(M(\ua))^{1/2}&\leq \starsum_{\ua}^{q_3}\#\Null_{\hat{q}_3}(M(\ua)) \ll q_3^{2+\ve} \prod_{p_i | \hat{q}_3} p_i^{s_{p_i}(F,G)+1}
    =q_3^{2+\ve} D(\hat{q}_3).
\end{align}
Finally, by combining \eqref{E7014}-\eqref{E7016} with \eqref{E701}, we arrive at the following bound:

\Lb
\label{lem:V sum bound n>1}
For every $q\in \N$, if $n>1$, then 
\begin{align}
    \sum_{|\underline{m}-\underline{m}_0|\leq V}|S(q;\underline{m})|&\ll q^{1+n/2+\ve} b_2q_3 D(b_1b_2\hat{q}_3)\Bo b_1^{1/2}c^{n}+ V^{n-1}b_1^{1/2}c+ V^n \Bc\nonumber,
\end{align}
where $q_3=c^2\hat{q}_3$ as defined in the statement of Lemma \ref{L: B(b_1,q_3, V,y): main bound}.
\Le

Recall that our ultimate goal was to find a suitable bound for $|T(q,\z)|$. Upon noting that the above treatment of $ \sum_{|\underline{m}-\underline{m}_0|\leq V}|S(q;\underline{m})|$ works for any value of $\y\in P\,\supp(\omega)$ we may now substitute the bound in Lemma \ref{lem:V sum bound n>1} into \eqref{E7000} to get the following bound for $T(q,\z)$:
\begin{align*}
    |T(q,\z)|&\ll 1+ P^n q^{1-n/2+\ve} b_1^{1/2}b_2q_3 D(b_1b_2\hat{q}_3)\Bo V^n b_1^{-1/2} + V^{n-1}c + c^n\Bc.\\
\end{align*}
If $q$ is sufficiently small ($q<P^2$ say), then the right hand term dominates over $1$ for every $n\geq 1$. Therefore, we finally reach the following bound for $|T(q,\z)|$:
\[|T(q,\z)|\ll P^n q^{1-n/2+\ve} b_1^{1/2}b_2q_3 D(b_1b_2\hat{q}_3)\Bo V^n b_1^{-1/2} + V^{n-1}c + c^n\Bc,\]
where $q_3=c^2\hat{q}_3$ as defined in Lemma \ref{L: B(b_1,q_3, V,y): main bound}. Note that if we use a weaker bound $c\leq b_3^{1/3}q_4^{1/2}$ and $D(\hat{q}_3)\leq D(q_3)$, and use the quality $q_3=b_3q_4$, the above bound becomes:
\Pb
\label{proposition:Poissonquadfinal} 
For every $q=b_1b_2q_3<P^2$, $\z$, and every $\ve>0$, if $n>1$, we have 
\[|T(q,\z)|\ll P^n q^{1-n/2+\ve} b_1^{1/2}b_2q_3 D(q)\Bo V^n b_1^{-1/2} + V^{n-1}b_3^{1/3}q_4^{1/2} + b_3^{n/3}q_4^{n/2}\Bc,\]
where $n$ is the number of variables of $F, G$, $b_3$ is the 4th power-free cube part of $q$, $q_i$ is the i-th power-full part of $q$.  
\Pe
The bound for the $n=1$ case is much simpler to derive than in the $n>1$ case. By Lemma \ref{T63} and Propositions \ref{P:sq free:n=1}, \ref{P:n=1 exponential sum bound: No Kloosterman}, and \ref{T602}, we have
\begin{align}
\label{eq: Exponential sums finalisation, n=1: 1}
    \sum_{|m-m_0|\leq V}|S(q;m)| &\ll q^{1/2+\ve}b_1^{3/2} b_2^2D(b_1b_2) \starsum_{\ua}^{q_3}(q_3,M(\ua))^{1/2}\sum_{|m-m_0|\leq V} \Delta'_{q_3}(c_3m+\mathfrak{b})\nonumber\\
    &\leq q^{1/2+\ve}b_1^{3/2} b_2^2D(b_1b_2)\starsum_{\ua}^{q_3}(q_3,M(\ua))^{1/2} \sum_{|m-m_0|\leq \max\{V,q_3\}} \Delta'_{q_3}(m+\mathfrak{b}')\nonumber\\
    &= q^{1/2+\ve}b_1^{3/2} b_2^2D(b_1b_2)\starsum_{\ua}^{q_3}(q_3,M(\ua))^{1/2}\Bo 1 + \frac{V}{(q,M(\ua))} \Bc\nonumber\\
    &\leq q^{1/2+\ve}b_1^{3/2} b_2^2D(b_1b_2)\starsum_{\ua}^{q_3}\big{(}(q_3,M(\ua))^{1/2} + V \big{)}.
\end{align}
We trivially have $\sum_{\ua} V\leq q_3^2 V$. As for the other part of the sum, upon recalling that $q_3=\hat{q}_3c^2$, we have
\begin{align*}
    \starsum_{\ua}^{q_3}(q_3,M(\ua))^{1/2}&\leq c \starsum_{\ua}^{q_3}(\hat{q}_3,M(\ua))^{1/2}\leq c^5 \starsum_{\ua}^{\hat{q}_3}(\hat{q}_3,M(\ua))^{1/2}\ll q_3^2 c D(\hat{q}_3)\leq q_3^2 c D(q_3),
\end{align*}
by the same argument as the proof of Proposition \ref{T602}. Combining this with \eqref{eq: Exponential sums finalisation, n=1: 1} gives the following result.
\Lb
Let $q=b_1b_2q_3\in \N$, $m\in \Z$, $q_3:=c^2 \hat{q}_3$ be defined as in Lemma \ref{L: B(b_1,q_3, V,y): main bound}. Then for every $\ve>0$,
\begin{align*}
     \sum_{|m-m_0|\leq V}|S(q;m)| &\ll q^{2+\ve} (b_2b_3q_4)^{1/2} D(q) (V+c).
\end{align*}
\Le
Finally, upon recalling that $q_3=b_3q_4$, $c\leq b_3^{1/3}q_3^{1/2}$, we may combine this lemma with \eqref{E7000} to get our final bound for $|T(q,\z)|$ in the $n=1$ case:
\Pb
\label{proposition:n=1Poissonquadfinal} 
For every $q<P^2$, $\z$, and every $\ve>0$, if $n=1$, we have
\[|T(q,\z)|\ll P q^{1+\ve} (b_2q_3)^{1/2} D(q) (V + b_3^{1/3}q_4^{1/2}),\]
where $n$ is the number of variables of $F, G$, $b_3$ is the 4th power-free cube part of $q$, and $q_4$ is the 4th power-full part of $q$.  
\Pe
\section{Finalisation of the Poisson bound}
\label{hsum}
In this section, we will finalise our main bounds coming from Poisson summation. For a fixed value of $t$, Lemmas \ref{T46} and \ref{T45} allow us to consider bounding the sum
\[\sup_{t\ll|\uz|\ll t}\sum_{\underline{h}\ll H}\,\, |T_{\underline{h}}(q,\uz)|,\]
where 
\begin{equation*}
    T_{\underline{h}}(q,\underline{z}):= \starsum_{\ua\bmod{q}}\sum_{\underline{x}\in\mathbb{Z}^n} \omega_{\underline{h}}(\underline{x}/P)e((a_1/q+z_1)F_{\h}(\underline{x})+(a_2/q+z_2)G_{\h}(\underline{x})),
\end{equation*}
is the quadratic exponential sum as defined in \eqref{eq:Thmaindef}. We may therefore apply our bounds for quadratic exponential sums in Propositions \ref{proposition:Poissonquadfinal} and \ref{proposition:n=1Poissonquadfinal} to estimate these. 

Now that $\underline{h}$ is allowed to vary, we will define 
\begin{align}
    \label{eq: Poisson Finalisation: m_p(h) defn}
    s_p(\underline{h}):=s_p(F_{\underline{h}}^{(0)},G_{\underline{h}}^{(0)}),
\end{align} 
where $F_{\underline{h}}^{(0)}$ and $G_{\underline{h}}^{(0)}$ denote the leading quadratic parts of $ F_\h$ and $G_\h$ respectively. We recall that $q=b_1b_2q_3$, where $q_3$ is the cube-full part of $q$, and $b_1,b_2$ are the square-free and cube-free square parts of $q$. Since we are fixing $q$ for now, $b_1$, $b_2$, and $q_3$ are also fixed. Recall that we may write $b_i=b_{i,0}b_{i,1}\cdots b_{i,n}$, $q_3=q_{3,0}q_{3,1}\cdots q_{3,n}$ where $b_{i,j}$, and $q_{3,j}$ now depend on $\underline{h}$ and are defined to be
\[b_{i,j}(\underline{h})=\prod_{\substack{p^i||b_i\\s_p(\underline{h})=j-1}} p^i,\quad q_{3,j}(\underline{h})=\prod_{\substack{p^e||q_3\\s_p(\underline{h})=j-1}} p^e.\]
We see that for any $q$ fixed, there are at most $O(q^{\ve})=O(P^{\ve})$ possible choices for 
\[\underline{c}=(b_{1,0},\cdots, b_{1,n},b_{2,0},\cdots, b_{2,n},q_{3,0},\cdots, q_{3,n})\]
since there are only at most $O(q^{\ve})$ partitions of $q$ into multiplicative factors. Therefore using the triangle inequality, we have that
\[\sum_{\underline{h}\ll H} |T_{\underline{h}}(q,z)|\leq P^{\ve}\max_{\underline{c}}\Big{\{}\sum_{\substack{\underline{h}\\\underline{c}(\underline{h})=\underline{c}}} |T_{\underline{h}}(q,z)|\Big{\}}=P^{\ve}\sum_{\substack{\underline{h}\\ \underline{c}(\underline{h})=\underline{c}'}} |T_{\underline{h}}(q,z)|\]
for some particular $\underline{c}'$, and $\underline{c}(\underline{h}):=(b_{1,0}(\underline{h}),\cdots, q_{3,n}(\underline{h}))$. We can then decompose this sum further by grouping $\underline{h}$'s with $s_{\infty}(\underline{h})=s$.
\begin{equation}
    \label{eq: Finalisation of Poisson: h sum bounded by H_s sum: 1}
    \implies \sum_{\underline{h}\ll H} |T_{\underline{h}}(q,z)|\leq P^{\ve}\sum_{s=-1}^{n-1}\sum_{\underline{h}\in\mathcal{H}_s} |T_{\underline{h}}(q,z)|,
\end{equation}
where 
\begin{equation}
    \label{eq: Finalisation of Poisson: H_s definition}
    \mathcal{H}_s:=\{\underline{h}\in\mathbb{Z}^n:\, \underline{h}\ll H,\; \underline{c}'(\underline{h})=\underline{c}',\; s_{\infty}(\underline{h})=s\}.
\end{equation}
Here, given $\nu$ either a prime, or $\infty$ we define
\begin{equation}
s_{\nu}(\underline{h})=s_{\nu}(F^{(0)}_{\underline{h}},G^{(0)}_{\underline{h}}).
\end{equation}
We now aim to estimate the size of $\mathcal{H}_s$. We start by noting that we must have that $\mathcal{H}_s=\emptyset$ unless $b_{1,i}=b_{2,i}=q_{3,i}=1$ for $i\leq s$. This is because $s_p(\underline{h})\geq s_{\infty}(\underline{h})$ for every $p$. To get a bound on $\#\mathcal{H}_s$ we will start by constructing a set which contains $\mathcal{H}_s$ that is easier to work with. Let 
\begin{align*}
    V_{\nu,i}&:=\{\underline{h}\in\mathbb{A}_{\overline{\FF}_{\nu}}^n\: |\: s_{\nu}(\underline{h})\geq i-1\}.
\end{align*}
Then, upon defining $[\h]_p$ to be the reduction modulo $p$ of a point $\h\in\Z^n$, we have
\begin{equation}
    \label{eq: Finalisation of Poisson: H_s subset defined by V_p,i}
    \mathcal{H}_s\subset \{\underline{h}\in V_{\infty,s+1}\cap \big{[}-H,H\big{]}^n| [\underline{h}]_p\in V_{p,i} \mbox{ for all } p|b_{1,i}b_{2,i}q_{3,i}\}.
\end{equation}
In order to bound this larger set, we will need the following lemma, which is analogous to \cite[Lemma 8.2]{Marmon_Vishe}.
\begin{lemma}
\label{T71}
Let $F$, $G$ be a pair of forms of degree $d_1$, $d_2$, and define $\sigma:=s_{\infty}(F,G)$. Then there is an absolute constant $C$ s.t.
\[\dim(V_{\nu,i})\leq \min\{n, n+\sigma+1-i\}\]
as long as $\nu=p>C$ or $\nu=\infty$.
\end{lemma}
\begin{proof}
We prove this result for any pair of forms instead of two cubics as it does not change the argument. Since
\begin{align*}
    s_{\nu}(F^{(0)}_{\underline{h}},G^{(0)}_{\underline{h}})=\dim(\{\underline{x}\in\mathbb{P}_{\overline{\FF}_{\nu}}^{n-1}\: |\:\underline{h}\cdot \nabla F^{(0)}(\underline{x})=\underline{h}\cdot \nabla G^{(0)}(\underline{x})=0, &\rank\begin{pmatrix}
\underline{h}\cdot \nabla^2 F^{(0)}(\underline{x})\\
\underline{h}\cdot \nabla^2 G^{(0)}(\underline{x})
\end{pmatrix}<2\}),
\end{align*}
we can use \cite[Lemma 3(ii)]{Marmon08} to conclude that $\dim(V_{\nu,i})\leq \min\{n,n+\sigma+1-i\}$, provided that $\nu=p\gg_{d_1,d_2} 1$. Therefore we only need to check $V_{\infty,i}$. We will use a slight modification to the argument used in \cite[Lemma 1]{Browning-Heath-Brown09} in order to show that $\dim (V_{\infty,i})\leq \min\{n,n+\sigma+1-i\}$: Let 
\[U(F,G)=U:=\{(\underline{x},\y)\in\mathbb{A}_{\Q}^{2n}\, |\,\underline{y}\cdot \nabla F(\underline{x})=\underline{y}\cdot \nabla G(\underline{x})=0, \,\rank\begin{pmatrix}
\underline{y}\cdot \nabla^2 F(\underline{x})\\
\underline{y}\cdot \nabla^2 G(\underline{x})
\end{pmatrix}<2\},\]
for $F,G$ homogeneous forms of degree 3, and let $D:=\{(\x,\y)\in\mathbb{A}_{\Q}^{2n} \: | \: \x=\y\}$. Then by the Affine Dimension Theorem, we have that
\begin{equation}
    \label{eq: dim U : affine dimension theorem}
    \dim(U)\leq \dim(U\cap D) - \dim(D) +2n=\dim(U\cap D) +n.
\end{equation}
Next, we note that 
\begin{align*}
    U\cap D&=\{\x\in\mathbb{A}_{\Q}^{n} \: | \:\underline{x}\cdot \nabla F(\underline{x})=\underline{x}\cdot \nabla G(\underline{x})=0, \;\rank\begin{pmatrix}
\nabla(\underline{x}\cdot \nabla F(\underline{x}))\\
\nabla(\underline{x}\cdot \nabla G(\underline{x}))
\end{pmatrix}<2\}\\
&=\{\x\in\mathbb{A}_{\Q}^{n} \: | \:F(\underline{x})=G(\underline{x})=0, \;\rank\begin{pmatrix}
\nabla F(\underline{x})\\
\nabla G(\underline{x})
\end{pmatrix}<2\},
\end{align*}
by Euler's identity. Hence, by \eqref{eq: dim U : affine dimension theorem}, we have
\[\dim(U\cap D)= \sigma+1,\]
and so 
\begin{equation}
    \label{eq: dim U : affine dimension theorem: 2}
    \dim(U)\leq n+\sigma+1.
\end{equation}
Finally we let $F=F^{(0)}$, $G=G^{(0)}$. If
\[\dim(V_{\infty, i})> n+\sigma+1-i,\] 
then, by definition we have that
\begin{align*}
    \dim ( \{(\underline{x},\h)\in\mathbb{A}_{\Q}^{2n}\: |\: s_{\infty}(F_{\h}^{(0)}, G_{\h}^{(0)})\geq i-1, \: \x\in s_{\infty}(F_{\h}^{(0)}, G_{\h}^{(0)})\})&> (n+\sigma+1-i) + i\\
    \\&=n+\sigma+1.
\end{align*}
It is easy to check that
\[\{(\underline{x},\h)\in\mathbb{A}_{\Q}^{2n}\: |\: s_{\infty}(F_{\h}^{(0)}, G_{\h}^{(0)})\geq i-1, \: \x\in s_{\infty}(F_{\h}^{(0)}, G_{\h}^{(0)})\}\subset U((F^{(0)}, G^{(0)})),\]
and so
\[\dim(U(F^{(0)}, G^{(0)}))>n+\sigma+1.\]
This contradicts \eqref{eq: dim U : affine dimension theorem: 2}. Hence $\dim(V_{\infty, i}) \leq n+\sigma+1-i$ as required.
\end{proof}
We can now use \eqref{eq: Finalisation of Poisson: H_s subset defined by V_p,i} and the argument found in \cite[Section 7]{Browning-Heath-Brown09} to get the following upper bound for $\#\mathcal{H}$:

\begin{equation}
    \label{E71}
    \#\mathcal{H}_s\ll q^{\ve} 
    \max_{s+1\leq\eta\leq n}\frac{H^{n-\eta}}{\prod_{i=\eta+1}^n(b_{1,i}b_{2,i}^{1/2}\tilde{q}_{3,i})^{i-\eta}},
\end{equation}
where
\[\tilde{q}_{3,i}:=\prod_{\substack{p|q_3\\s_p(\underline{h})=i-1}} p.\]
For convenience set
\begin{equation}
    \label{eq: Finalisation of Poisson: U_s definition}
    \mathcal{U}_s:=\sum_{\underline{h}\in\mathcal{H}_s} T_{\underline{h}}(q,\underline{z}),
\end{equation}
(recall that $\sum_{\h\ll H}T_{\h}(q,\z)\ll P^{\ve} \sum_{s=-1}^{n-1} \mathcal{U}_s$ by \eqref{eq: Finalisation of Poisson: h sum bounded by H_s sum: 1}). We will use \eqref{E71} to bound $\mathcal{U}_s$ later, but for now, we need to find a bound on $|T_{\underline{h}}(q,\underline{z})|$. To do this we will need to apply the hyperplane intersections lemma, namely Lemma \ref{lem:hyperplane} and then apply the bounds found in Propositions \ref{proposition:Poissonquadfinal} and \ref{proposition:n=1Poissonquadfinal}.

Let $\eta$ be chosen so as to maximize the expression in \eqref{E71}. Let $\Pi$ be the set of primes $p|q$ so that $r=\omega(q)$, and $\{F_1,F_2\}=\{F_{\underline{h}}^{(0)},G_{\underline{h}}^{(0)}\}.$ We may now invoke Lemma \ref{lem:hyperplane} to find a lattice $\Lambda_{\eta}$ of rank $n-\eta$ and a basis $\underline{e}_1,\cdots,\underline{e}_{n-\eta}$ for $\Lambda_{\eta}$ s.t. for every $\underline{t}\in\mathbb{Z}^n$, the polynomials
\[\tilde{F}_{\h,\underline{t}}(\y):=F_{\underline{h}}^{(0)}(\underline{t}+\sum_{i=1}^{n-\eta} y_i\underline{e}_i),\quad \tilde{G}_{\h,\underline{t}}(\y):=G_{\underline{h}}^{(0)}(\underline{t}+\sum_{i=1}^{n-\eta} y_i\underline{e}_i)\]
satisfy
\begin{equation}
    \label{E72}
    s_{\nu}(\tilde{F}_{\h,\underline{t}},\tilde{G}_{\h,\underline{t}})=\max\{-1, s_{\nu}(F^{(0)}_{\underline{h}},G^{(0)}_{\underline{h}})-\eta\},
\end{equation}
for every $\nu\in\{\infty\}\cup\Pi_{cr}.$ We also note that $\deg(\tilde{F}_{\h,\underline{t}})=\deg(\tilde{G}_{\h,\underline{t}})=2$ (this is necessary in order to be able to use the bounds from the previous section). In order to apply the bounds found in the previous section, we must first fix our choice of basis $\{\underline{e}_1,\cdots,\underline{e}_n\}$, and so we will use the same process as earlier when we fixed $(b_{1,0},\cdots, q_{3,n})$: We recall that the $L$ used in \eqref{eq:cond4} is of size $L=O(r+1)=O(\log(q))$. Therefore there are at most $O(\log(q)^n)$ choices of basis satisfying \eqref{eq:cond4}, and so by \eqref{eq: Finalisation of Poisson: U_s definition}, and the triangle inequality, there is one such choice for which
\begin{equation}
    \label{eq: Finalisation of Poisson: U_s bound in terms of H_s' sum: 2}
    \mathcal{U}_s\ll \log(q)^n\sum_{\underline{h}\in\mathcal{H}_s}{}^{'}|T_{\underline{h}}(q,\underline{z})|\ll P^{\ve}\sum_{\underline{h}\in\mathcal{H}_s}{}^{'}|T_{\underline{h}}(q,\underline{z})|,
\end{equation}
where $\sum'$ denotes that the sum is taken over the vectors $\underline{h}$ in the original sum for which $\eqref{E72}$ holds for our chosen basis $\{\underline{e}_1,\cdots,\underline{e}_n\}$. For such $\underline{h}$, we can now separate the $\underline{x}$ sum defining $T_{\underline{h}}(q,\underline{z})$ into cosets $\underline{t}+\Lambda_\eta$ of $\Lambda_\eta$, where $\underline{t}$ runs over some subset $T_\eta\subset \Z^n$. All that is left to do is use Proposition \ref{proposition:Poissonquadfinal} (or Proposition \ref{proposition:n=1Poissonquadfinal} for $\eta=n-1$) on each coset, and determine the size of $T_\eta$, as this bounds the number of cosets that we have. We claim that if $\Lambda_{\eta}$ is chosen according to Lemma \ref{lem:hyperplane}, then $\# T_\eta=O(P^\eta)$. Indeed, consider $\underline{x}$ in terms of our basis $\underline{e}_1,\cdots,\underline{e}_n$, i.e. writing
\[\underline{x}=\sum_{i=1}^n u_i\underline{e}_i.\]
Now, if $\pi_i$ denotes the projection onto the orthogonal subspace spanned by the vectors $\underline{e}_j$, $i\neq j$, we have
\[||\underline{x}||\geq ||\pi_i \underline{x}||=|u_i|\cdot||\pi\underline{e}_i||=|u_i|\frac{|\det(\Lambda)|}{|\det(\Lambda_i)|},\]
where $\Lambda\subset\mathbb{Z}^n$ denotes the full-dimensional lattice spanned by $\underline{e}_1,\cdots,\underline{e}_n$ and $\Lambda_i$ the lattice spanned by each $\underline{e}_j\neq \underline{e}_i$. Now by \eqref{eq:cond4} and \eqref{eq:cond5}, we get that
\begin{equation}
    \label{E73}
    |u_i|\ll \frac{||\underline{x}||}{L}.
\end{equation}
Therefore we must have $|u_i|\ll P$ since we need $||\underline{x}||\ll P$. Hence, since $\Lambda_{\eta}=<\e_1,\cdots, \e_{n-\eta}>$, we may conclude that $\underline{t}$ is of the form $\underline{t}=\sum_{i=n-\eta+1}^{n} \lambda_i\underline{e}_i$ s.t. $|\lambda_i|\ll P$. We now choose $T_\eta$ to be the collection of such $\underline{t}$ leading us to conclude that $\#T_{\eta}=O(P^{\eta})$.

In order to complete the hyperplane intersections step, we will now define new weight functions in $n-\eta$ variables. In particular, we set
\begin{equation*}
    \tilde{\omega}_{\h,\un{t}}(y_1,\cdots, y_{n-\eta}):=\omega_{\h}\Bo P^{-1}\un{t}+L^{-1}\sum_{i=1}^{\eta} y_i\un{e}_i\Bc.
\end{equation*}
This gives us 
\begin{equation}
    \label{eq: Finalisation of Poisson: T_h(q,z) bounded by hyperplane intersections: 3}
    |T_{\h}(q,\z)|\leq \sum_{\un{t}\in T_{\eta}} |T_{\h,\un{t}}(q,\z)|, \mbox{ where }  T_{\h,\un{t}}(q,\z)=T_{n-\eta}(q,\z;\tilde{F}_{\h,\un{t}},\tilde{G}_{\h,\un{t}},\tilde{\omega}_{\h,\un{t}},P/L).
\end{equation}
We now need to verify that $T_{\h,\un{t}}(q,\z)$ and $\tilde{\omega}_{\h,\un{t}}$ satisfy the various properties that we assumed in order to acquire the results we have found in the previous sections. Firstly, we refer to the proof Proposition 2 of \cite{Browning-Heath-Brown09} to see that $\tilde{\omega}_{\h,\un{t}}\in\mathcal{W}_{n-\eta}$ for $\un{t}\ll P$. We also see that
\[||\tilde{F}_{\h,\un{t}}||_{P/L}\ll L^2||F_{\h}||P\ll P^{\ve}H||F||_p\ll P^{\ve}H,\]
and similarly $||\tilde{G}_{\h,\un{t}}||_{P/L}\ll P^{\ve}H$. Next, we note that $\eta\geq s+1$, and so we automatically have $s_{\infty}(\tilde{F}_{\h,\ut},\tilde{G}_{\h,\ut})=-1$. This covers all conditions that we have needed in the previous sections on exponential sums. 

Therefore, by \eqref{eq: Finalisation of Poisson: U_s bound in terms of H_s' sum: 2}, \eqref{eq: Finalisation of Poisson: T_h(q,z) bounded by hyperplane intersections: 3}, and \eqref{E71}:
\begin{align}
    \mathcal{U}_s&\ll P^{\ve}\sum_{\underline{h}\in\mathcal{H}_s}{}^{'}\sum_{\un{t}\in T_{\eta}} |T_{\h,\un{t}}(q,\z)| \nonumber \\
    &\ll P^{\ve} \#\mathcal{H}_s \# T_{\eta} \max_{\underline{h}\in\mathcal{H}_s}{}^{'}\max_{\un{t}\in T_{\eta}} |T_{\h,\un{t}}(q,\z)| \nonumber\\
    \label{eq: Finalisation of Poisson: U_s final bound: 4}
    &\ll \max_{s+1\leq\eta\leq n}\frac{P^{\eta+\ve}H^{n-\eta}}{\prod_{i=\eta+1}^n((b_{1,i}b_{2,i}^{1/2}\tilde{q}_{3,i}))^{i-\eta}}\cdot \max_{\h\in\mathcal{H}_s}{}'\max_{\un{t}\in T_{\eta}} T_{\h,\un{t}}(q,\z).
\end{align}
Recall that 
\begin{equation}
    \label{eq: Finalisation of Poisson: h sum bounded by U_s}
    \sum_{\h\ll H}T_{\h}(q,\z)\ll P^{\ve} \sum_{s=-1}^{n-1} \mathcal{U}_s\ll P^{\ve} \max_{-1\leq s\leq n-1} \mathcal{U}_s.
\end{equation}
by \eqref{eq: Finalisation of Poisson: h sum bounded by H_s sum: 1} and \eqref{eq: Finalisation of Poisson: U_s definition}. We will therefore be able to attain our final bound for $\sum_{\h\ll H}T_{\h}(q,\z)$ if we can find a bound for $T_{\h,\un{t}}(q,\z)$ by \eqref{eq: Finalisation of Poisson: U_s final bound: 4}.

We may use Propositions \ref{proposition:Poissonquadfinal} and \ref{proposition:n=1Poissonquadfinal} to bound $T_{\h,\un{t}}(q,\z)$ from above when $\eta<n-1$ and $\eta=n-1$ respectively. When $\eta=n$, we may proceed by a much simpler argument to bound $T_{\h,\ut}(q,\z)$. We trivially have
\[|T_{\h}(q,\z)| \leq \starsum_{\ua}\sum_{\underline{y}\in\mathbb{Z}^n} \omega_{\h}(\y/P) \ll q^2 P^n,\]
and by Lemma \ref{T71} ($v=\infty$, $i=n$), we have that 
\[\#\{\h\in \mathbb{A}_{\Q}^n \:|\: s_{\infty}(\h)= n-1\}= O(1).\]
Hence
\begin{align}
    \label{E75}
    \sum_{\substack{\h\ll H \\ s_{\infty}(\h)=n-1}}|T_{\h (q,\z)}|\ll q^2P^n.
\end{align}
Returning to $\eta\leq n-1$: By \eqref{E72}, we may use the proof of Proposition 2 from \cite{Browning-Heath-Brown09} to conclude that for every $\un{t}\in T_{\eta}$, we have
\[D_{\tilde{F}_{\h,\ut},\tilde{G}_{\h,\ut}}(b_{1,i}b_{2,i}\tilde{q}_{3,i})\ll q^{\ve}\prod_{i=\eta+1}^n (b_{1,i}b_{2,i}^{1/2}\tilde{q}_{3,i})^{i-\eta},\]
when $\eta<n-1$. When $\eta=n-1$, $(p,\cont(\tilde{F}_{\h,\ut}),\cont(\tilde{G}_{\h,\ut}))=p$ if and only if $p|\tilde{F}_{\h,\ut},\tilde{G}_{\h,\ut}$ or $p\ll P^{\ve}$. In particular, $p| b_{1,n}b_{2,n}^{1/2}\tilde{q}_{3,n}$ or $p\ll P^{\epsilon}\asymp q^{\ve}$, and so we again have
\[D_{\tilde{F}_{\h,\ut},\tilde{G}_{\h,\ut}}(b_{1,n}b_{2,n}\tilde{q}_{3,n})\ll q^{\epsilon}b_{1,n}b_{2,n}^{1/2}\tilde{q}_{3,n}.\] 
Therefore, by \eqref{eq: Finalisation of Poisson: U_s final bound: 4} \eqref{eq: Finalisation of Poisson: h sum bounded by U_s}, and Propositions \ref{proposition:Poissonquadfinal} and \ref{proposition:n=1Poissonquadfinal} and \eqref{E75}, we may conclude the following:
\Pb
\label{T76}
Let $q<P^2$, and let
\[\mathcal{Y}_{\eta}:=\frac{H^{n-\eta}}{q^{(n-\eta)/2}}b_1^{-1}\Bo V^{n-\eta}+V^{n-\eta-1}b_1^{1/2}b_3^{1/3}q_4^{1/2} + b_1^{1/2}b_3^{(n-\eta)/3}q_4^{(n-\eta)/2}\Bc,\]
for $\eta\in\{0,\cdots, n-2\}$ and let
\[\mathcal{Y}_{n-1}:=\frac{H}{q^{1/2}}b_1^{-1/2}(V+b_3^{1/3}q_4^{1/2}).\]
Then
\[\sum_{\h\in H_s}|T_{\h}(q,\z)|\ll q^2P^{n+\ve}\Bo 1+\sum_{\eta=0}^{n-1} \mathcal{Y}_{\eta}\Bc.\]
\Pe

\section{Weyl Differencing}
\label{Weyl}
In this section, we will derive several auxiliary bounds using Weyl differencing which will serve as complimentary bounds to the more powerful ones coming from van der Corput differencing and Poisson summation. We will need a bound which uses Weyl differencing twice, as well as two bounds which come from applying variations of van der Corput differencing once, followed by a single application of Weyl differencing on the resulting quadratic exponential sum. In the case of the former: The topic of performing Weyl differencing repeatedly on a system of forms has already been covered extensively by Lee in the context of function fields \cite{Lee11}. The Weyl differencing arguments that are used in his paper do not rely on being in a function fields setting, and so we may freely invoke the results in \cite[Section 3]{Lee11}. In particular, upon setting $d=3$ and $R=2$, an application of \cite[Lemma 3.7]{Lee11} gives us
\begin{equation*}
    |S(\underline{a}/q+\underline{z})|\ll P^{n+\ve}\Big{(} P^{-4}+q^2|\underline{z}|^2+q^2P^{-6}+q^{-1}\min\{1,\frac{1}{|\underline{z}|P^3}\}\Big{)}^{(n-\sigma'-1)/16},
\end{equation*}
where 
\begin{equation}
    \label{eq: sigma' def}
    \sigma'=\sigma'(F^{(0)},G^{(0)}):= \dim \{\x\in\mb{P}_{\C}^{n-1} \: : \: \rank\begin{pmatrix}
    \nabla F^{(0)}(\x) \\ \nabla G^{(0)}(\x)
    \end{pmatrix}<2\},
\end{equation}
and $\Fnull$, $\Gnull$ are defined to be the top forms of $F$ and $G$ respectively. However, we may use Lemma \ref{P: background on a pair of quadrics: sigma'<= sigma+1} to conclude that $\sigma'\leq \sigma(\Fnull,\Gnull)+1$. Hence, we arrive at the following:
\Pb[Weyl/Weyl]
\label{T815}
Let $F$, $G$ be cubic polynomials such that
\[\|\Fnull\|,\|\Gnull\|\asymp 1,\]
and $\sigma(F^{(0)},G^{(0)})=\sigma$. Then:
\[|S(\underline{a}/q+\underline{z})|\ll P^{n+\ve}\Big{(} P^{-4}+q^2|\underline{z}|^2+q^2P^{-6}+q^{-1}\min\{1,\frac{1}{|\underline{z}|P^3}\}\Big{)}^{(n-\sigma-2)/16}.\]
\Pe

We now aim to bound the exponential sum, 
\[T(q,\z):=\starsum_{\ua}\sum_{\x\in \Z^n} \omega(\x/P) e([a_1/q+z_1] F(\x) + [a_2/q+z_2] G(\x))\]
that we get after performing van der Corput differencing once. In this case, $F$ and $G$ are quadratic polynomials such that $\|\Fnull\|,\|\Gnull\|\ll H$, for some $1\leq H\leq P$. The aforementioned work of Lee has also kept an explicit dependence of the dependence on $H$ throughout the Weyl differencing process. In particular the following lemma is a direct consequence of \cite[Equation (3.20)]{Lee11}:
\Pb[van der Corput/Weyl]
\label{T814}
Let $F$, $G$ be quadratic polynomials such that 
\[\|\Fnull\|,\|\Gnull\|\leq H,\]
and let $\sigma:=\sigma(\Fnull,\Gnull)$. Then:
\begin{align*}
    |T(\ua,q,\z)|\ll P^{n+\ve}\Big{(} P^{-2}+q^2H^2|\underline{z}|^2+q^2&P^{-4}+q^{-1}H^2\min\{1,\frac{1}{H|\underline{z}|P^2}\}\Big{)}^{(n-\sigma-2)/4}.
\end{align*}
\Pe
We refer the readers not familiar with the function field version to the first author's PhD thesis \cite[Section 6]{Northey} for a detailed proof of this result.
\section{Minor Arcs Estimate}
\label{sec:minor}
In this Section, we will combine all of the approaches we have been developing throughout this paper to finally prove Proposition \ref{P02}. In particular we aim to show that, provided that $F,G$ intersect smoothly, and $n\geq 39$, we have
\[S_{\mathfrak{m}}=O(P^{n-6-\delta}),\]
for some $\delta>0$. To achieve this, we will split the $q$ sum of $S_{\mathfrak{m}}$ into square-free, cube-free square, 4th power-free cube, and 4th power-full parts ($b_1,b_2,b_3, q_4$ respectively), and further split these sums into $O(P^{\ve}$) dyadic ranges. In particular, we will be focusing on the sum
\[\Dy:=\sum_{b_1=R_1}^{2R_1}\sum_{b_2=R_2}^{2R_2}\sum_{b_3=R_3}^{2R_3}\sum_{q_4=R_4}^{2R_4}\starsum_{\ua}\int_{t\ll|\z|\ll t} | S_{\ua}(q,\z) |\:d\z,\]
where, $\un{R}:=(R_1,R_2,R_3)$, and 
\begin{equation}
    \label{E91}
    q=b_1b_2b_3q_4, \quad R<q\leq 2R,\quad R_i<b_i\leq 2R_i, \,i\in\{1,2,3\},\quad R_4<q_4\leq 2R_4,
\end{equation}
(the latter is apparent from the definition of $\Dy$, but it will be helpful to be able to reference this later). From the definition of $S_{\mathfrak{m}}$, we need only consider $\Dy$ when
\begin{equation}
    \label{E92}
    R\leq Q, \quad R_1R_2R_3R_4\asymp R,\quad 0\leq t\leq (RQ^{1/2})^{-1}. 
\end{equation}
Likewise, we must also either have
\begin{equation}
    \label{E93}
    R\geq P^{\Delta} \quad \mbox{or}\quad t\geq P^{-3+\Delta}.
\end{equation}
Now, upon bounding $S_{\ua}(q,\z)$ trivially for $t\leq P^{-5}$, we see that
\begin{equation}
    \label{E94}
    S_{\mathfrak{m}}\ll P^{\ve} \max_{\substack{R,\un{R},\ut \\ \eqref{E92}, \eqref{E93}, \, t> P^{-5}}} \Dy + O(P^{n-7}).
\end{equation}
Our aim in this section is to show that $\Dy\ll P^{n-6-\delta}$ for some $\delta>0$, as this is sufficient to bound our minor arcs by $P^{n-6-\delta}$ by \eqref{E94}. Note that this is equivalent to proving that
\begin{equation}
    \log_P(\Dy):=B_P(\phi, \tau, \un{\phi})\leq n-6-\delta,
\end{equation}
for some $\delta>0$, and for $P$ sufficiently large (so that the implied constant in \eqref{E94} becomes negligible), where
\begin{equation}
    \phi:= \log_P(R), \quad \tau:= \log_P(t), \quad \log_P(R_i):=\phi_i, \; i\in\{1,2,3,4\}.
\end{equation}
Finally, as mentioned in Section \ref{C: Initial setup} we will choose 
\[Q\asymp P^{3/2},\]
from this point onwards (this choice will be explained in Section \ref{S: Explaining Q}). With this last bit of setup, we are now ready to start the process of bounding $S_{\mathfrak{m}}$. We will do this by applying a total of five different bounds based on different combinations of van der Corput differencing, Weyl differencing, and Poisson summation to bound $\Dy$ for different ranges of $R$ and $t$. In each range, we will take the minimum of all available bounds. To do so for all possible values of $R$ and $t$ is incredibly complicated. Therefore, instead of the tedius process of manually comparing and simplifying these bounds, a route which is traditionally taken, we take the idea of automatising this process as in \cite{Marmon_Vishe} one step further. We will directly feed these bounds to the existing Min-Max algorithm in Mathematica and obtain an explicit value of the minimum value of our bounds on the Minor arcs. We have also verified this value using an open source algorithm \cite{Northey1} designed by the first author. In its current form, this algorithm is significantly less efficient than inbuilt one in Mathematica, but it allowed the authors to double check the bounds coming from this inbuilt function.

Throughout this section, we will use the following Lemma:
\Lb
\label{T95}
Let $q=b_1b_2\cdots b_k q_{k+1}$, where $b_i$ is the i-th power, (i+1)th powerfree part of $q$ and let $q_{k+1}$ be the (k+1)th power-full part of $q$. Then 
\[\sum_{\substack{b_i=R_i\\ i\in\{1,\cdots,\, k\}}}^{2R_i}\sum_{q_{k+1}=R_{k+1}}^{2R_{k+1}} b_1^{a_1} b_2^{a_2}\cdots b_k^{a_k} q_3^{a_{k+1}} \ll \prod_{i=1}^{k+1} R_i^{a_i+1/i},\]
for every $a_1,\cdots,a_{k+1}\geq 0$.
\Le
The proof of this lemma is standard, and is similar to \cite[Lemma 20]{Browning-Heath-Brown09} so we omit here. This Lemma enables us to get away with using slightly worse exponential sum bounds for the perfect square and cube-full parts of $q$ (close inspection of the bounds found in Section \ref{Sec: Exponential Sums: Initial Bounds} will show that our bounds in these cases are indeed worse). We have stated Lemma \ref{T95} in this level of generality because it will be useful for us when considering the singular series of the major arcs. We will spend the remainder of this section finding our final bounds for the minor arcs in the case when $F$,$G$ are non-singular. 
\subsection{Averaged van der Corput/Poisson}
In this section, we will find a bound for $B_P(\phi, \tau, \un{\phi}):=\log_P(\Dy)$ by combining the improved averaged van der Corput differencing process with Poisson summation. We will aim to show that $B_P(\phi, \tau, \un{\phi})\leq n-6-\delta$ for some $\delta>0$, provided that $n$ is sufficiently large. By Proposition \ref{T45}, we have

\begin{align}
    \label{E913}
    \Dy\ll_{\ve,N} P^{-N}+\sum_{q,\eqref{E91}}H^{-n/2+1}P^{n/2-1+\ve}q((HP^2)^{-1}+t)^2\Big{(}\max_{\uz}\sum_{|\h|\ll H}|T_{\underline{h}}(q,\underline{z})|\; \Big{)}^{1/2},
\end{align}
where $|\z| \asymp \max\{(HP^2)^{-1},t\}$. By Proposition $\ref{T76}$ we have
\begin{equation}
\label{eq:aux}
\sum_{\h\ll H}|T_{\h}(q,\z)|\ll q^2P^{n+\ve}\Co 1+\sum_{\eta=0}^{n-1} \mathcal{Y}_{\eta}\Cc,
\end{equation}

where
\begin{equation}
\label{eq: Minor Arcs bound: Y definition}
    \mathcal{Y}_{\eta}(q,b_1,b_3,q_4,|\z|):=\frac{H^{n-\eta}}{q^{(n-\eta)/2}}b_1^{-1}\Bo V^{n-\eta}+b_1^{1/2}b_3^{1/3}q_4^{1/2}V^{n-\eta-1} + b_1^{1/2}b_3^{(n-\eta)/3}q_4^{(n-\eta)/2}\Bc
\end{equation}
for $\eta\in\{0,\cdots, n-2\}$,
\begin{align}
    \label{eq: Minor Arcs bound: Y definition: n-1}
    \mathcal{Y}_{n-1}:=\frac{H}{q^{1/2}}b_1^{-1}(b_1^{1/2}V+b_1^{1/2}b_3^{1/3}q_4^{1/2}),
\end{align}
and
\begin{align}
    \label{eq: minor arcs bound, Poisson H definition}
    H(q)&:=\max\{P^{10/(n-2)+\ve'}, P^{2/(n+2)+\ve'}q^{6/(n+2)}\}\\
    \label{eq: minor arcs bound, Poisson V definition}
    V(q,|\z|)&:=1+qP^{\ve-1}\max\{1,\sqrt{HP^2|\z|}\}.
\end{align}
Our choice of $H$ is informed by the fact that $H$ needs to be large enough so that the contribution from the $\eta=n$ term (or equivalently the term $1$ in \eqref{eq:aux}) is satisfactory . We note that $V(q,|\z|)\asymp V(q,t)$ in the range of $\z$ that we have. Hence (assuming $N$ is chosen sufficiently large):
\begin{align}
    \Dy&\ll \sum_{q,\eqref{E91}}H^{-n/2+1}P^{n-1+\ve}q^2((HP^2)^{-1}+t)^2 \Bo 1+\sum_{\eta=0}^{n-1}\mathcal{Y}_{\eta}(q,b_1,b_3,q_4,t)\Bc^{1/2}\nonumber\\
    \label{E914}
    &\ll P^{n-1+\ve}\sum_{\substack{b_i=R_i\\i\in\{1,2,3\}}}^{2R_i}\sum_{q_4=R_4}^{2R_4}R^2H^{-n/2+1}((HP^2)^{-1}+t)^2 (1+\mathcal{Y}_{0}+\cdots+\mathcal{Y}_{n-1})^{1/2}\\
    \label{E9140}
    &\ll P^{n-1+\ve}\mathcal{R}R_1^{1/2}R^{2}H^{-n/2+1} ((HP^2)^{-1}+t)^2 (1+\mathcal{Y}_{0}+\cdots+\mathcal{Y}_{n-1})^{1/2},
\end{align}
where $\mathcal{R}:=R_1^{1/2}R_2^{1/2}R_3^{1/3}R_4^{1/4}$, $H=H(R)$, $V=V(R,t)$, and $\mathcal{Y}_{i}=\mathcal{Y}_i(R,R_1,R_3,R_4,t)$ in \eqref{E914}-\eqref{E9140}. For the most part, we will continue to use $H$, $V$, and $\mathcal{Y}_{i}$ instead of $H(R)$, $V(R,t)$ and $\mathcal{Y}_i(R,R_1,R_3,R_4,t)$ to avoid making the algebra more complicated than it already is. The final assertion is by Lemma \ref{T95}. 

We will start by simplifying the right-most bracket:
\Lb
\label{L: Minor Arcs bound: Y bracket geometric series}
For every $R,R_1,R_3,R_4,t$ satisfying \eqref{E92}, we have
\[(1+\mathcal{Y}_{0}+\cdots+\mathcal{Y}_{n-1})\ll (1+\mathcal{Y}_{0}).\]
\Le

\begin{proof}
For this proof, we will introduce the following sequence:
\begin{align*}
    \mathcal{Y}_{\eta}':= \frac{H^{n-\eta}}{R^{(n-\eta)/2}}R_1^{-1}\Bo R_1^{\eta/n}V^{n-\eta}+R_1^{1/2}R_3&^{1/3-\eta/3n}R_4^{1/2-\eta/2n}V^{n-\eta-1+\eta/n}\\
    &\qquad\qquad\qquad\qquad+ R_1^{1/2}R_3^{(n-\eta)/3}R_4^{(n-\eta)/2}\Bc.
\end{align*}
We will prove that this sequence has the following three properties:
\begin{enumerate}
    \item $\mathcal{Y}_{\eta}\ll \mathcal{Y}_{\eta}'$ for every $\eta\in\{0,\cdots, n-1\}$.
    \item $\mathcal{Y}_{0}'=\mathcal{Y}_{0}$, and $\mathcal{Y}_{n}'\asymp 1$.
    \item $\sum_{\eta=0}^n\mathcal{Y}_{\eta}'$ is a sum of three geometric series.
\end{enumerate}
Verifying these three facts is sufficient to complete the proof since properties 1 and 2 imply that $(1+\mathcal{Y}_{0}+\cdots+\mathcal{Y}_{n-1})\ll (\mathcal{Y}_{0}'+\cdots+\mathcal{Y}_{n-1}'+\mathcal{Y}_n')$, property 3 implies that $(\mathcal{Y}_{0}'+\cdots+\mathcal{Y}_{n-1}'+\mathcal{Y}_n')\ll (\mathcal{Y}_{0}'+\mathcal{Y}_{n}')$, and property 2 implies that $(\mathcal{Y}_{0}'+\mathcal{Y}_{n}')=(1+\mathcal{Y}_{0})$.

For property 1, we note that the term outside of the bracket of $\mathcal{Y}_{\eta}'$ is equal to the analogous term in $\mathcal{Y}_{\eta}$. It therefore suffices to bound each term in the bracket of $\mathcal{Y}_{\eta}$ from above by a term in $\mathcal{Y}_{\eta}'$: We clearly have $V^{n-\eta}\leq R_1^{\eta/n}V^{n-\eta}$ when $\eta\in\{1,\cdots,n-2\}$ and $R_1^{1/2}V\leq R_1^{(n-1)/n}V$ for every $n\geq 1$. The third term of $\mathcal{Y}_{\eta}$ and $\mathcal{Y}_{\eta}'$ coincide with each other for every $\eta\in\{1,\cdots,n-1\}$.

As for the middle term, \[R_1^{1/2}R_3^{1/3}R_4^{1/2}V^{n-\eta-1}\leq R_1^{1/2}R_3^{1/3-\eta/3n}R_4^{1/2-\eta/2n}V^{n-\eta-1+\eta/n}\]
 if and only if $V\geq R_3^{1/3}R_4^{1/2}$. However, if $V<R_3^{1/3}R_4^{1/2}$, then 
\[R_1^{1/2}R_3^{1/3}R_4^{1/2}V^{n-\eta-1}\leq R_1^{1/2}R_3^{(n-\eta)/3}R_4^{(n-\eta)/2},\] 
which is the third term of $\mathcal{Y}_{\eta}'$. Hence we have $\mathcal{Y}_{\eta}\ll \mathcal{Y}_{\eta}'$.

Property 2 is trivial so we will move to verifying property 3. Again, we will go term by term: Let 
\[\mathcal{Y}_{\eta,1}':=\frac{H^{n-\eta}}{R^{(n-\eta)/2}}R_1^{-1}\cdot R_1^{\eta/n}V^{n-\eta}.\]
Then 
\[\mathcal{Y}_{\eta+1,1}'=HR^{-1/2}R_1^{1/n}V^{-1}\mathcal{Y}_{\eta,1}'.\]
If we similarly define $\mathcal{Y}_{\eta,2}'$ and $\mathcal{Y}_{\eta,3}'$ in the obvious way, then we see that 
\[\mathcal{Y}_{\eta+1,2}'=HR^{-1/2}R_3^{-1/3n}R_4^{-1/2n}V^{-1+1/n}\mathcal{Y}_{\eta,1}', \quad \mathcal{Y}_{\eta+1,3}'=HR^{-1/2}R_3^{-1/3n}R_4^{-1/2n}\mathcal{Y}_{\eta,3}'.\]
Hence we may represent $\sum \mathcal{Y}_{\eta}'$ as a sum of three geometric series, as required. This completes the proof.
\end{proof}
We may use Lemma \ref{L: Minor Arcs bound: Y bracket geometric series} to conclude that 
\begin{equation}
    \label{eq: minor arcs bound: AVdc-Poisson Dyadic: 1}
    \Dy\ll P^{n-1+\ve}\mathcal{R}R_1^{1/2}R^{2}H^{-n/2+1} ((HP^2)^{-1}+t)^2(1+\mathcal{Y}_{0})^{1/2}.
\end{equation}
We now aim to simplify this expression further by showing that\\
$V^{n}\leq R^{1/2}R_3^{1/3}R_4^{1/2}V^{n-1}$, or equivalently that $V\leq R^{1/2}R_3^{1/3}R_4^{1/2}$. Doing this, will let us show the following:
\Lb
\label{L: Minor Arcs bound: Y bracket bound}
Let $Q=P^{3/2}$ and let $H$ and $V$ be defined as above. If $n\geq 23$ then
\[V\leq R^{1/2}.\]
In particular
\begin{align*}
    \mathcal{Y}_{0}&\leq R^{(1-n)/2}R_1^{-1}H^n(R^{1/2} V^{n-1} + R_3^{n/3-1/2}R_4^{(n-1)/2})\\
\end{align*}
\Le
\begin{proof}
As mentioned we will firstly prove that $V\leq R^{1/2}$. Recall that 
\[V=V(R,t)=1+RP^{-1+\ve}\max\{1,H(R)P^2t\}^{1/2}.\]
We clearly have $1\leq R^{1/2}$. 

When $V=RP^{-1+\ve}$, we note that $R>P^{1-\ve}$ otherwise $RP^{-1+\ve}\leq 1$, and so $V$ cannot be equal to $RP^{-1+\ve}$. Furthermore we see that $R\leq Q=P^{3/2}$, or equivalently $P^{-1}\leq R^{-2/3}$. Hence, provided that $\ve$ is chosen small enough so that $P^{\ve}\leq R^{1/6}$, then we also have $P^{-1+\ve}<R^{-1/2}$. Since $R>P^{1-\ve}$, $\ve<0.1$ would suffice for example.

Finally, we consider when $V=RP^{-1+\ve}(H(R)P^2t)^{1/2}$. In this case, since $t\leq (RQ^{1/2})^{-1}$,
\[V\leq R^{1/2} Q^{-1/4} P^{\ve}\max\{P^{5/(n-2)+\ve}, P^{1/(n+2)} R^{3/(n+2)}\}.\]
But since $R\leq Q$, $P<Q$ (and $5/(n-2)>4/(n+2)$), we have 
\[V< R^{1/2}Q^{5/(n-2)+\ve-1/4}<R^{1/2},\]
provided $n\geq 23$.

This concludes the proof that $V\leq R^{1/2}$. For the second statement of the lemma, we start by noting that $V^{n}\leq R^{1/2}V^{n-1}$, and so by \eqref{eq: Minor Arcs bound: Y definition} and \eqref{E92}:
\begin{align*}
    \mathcal{Y}_{0}(R,R_1,R_3,R_4,t):&=\frac{H^{n}}{R^{n/2}}R_1^{-1}\Bo V^{n}+R_1^{1/2}R_3^{1/3}R_4^{1/2}V^{n-1} +R_1^{1/2}R_3^{n/3}R_4^{n/2}\Bc\\
    &\leq \frac{H^{n}}{R^{n/2}}R_1^{-1}\Bo V^{n}+R^{1/2}V^{n-1} +R^{1/2}R_3^{n/3-1/2}R_4^{(n-1)/2}\Bc\\
    &\ll \frac{H^{n}}{R^{n/2}}R_1^{-1}\Bo R^{1/2}V^{n-1} +R^{1/2}R_3^{n/3-1/2}R_4^{(n-1)/2}\Bc\\
    &=R^{(1-n)/2}R_1^{-1}H^n(V^{n-1} + R_3^{n/3-1/2}R_4^{(n-1)/2}).
\end{align*}
\end{proof}
Hence, if we let
\begin{align}
    \label{eq: minor arcs bound: X_1 definition}
    \mathcal{X}_{1}(R,R_3,R_4,t)=\mathcal{X}_{1}&:=R^{(1-n)/2} H(R)^n V(R,t)^{n-1},\\
    \label{eq: minor arcs bound: X_2 definition}
    \mathcal{X}_{2}(R,R_3,R_4)=\mathcal{X}_{2}&:=R^{(1-n)/2}R_3^{n/3-1/2}R_4^{(n-1)/2} H(R)^n,
\end{align}
then we now may Lemma \ref{L: Minor Arcs bound: Y bracket bound} and \eqref{eq: minor arcs bound: AVdc-Poisson Dyadic: 1} to bound $\Dy$ as follows:
\begin{align}
    \Dy&\ll P^{n-1+\ve}\mathcal{R}R^{2}H^{-n/2+1} ((HP^2)^{-1}+t)^2(R_1+\mathcal{X}_{1}+\mathcal{X}_{2})^{1/2})\nonumber\\
    \label{eq: minor arcs bound: AVdc-Poisson Dyadic: 2}
    &\ll P^{n-1+\ve}R^{5/2}H^{(2-n)/2} \max\{(HP^2)^{-1},t\}^2\max\{R,\mathcal{X}_{1},\mathcal{X}_{2}\}^{1/2}.
\end{align}
Finally, note that $\Dy\ll P^{n-6-\delta}$ for some $\delta>0$ if $\log_{P}(\Dy)\leq n-6-\delta$ (provided $P$ is chosen large enough) and so it is sensible to consider bounding $B_P(\phi, \tau, \un{\phi}):=\log_P(\Dy)$. By \eqref{eq: minor arcs bound: AVdc-Poisson Dyadic: 2} and upon letting $R:=P^{\phi}$, $R_i:=P^{\phi_i}$, $t:=P^{\tau}$, we have
\begin{align}
B_P(\phi, \tau, \un{\phi}) &\leq \log_P(P^{n-1+\ve}R^{5/2}H^{(2-n)/2} \max\{(HP^2)^{-1},t\}^2\max\{R,\mathcal{X}_{1},\mathcal{X}_{2}\}^{1/2})\nonumber\\
                         &= n-1 + \ve + \frac{5\phi}{2}+ \frac{(2-n)}{2}\cdot \log_P(H) + 2\max\{-2-\log_P(H),\tau\} \,+ \nonumber\\
        \label{eq: minor arcs bound: AVdc-Poisson Dyadic: 3}
                         &\quad\quad \qquad \frac{1}{2}\max\{\phi, \log_P(\mathcal{X}_{1}),\log_P(\mathcal{X}_{2})\} +\log_P(C),
\end{align}
where $C$ is the implied constant in \eqref{eq: minor arcs bound: AVdc-Poisson Dyadic: 2}. If $P$ is made to be sufficiently large, $\log_P(C)$ can be absorbed into $\ve$. Hence (recalling \eqref{eq: minor arcs bound, Poisson H definition} - \eqref{eq: minor arcs bound, Poisson V definition}, \eqref{eq: minor arcs bound: X_1 definition}-\eqref{eq: minor arcs bound: X_2 definition}), if we set 
\begin{align}
\label{eq: minor arcs bound: log H defn (1)}
    \hat{H}(\phi)&:= \max\{\frac{10}{n-2}+\ve', \frac{2}{n+2}+\ve'+\frac{6\phi}{n+2}\},\\
\label{eq: minor arcs bound: log V defn (2)}    
    \hat{V}(\phi,\tau)&:= \max\{0, -1+\phi, \phi+\frac{\tau+\hat{H}(\phi)}{2} \},\\
\label{eq: minor arcs bound: tau_bracket defn (3)}
    \tau\_\text{brac}(\phi,\tau)&:=\max\{-2-\hat{H}(\phi),\,\tau\},\\
\label{eq: minor arcs bound: X_brac defn (4)}
    \mathcal{X}\_\text{brac}(\phi,\tau,\phi_3,\phi_4)&:=\max\{\phi,\, \frac{(1-n)\phi}{2} + n\,\hat{H}(\phi)+ (n-1)\, \hat{V}(\phi,\tau),\\
    &\quad\quad\quad\quad\quad\;\frac{(1-n)\phi}{2}+\Bo\frac{n}{3}-\frac{1}{2}\Bc\phi_3+\frac{(n-1)\phi_4}{2}+ n\,\hat{H}(\phi)\},\nonumber
\end{align}
(for some small $\ve'>0$ that we may choose freely) then \eqref{eq: minor arcs bound: AVdc-Poisson Dyadic: 3} gives us the following:
\Lb
\label{eq: minor arcs bound: Final bound AVDC/Poisson}
Let $n$ be fixed, and 
\begin{align*}
    B_{AV/P}(\phi,\tau,\phi_3,\phi_4):= n-1 + \frac{5\phi}{2}+ \frac{(2-n)}{2} \hat{H}(\phi)+ 2 \tau\_&\mathrm{brac}(\phi,\tau)+\frac{1}{2}\mathcal{X}\_\mathrm{brac}(\phi,\tau,\phi_3,\phi_4).
\end{align*}
Then $B_{AV/P}(\phi,\tau,\phi_3,\phi_4)$ is a continuous, piecewise linear function, and for every $\ve>0$, there is a sufficiently large $P$ such that 
\[B_P(\phi, \tau, \un{\phi}) \leq B_{AV/P}(\phi,\tau,\phi_3,\phi_4) +\ve, \]
for every $\phi\in [0,3/2]$, $\phi_i\in [0,\phi]$, $\phi_1+\phi_2+\phi_3+\phi_4=\phi$, $\tau\in [-5,-\phi-0.75]$.
\Le
The naming convention used is to make it easier to parse the algorithm's input. For example, $\tau\_\mathrm{brac}$ and $\hat{H}$ correspond to \ti{Tau\_bracket} and  \ti{H\_Poisson} respectively in the algorithm's code.
\subsubsection{The Limiting Case}
\label{S: Limiting Case}
 In this subsection, we will briefly illustrate why we should expect the condition $n\geq 39$ to appear in Proposition \ref{P02} (or equivalently, why we should expect $\Dy\ll P^{n-6-\delta}$ to be true for $n\geq 39$). In general, we expect the limiting condition on $n$ to be determined by the so-called ``generic case" for $(R,t,\un{R})$, which is
\[R=Q=P^{3/2}, \;\; \tau= (RQ^{1/2})^{-1}=P^{-9/4}, \;\; R_1=R= P^{3/2}, \;\; R_2=R_3=R_4=1. \]
This is the case where $R$ is as large as possible and is square-free, and $t$ is as large as possible. In this case, we expect the averaged van der Corput/Poisson bound to dominate over the other bounds since it is our main bound. We will therefore pinpoint which component of \eqref{eq: minor arcs bound: AVdc-Poisson Dyadic: 2} dominates and then solve this part by hand. When we do this, we will see that the condition $n\geq 39$ arises naturally.

Firstly, it is easy to check via the definitions of $H$ and $V$ \eqref{eq: minor arcs bound, Poisson H definition} - \eqref{eq: minor arcs bound, Poisson V definition} that when $R\asymp P^{3/2}$, $t\asymp P^{-9/4}$, $R_1=R$, $R_2=R_3=R_4=1$, we have
\begin{align}
    \label{H def, generic R,t}
    H&=\max\{P^{10/(n-2)+\epsilon'},P^{11/(n+2)+\epsilon'}\},\\
    V&=P^{1/2}\max\{1,\: P^{10/(n-2)-1/4+\epsilon'},\: P^{11/(n+2)-1/4+\epsilon'}\}^{1/2}\nonumber\\
    \label{V def, generic R,t}
    &=P^{3/8+\epsilon'/2}\max\{P^{5/(n-2)},\: P^{11/2(n+2)}\}.
\end{align}
We could remove the ``1" term from $V$ since $P^{10/(n-2)-1/4+\epsilon'}>1$ when $n\leq 42$, and when $n\geq 42$ we must have $ P^{11/(n+2)-1/4+\epsilon'}>1$. It makes sense to consider the cases $n\leq 42$ and $n>42$ separately so that we can simplify $H$ and $V$ further. We will just consider $n\leq 42$ here to avoid repetition, as the purpose here is to only illustrate the final limit of the bounds here.\\
In particular, when $n\leq 42$, then by \eqref{H def, generic R,t} and \eqref{V def, generic R,t}, we have
\begin{align}
    \label{H def V def, generic R,t, n<42}
    H=P^{10/(n-2)+\epsilon'}, \quad V=P^{3/8+5/(n-2)+\epsilon'/2}.
\end{align}
We aim to insert these values into the right-hand side of \eqref{eq: minor arcs bound: AVdc-Poisson Dyadic: 2}, but we will firstly perform some simplifications. In particular, we note that
\begin{align}
    \label{tau bracket, generic R,t, n<42}
    \max\{(HP^2)^{-1}, t\} = \max\{P^{-2-10/(n-2)-\epsilon'}, P^{-9/4}\} = P^{-9/4},
\end{align}
since $n\leq 42$. Similarly by \eqref{eq: minor arcs bound: X_1 definition} - \eqref{eq: minor arcs bound: X_2 definition}, we see that $\mathcal{X}_1>\mathcal{X}_2$ since $R_3=R_4=1$ and $V>1$. Hence
\begin{align}
    \label{X bracket, generic R,t, n<42}
    \max\{R, \mathcal{X}_1, \mathcal{X}_2\}&= \max\{R, R^{(1-n)/2}\cdot P^{10n/(n-2)+n\epsilon'}\cdot P^{3(n-1)/8+5(n-1)/(n-2)+(n-1)\epsilon'/2}\}\nonumber\\
    &=\max\{P^{3/2}, P^{3(1-n)/4+10n/(n-2)+3(n-1)/8+5(n-1)/(n-2)+\epsilon'}\nonumber\}\\
    &=\max\{P^{3/2}, P^{3(1-n)/8+(15n-5)/(n-2)+\epsilon'}\nonumber\}\\
    &=P^{3/2},
\end{align}
provided that $n\geq 38.8111\cdots+\ep'$. In other words, as long as $n\geq 39$ and $\ep'$ is chosen small enough, we have $\max\{R, \mathcal{X}_1, \mathcal{X}_2\}=R=P^{3/2}$. Inserting \eqref{H def V def, generic R,t, n<42} - \eqref{X bracket, generic R,t, n<42} into \eqref{eq: minor arcs bound: AVdc-Poisson Dyadic: 2} gives the following:
\begin{align*}
    \Dy&\ll P^{n-1+\ve}R^{5/2}H^{(2-n)/2} \max\{(HP^2)^{-1},t\}^2\max\{R,\mathcal{X}_{1},\mathcal{X}_{2}\}^{1/2}\\
       &= P^{n-1+\ve}\cdot P^{15/4} \cdot P^{[(2-n)/2] \times [10/(n-2) + \ep']} \cdot P^{-9/2} \cdot P^{3/4}\\
       &= P^{n-1+18/4-5-9/2+\ep-(n-2)\ep'/2}\\
       &= P^{n-6-\delta(\ep,\ep')},
\end{align*}
where $\delta>0$ provided that $\ep$ is chosen sufficiently small with respect to $\ep'$ (and $n>2$). 
\subsection{Pointwise van der Corput/Poisson}
Next, we will find a bound for $B_P(\phi, \tau, \un{\phi})$ by combining the improved Pointwise van der Corput differencing process with Poisson summation. This time, we may assume $t\ll |\z|\ll t$. By Propositions \ref{T46} and \ref{T76}, the fact that the $\mathcal{Y}_i$s are a geometric series, and Lemmas \ref{L: Minor Arcs bound: Y bracket geometric series}-\ref{L: Minor Arcs bound: Y bracket bound}, (using the same values for $\mathcal{Y}, V, H$), we have:
\begin{align}
    \label{E918}
    \Dy&\ll \sum_{q,\eqref{E91}} \int_{t\ll |\z|\ll t} H(q)^{-n/2}P^{n/2}q\Bo\sum_{\underline{h}\ll H}|T_{\underline{h}}(q,\underline{z})|\Bc^{1/2}d\underline{z}\nonumber\\
    &\ll P^{n+\ve}\sum_{q,\eqref{E91}} \int_{t\ll |\z|\ll t} H(q)^{-n/2}q^2(1+\mathcal{Y}_{0}(q,b_1,q_3,|\z|))^{1/2}d\underline{z}\nonumber\\
    &\ll P^{n+\ve}\sum_{q,\eqref{E91}} t^2 H(R)^{-n/2}R^2(1+\mathcal{Y}_{0}(R,R_1,R_3,t))^{1/2}\nonumber\\
    &\ll P^{n+\ve} \mathcal{R} t^2 H(R)^{-n/2}R^2(R_1+\mathcal{X}_{1}+\mathcal{X}_{2})^{1/2}\nonumber\\
    &\ll P^{n+\ve} R^{5/2} t^2 H(R)^{-n/2}(R+\mathcal{X}_{1}+\mathcal{X}_{2})^{1/2}.
\end{align}
where the $\mathcal{X}_i$s are defined as in \eqref{eq: minor arcs bound: X_1 definition}-\eqref{eq: minor arcs bound: X_2 definition}. Taking logs and recalling the definitions \eqref{eq: minor arcs bound: log H defn (1)}-\eqref{eq: minor arcs bound: X_brac defn (4)} gives us
\[B_P(\phi, \tau, \un{\phi})\leq n+\ve + \frac{5\phi}{2} + 2\tau -\frac{n}{2}\, H + \frac{1}{2}\, \mathcal{X}\_\mathrm{bracket}+ \log_P(C),\]
where $C$ is the implied constant in \eqref{E918}. Hence, we arrive at the following:
\Lb
\label{T918}
Let $n$ be fixed, $\log_P \Dy:= B_P(\phi, \tau, \un{\phi})$, and 
\[B_{PV/P}(\phi,\tau,\phi_3,\phi_4):=n + \frac{5\phi}{2} + 2\tau -\frac{n}{2}\, H + \frac{1}{2}\, \mathcal{X}\_\mathrm{bracket}.\]
Then $B_{PV/P}(\phi,\tau,\phi_3,\phi_4)$ is a continuous, piecewise linear function, and for every $\ve>0$, there is a sufficiently large $P$ such that 
\[B_P(\phi, \tau, \un{\phi}) \leq B_{PV/P}(\phi,\tau,\phi_3,\phi_4) +\ve, \]
for every $\phi\in [0,3/2]$, $\phi_i\in [0,\phi]$, $\phi_1+\phi_2+\phi_3+\phi_4=\phi$, $\tau\in [-5,-\phi-0.75]$.
\Le

\subsection{Averaged van der Corput/Weyl}
We will now find a bound for $B_P(\phi, \tau, \un{\phi})$ using the Averaged van der Corput differencing process discussed in Section \ref{vdc}, followed by one Weyl differencing step as in Section \ref{Weyl}. By Proposition $\ref{T45}$ (upon choosing $N$ to be sufficiently large), we have

\begin{align}
    \label{E97}
    \Dy\ll_{\ve,N} P^{-N}+ \sum_{q,\eqref{E91}}H^{-n/2+1}&P^{n/2-1+\ve}q((HP^2)^{-1}+t)^2 \Big{(}\max_{t\ll|\uz|\ll t}\sum_{|\h|\ll H}|T_{\underline{h}}(q,\underline{z})|\; \Big{)}^{1/2}.
\end{align}
We may now use Proposition \ref{T814} and \eqref{E91} - \eqref{E92} to bound $T_{\underline{h}}(q,\z)$ as follows:
\begin{align}
    \label{eq: minor arcs bound: VDC/Weyl T_h bound: 1}
    |T_{\underline{h}}(q,\z)|\ll R^2&P^{n+\ve} \Big{(} P^{-2}+H^2R^2t^2+R^2P^{-4}+R^{-1}H^2\min\{1,\frac{1}{HtP^2}\}\Big{)}^{(n-\sigma_{\infty}(\h)-2)/4}.
\end{align}
In this subsection, we will choose 
\begin{equation}
    \label{eq: minor arcs bound: H_Weyl defintion}
    H\asymp \max\{R^{1/6},(RtP^2)^{1/5}\}.
\end{equation}
$H$ is chosen so as to simplify the bounds here, as will be evident from our subsequent results. Note  $H=(RtP^2)^{1/5}$ when $t\geq (HP^2)^{-1}$, and $H=R^{1/6}$ when $t\leq (HP^2)^{-1}$.  This is convenient for us since considering these two cases for $t$ separately is natural due to the \ti{min} bracket in \eqref{eq: minor arcs bound: VDC/Weyl T_h bound: 1}.

Before we substitute \eqref{eq: minor arcs bound: VDC/Weyl T_h bound: 1} back into \eqref{E97}, we will simplify this expression significantly using the following Lemma:

\Lb
\label{L: minor arcs bound: VDC/Weyl bracket simplification}
Let $q\asymp R\leq Q$, $Q=P^{3/2}$, $|\z|\asymp t\leq (qQ^{1/2})^{-1}$, and $|\h|\ll H$, where $H$ is defined as in \eqref{eq: minor arcs bound: H_Weyl defintion}. Finally let $\sigma_{\infty}(\h):=s_{\infty}(F_{\h},G_{\h})$. Then
\[T_{\underline{h}}(q,\z)\ll R^2P^{n+\ve}\Big{(}R^{-1}H^2\min\{1,\frac{1}{HtP^2}\}\Big{)}^{(n-\sigma_{\infty}(\h)-2)/4}.\]
\Le

\begin{proof}
Firstly we will assume that $t>(HP^2)^{-1}$. In this case the right-most term simplifies to $H/(RtP^2)$. Before we get into the proof that $H/(RtP^2)$ dominates all other terms, we will show the for our choice of $H$ (see \eqref{eq: minor arcs bound: H_Weyl defintion}), the following is true:
\begin{align}
    \label{eq: minor arcs bound: H<<P^1/4}
    H\ll P^{1/4}.
\end{align}
Indeed, 
\begin{align*}
    H\asymp (RtP^2)^{1/5}\ll Q^{-1/10}P^{2/5}\asymp P^{2/5-3/20}=P^{1/4}.
\end{align*}
This will be useful to us as we attempt to show that $H/(RtP^2)$ dominates all other terms for every value of $t$ and $R$. We now turn to proving this. Going from left to right in the bracket of \eqref{eq: minor arcs bound: VDC/Weyl T_h bound: 1}, we firstly see that 
\[P^{-2}\ll \frac{H}{RtP^2} \quad \Leftrightarrow \quad H\gg Rt.\]
But, we know that $t\leq (RQ^{1/2})^{-1}$, and so $Rt\ll 1$. We certainly have that $H\gg 1$, and so $H\gg Rt$ must be true. Next,
\[H^2R^2t^2\ll \frac{H}{RtP^2} \quad \Leftrightarrow \quad HR^3t^3P^2\ll 1.\]
Using the fact that $H\ll P^{1/4}$ by \eqref{eq: minor arcs bound: H<<P^1/4}, and $Q\asymp P^{3/2}$ and $Rt\ll Q^{-1/2}$ by the assumptions in the Lemma, we see that 
\[HR^3t^3P^2\ll P^{1/4}Q^{-3/2}P^2\asymp P^{9/4}(P^{-3/2})^{3/2}=1,\]
as required. Finally,
\[R^2P^{-4} \ll \frac{H}{RtP^2} \quad \Leftrightarrow \quad H\gg R^3 tP^{-2}.\]
This one has a few more steps. Recall that we are trying to show the dominance of the right term for every $\ut$ and $R$. By our choice of $H$ and the fact that $t\ll (RQ^{1/2})^{-1}$, $R\leq Q$, we have
\begin{align*}
    &R^3 tP^{-2}\ll H=(RtP^2)^{1/5} \;\forall \; t, R \quad
    \Leftrightarrow \quad R^{14/5} t^{4/5}P^{-12/5}\ll 1 \;\forall \; t, R\\
    \Leftrightarrow \quad &\max\{R\}^{7} \max\{t\}^{2}P^{-6}\ll 1,\quad
    \Leftrightarrow \quad Q^{7}(RQ^{-1/2})^{-2}P^{-6}\ll 1\\
    \Leftrightarrow \quad &Q^{4}P^{-6}\ll 1,
    \Leftrightarrow \quad Q\ll P^{3/2},
\end{align*}
which is true. Hence, for our choices of $H$ and $Q$, we have shown that\\
$H^2R^{-1}\min\{1, (HtP^2)^{-1})\}=H/(RtP^2)$ dominates over all other terms in the expression for every $R\leq Q\asymp P^{3/2}$ and $(HP^2)^{-1}\leq t \leq (RQ^{1/2})^{-1}$.

A similar set of arguments can be used in the case that $t<(HP^2)^{-1}$. In this case, we have $H=R^{1/6}$, and 
\[H^2R^{-1}\min\{1, (HtP^2)^{-1})\}=H^2R^{-1}=R^{-2/3}.\]
Again going from left to right in the bracket of \eqref{eq: minor arcs bound: VDC/Weyl T_h bound: 1}:
\[P^{-2}\ll H^2R^{-1}=R^{-2/3} \quad \Leftrightarrow \quad P^{2}\gg R^{2/3}  \quad \Leftrightarrow \quad R\ll P^3,\]
which is true since $R\leq Q\asymp P^{3/2}$. Next,
\[H^2R^2t^2\ll H^2R^{-1} \quad \Leftrightarrow \quad R(Rt)^2\ll 1 \quad \Leftrightarrow \quad RQ^{-1}\ll 1 \quad \Leftrightarrow \quad R\ll Q,\]
which is again true by our assumptions from the Lemma. We used the fact that $Rt\leq Q^{-1/2}$ since $t\leq (RQ^{1/2})^{-1}$. Finally
\[R^2P^{-4} \ll H^2R^{-1}=R^{-2/3} \quad \Leftrightarrow \quad R^{8/3} \ll P^4 \quad \Leftrightarrow \quad R \ll P^{3/2}.\]
This is also true since $R\leq Q\asymp P^{3/2}$. Hence, we have shown that\\
$H^2R^{-1}\min\{1, (HtP^2)^{-1})\}=H^2R^{-1}$ dominates over all other terms in the expression for every $R\leq Q\asymp P^{3/2}$ and $t \leq (HP^2)^{-1}$. This completes the proof of the lemma.
\end{proof}
We could now substitute the results from Lemma \ref{L: minor arcs bound: VDC/Weyl bracket simplification} into \eqref{E97} directly, but the expression is rather complicated so we will instead just focus on the $\h$ sum inside of the integral for now. Our treatment of it will be analogous to the proof of the $\h$ sum bound in Section \ref{hsum}, but it will be a much simpler process this time around. The reason for our choice of $H$ will also become apparent as we deal with this sum. We aim to show the following:
\Lb
\label{L: minor arcs bound: VDC/Weyl H sum bound}
Let $q\asymp R\leq Q$, $Q=P^{3/2}$, $|\z|\asymp t\leq (qQ^{1/2})^{-1}$, and $|\h|\ll H$, where $H$ is defined as in \eqref{eq: minor arcs bound: H_Weyl defintion}. Then
\[\sum_{|\h|\ll H} |T_{\underline{h}}(q,\z)|\ll_n R^2P^{n+\ve}H.\]
In particular, we save a factor of $H^n$ over the trivial bound.
\Le
\begin{proof}
We will again consider the cases when $t\geq (HP^2)^{-1}$ and $t\leq (HP^2)^{-1}$ separately. Starting with $t\geq (HP^2)^{-1}$ first: By Lemma \ref{L: minor arcs bound: VDC/Weyl bracket simplification}, we have
\begin{align}
    \sum_{|\h|\ll H} |T_{\underline{h}}(q,\z)|&\ll R^2P^{n+\ve}\sum_{i=-1}^{n-1}\sum_{\substack{|\h|\ll H \\ \sigma_{\infty}(\h)=i}} \Bo \frac{H}{RtP^2}\Bc^{(n-i-2)/4}\nonumber\\
                              &\ll R^2P^{n+\ve}\max_{-1\leq i\leq n-1} \#\{|\h|\ll H \:|\: \sigma_{\infty}(\h)=i\} \Bo \frac{H}{RtP^2}\Bc^{(n-i-2)/4}\nonumber\\
                        \label{eq: m. arcs bound : H sum bound, t>(HP^2)^-1: 1}
                              &\ll R^2P^{n+\ve}\max_{-1\leq i\leq n-1} H^{n-i-1} \Bo \frac{H}{RtP^2}\Bc^{(n-i-2)/4},
\end{align}
by Lemma \ref{T71}. Recall that when $t\geq (HP^2)^{-1}$, we have $H\asymp (R|t|P^2)^{1/5}$. This value for $H$ has been chosen specifically so that $H=(H/(RtP^2))^{-1/4}$ when $t>(HP^2)^{-1}$. The reason for doing this is so that the product within the \ti{max} bracket in \eqref{eq: m. arcs bound : H sum bound, t>(HP^2)^-1: 1} will become $H$. Indeed, substituting this value for $H$ into \eqref{eq: m. arcs bound : H sum bound, t>(HP^2)^-1: 1} gives
\begin{align*}
    \sum_{|\h|\ll H} |T_{\underline{h}}(q,\z)|&\ll R^2P^{n+\ve}\max_{-1\leq i\leq n-1} (RtP^2)^{(n-i-1)/5} \Bo \frac{1}{(RtP^2)^{4/5}}\Bc^{(n-i-2)/4}\\
                              &= R^2P^{n+\ve}\max_{-1\leq i\leq n-1} (RtP^2)^{(n-i-1)/5} (RtP^2)^{-(n-i-2)/5}\\
                              &=R^2P^{n+\ve}(RtP^2)^{1/5}\\
                              &=R^2P^{n+\ve}H.
\end{align*}
In theory, it would be nice if we could choose $H$ to be even larger, so that we get something smaller than $R^2P^{n+\ve}H$. However, if one chooses $H$ to be larger than this value, then Lemma \ref{L: minor arcs bound: VDC/Weyl bracket simplification} becomes false (in particular, the term $H^2R^2t^2$ dominates when $H>P^{1/4}$). This is therefore the optimal choice for $H$ when $t>(HP^2)^{-1}$. 

The argument in the case the $t\leq (HP^2)^{-1}$ is almost identical. Recall that when $t\leq (HP^2)^{-1}$, we have $H\asymp R^{1/6}$. By Lemma \ref{L: minor arcs bound: VDC/Weyl bracket simplification}, we have
\begin{align*}
    \sum_{|\h|\ll H} |T_{\underline{h}}(q,\z)|&\ll R^2P^{n+\ve}\sum_{i=-1}^{n-1}\sum_{\substack{|\h|\ll H \\ \sigma_{\infty}(\h)=i}} \big{(} H^2R^{-1}\big{)}^{(n-i-2)/4}\nonumber\\
                              &\ll R^2P^{n+\ve}\max_{-1\leq i\leq n-1} \#\{|\h|\ll H \:|\: \sigma_{\infty}(\h)=i\} \, R^{-(n-i-2)/6}\nonumber\\
                              &\ll R^2P^{n+\ve}\max_{-1\leq i\leq n-1} H^{n-i-1}R^{-(n-i-2)/6}\nonumber\\
                              &\ll_n R^2P^{n+\ve}R^{1/6}\ll R^2P^{n+\ve}H,
\end{align*}
by Lemma \ref{T71}, and by the fact that when $t\leq (HP^2)^{-1}$, we have $H\asymp R^{1/6}$. This value for $H$ has again been chosen specifically so that $H^{n-i-1}R^{-(n-i-1)/6}=1$ for every $i$. when $t>(HP^2)^{-1}$. For the same reasons as before, we cannot choose $H$ to be larger than this without causing other issues, and so this makes our choice of $H$ in \eqref{eq: minor arcs bound: H_Weyl defintion} optimal for our situation.
\end{proof}
Substituting the result of Lemma \ref{L: minor arcs bound: VDC/Weyl H sum bound} back into \eqref{E97} gives
\begin{align*}
    \Dy&\ll P^{n-1+\ve}\sum_{q,\eqref{E91}}H^{-n/2+3/2}R^2((HP^2)^{-1}+t)^2.
\end{align*}
Finally, we split the $R$ sum into its cube-free and cube-full components, and use Lemma \ref{T95} as follows:

\begin{align}
    \Dy&\ll P^{n-1+\ve} \sum_{b_1=R_1}^{2R_1}\sum_{b_2=R_2}^{2R_2}\sum_{b_3=R_3}^{2R_3}\sum_{q_4=R_4}^{2R_4} R^2 H(R, t)^{(3-n)/2}((H(R,t)P^{2})^{-1}+ t)^2\nonumber\\
    \label{eq: minor arcs bound: Dy Avdc/Weyl final bound}
       &\ll P^{n-1+\ve}R^3 R_2^{-1/2}R_3^{-2/3}R_4^{-3/4}H(R, t)^{(3-n)/2}((H(R,t)P^{2})^{-1}+ t)^2\nonumber\\
       &\ll P^{n-1+\ve}R^3 R_3^{-2/3}R_4^{-3/4}H(R, t)^{(3-n)/2}((H(R,t)P^{2})^{-1}+ t)^2.
\end{align}
Therefore, upon setting $R:=P^{\phi}$, $R_i:=P^{\phi_i}$, $t:=P^{\tau}$ and (recall \eqref{eq: minor arcs bound: H_Weyl defintion})
\begin{align}
    \label{eq: minor arcs bound: VDC/Weyl H defn}
    \hat{H}\_\mathrm{Weyl}(\phi,\tau)&:=\max\Co\frac{\phi}{6}, \frac{2+\phi+\tau}{5}\Cc,\\
    \label{eq: minor arcs bound: tau_brac defn VDC}
    \tau\_\mathrm{brac}(\phi,\tau)&;= \max\{-2-\hat{H}\_\mathrm{Weyl}(\phi,\tau),\, \tau\},
\end{align} 
we have:
\begin{align*}
    B_P(\phi, \tau, \un{\phi}) \leq n-1+\ve + 3\phi -\frac{2\phi_3}{3} -&\frac{3\phi_4}{4}+ \log_P(C)+\frac{(3-n)}{2} \hat{H}\_\mathrm{Weyl}(\phi,\tau) + 2\tau\_ \mathrm{brac}(\phi,\tau),
\end{align*}
where C is the implied constant in \eqref{eq: minor arcs bound: Dy Avdc/Weyl final bound}. Hence, if $P$ is chosen to be sufficiently large, we may absorb $\log_P(C)$ into $\ve$, giving us the following:
\Lb
\label{T1101}
Let $n$ be fixed, and 
\[B_{AV/W}(\phi,\tau,\phi_3,\phi_4):=n-1+ 3\phi -\frac{2\phi_3}{3} -\frac{3\phi_4}{4}+\frac{(3-n)}{2} \hat{H}\_\mathrm{Weyl}(\phi,\tau) + 2\tau\_ \mathrm{brac}(\phi,\tau).\]
Then $B_{AV/W}(\phi,\tau,\phi_3,\phi_4)$ is a continuous, piecewise linear function, and for every $\ve>0$, there is a sufficiently large $P$ such that 
\[B_P(\phi, \tau, \un{\phi}) \leq B_{AV/W}(\phi,\tau,\phi_3,\phi_4) +\ve, \]
for every $\phi\in [0,3/2]$, $\phi_i\in [0,\phi]$, $\phi_1+\phi_2+\phi_3+\phi_4=\phi$, $\tau\in [-5,-\phi-0.75]$.
\Le
\subsubsection{Explaining the Choice of Q}
\label{S: Explaining Q}
As an aside, we will briefly explain our choice of $Q\asymp P^{3/2}$, as promised in Section \ref{C: Initial setup}. We see in the proof of Lemma \ref{L: minor arcs bound: VDC/Weyl bracket simplification}, that the optimal choice for $Q$ is $P^{3/2}$. In particular, if we choose any other value for $Q$, then we cannot simplify the Weyl bound to such a large extent. We normally optimise our choice for $Q$ based on our main bound, which in this case is the averaged van der Corput/Poisson bound. This value for $Q$ turns out to be 
\[Q\asymp P^{4(n+3)/3(n-2)},\]
which is the choice of $Q$ that guarantees $HP^2|\z|\ll 1$ for every $\z$ (optimising our $V$ term), where $H$ and $V$ are defined as in \eqref{eq: minor arcs bound, Poisson H definition} - \eqref{eq: minor arcs bound, Poisson V definition}. In the range of $n$ that we are considering, this value is largest when $n=39$, giving us $Q\asymp P^{1.5135\cdots}$, which is very close to the optimal choice for the van der Corput/Weyl bounds. In the end, the authors chose $Q\asymp P^{3/2}$ because it is simpler and it makes the van der Corput/Weyl bounds significantly easier to work with. Most importantly, this choice does not cause any issues for our Poisson bounds, since it is ``almost" optimal. 
\subsection{Pointwise van der Corput/Weyl}
In this subsection, we will find a bound for $B_P(\phi, \tau, \un{\phi})$ by using Pointwise van der Corput differencing, followed by one Weyl step. We start by applying Lemma \ref{T46} to $\Dy$:
\[\Dy\ll \sum_{q,\eqref{E91}} \int_{t\ll |\z|\ll t} H^{-n/2}P^{n/2}q\Bo\sum_{\underline{h}\ll H}|T_{\underline{h}}(q,\underline{z})|\Bc^{1/2}d\uz. \]
Upon setting $H:=\max\{q^{1/6}, (qtP^2)^{1/5}\}$ again, we may use Lemma \ref{L: minor arcs bound: VDC/Weyl H sum bound} and Proposition \ref{T814}) to conclude that
\begin{align}
    \label{E1103}
    \Dy&\ll P^{n+\ve}\sum_{q,\eqref{E91}}\int_{t\ll |\z|\ll t} H(q,t)^{(1-n)/2}q^2 d\uz \nonumber \\
       &\ll P^{n+\ve}R^3 R_3^{-2/3}R_4^{-3/4} t^2 H(R,t)^{(1-n)/2}.
\end{align}
Hence upon recalling \eqref{eq: minor arcs bound: VDC/Weyl H defn}, we have
\begin{align*}
    B_P(\phi, \tau, \un{\phi}) \leq n+\ve + 3\phi -\frac{2\phi_3}{3} -&\frac{3\phi_4}{4}+ 2\tau + \log_P(C) +\frac{1-n}{2} \hat{H}\_\mathrm{Weyl}(\phi,\tau),
\end{align*}
where C is the implied constant in \eqref{E1103}. Therefore, if $P$ is chosen to be sufficiently large, we may absorb $\log_P(C)$ into $\ve$, giving us the following:
\Lb
\label{T1104}
Let $n$ be fixed, and 
\[B_{PV/W}(\phi,\tau,\phi_3,\phi_4):=n+ 3\phi + 2\tau -\frac{2\phi_3}{3} -\frac{3\phi_4}{4}+\frac{1-n}{2} \hat{H}\_\mathrm{Weyl}(\phi,\tau).\]
Then $B_{PV/W}(\phi,\tau,\phi_3,\phi_4)$ is a continuous, piecewise linear function, and for every $\ve>0$, there is a sufficiently large $P$ such that 
\[B_P(\phi, \tau, \un{\phi}) \leq B_{PV/W}(\phi,\tau,\phi_3,\phi_4) +\ve, \]
for every $\phi\in [0,3/2]$, $\phi_i\in [0,\phi]$, $\phi_1+\phi_2+\phi_3+\phi_4=\phi$, $\tau\in [-5,-\phi-0.75]$.
\Le

\subsection{Weyl}
In this subsection, we will find a bound for $B_P(\phi, \tau, \un{\phi})$ by using Weyl differencing twice. We start by applying Proposition \ref{T815} to $\Dy$:
\begin{align*}
    \Dy\ll P^{n+\ve}\sum_{q,\eqref{E91}} &\starsum_{\ua} \int_{t\ll |\z|\ll t}\Big{(} P^{-4}+q^2|\underline{z}|^2+q^2P^{-6}+q^{-1}\min\{1,\frac{1}{|\underline{z}|P^3}\}\Big{)}^{(n-1)/16}d\underline{z}.
\end{align*}
Firstly, it is easy to use \eqref{E91}-\eqref{E93} to check that 
\[\max\{P^{-4}, q^2P^{-6}\}\leq q^{-1}\min\{1, (|\z|P^3)\}^{-1}.\]
Hence 
\begin{align}
\label{E1106}
    \Dy&\ll P^{n+\ve}\sum_{q,\eqref{E91}} \starsum_{\ua}  \int_{t\ll |\z|\ll t} \Big{(} q^2|\underline{z}|^2+q^{-1}\min\{1,\frac{1}{|\underline{z}|P^3}\}\Big{)}^{(n-1)/16}d\underline{z}\nonumber\\
       &\ll P^{n+\ve}\sum_{q,\eqref{E91}} q^2 t^2 \Big{(} q^2t^2+q^{-1}\min\{1,\frac{1}{tP^3}\}\Big{)}^{(n-1)/16}\nonumber\\
       &\ll P^{n+\ve}\sum_{q,\eqref{E91}} q^2 t^2 \Big{(} q^2t^2+q^{-1}\min\{1,\frac{1}{tP^3}\}\Big{)}^{(n-1)/16}\nonumber\\
       &\ll P^{n+\ve}R^3 R_3^{-2/3} t^2 \Big{(} R^2t^2+R^{-1}\min\{1,\frac{1}{tP^3}\}\Big{)}^{(n-1)/16}.
\end{align}
As usual, we are interested in $\log_P(\Dy)$ since this will be piecewise linear. The bound above gives
\begin{align*}
    B_P(\phi, \tau, \un{\phi})\leq n+\ve +3\phi + 2\tau-\frac{2\phi_3}{3} -\frac{3\phi_4}{4}  &+  \log_P(C)\\
    &+\frac{n-1}{16}\,\max\Co 2\phi+2\tau, \, -\phi+\min \{0,-3-\tau\}\Cc,
\end{align*}
where $\log_P(C)$ is the implied constant in \eqref{E1106}. Therefore, upon setting
\begin{align}
\label{eq: minor arcs bound: Weyl_brac}
    \mathrm{Weyl}\_\mathrm{brac}(\phi,\tau):=\max\Co 2\phi+2\tau, \, -\phi+\min \{0,-3-\tau\}\Cc,
\end{align}we arrive at the following bound for $B_P$:
\Lb
\label{T1107}
Let $n$ be fixed, $\log_P \Dy:= B_P(\phi, \tau, \un{\phi})$, and 
\[B_{Weyl}(\phi,\tau,\phi_3,\phi_4):=n +3\phi + 2\tau -\frac{2\phi_3}{3} - \frac{3\phi_4}{4} + \frac{n-1}{16}\,\mathrm{Weyl}\_\mathrm{brac}(\phi,\tau).\]
Then $B_{Weyl}(\phi,\tau,\phi_3,\phi_4)$ is a continuous, piecewise linear function, and for every $\ve>0$, there is a sufficiently large $P$ such that 
\[B_P(\phi, \tau, \un{\phi}) \leq B_{Weyl}(\phi,\tau,\phi_3,\phi_4) +\ve, \]
for every $\phi\in [0,3/2]$, $\phi_i\in [0,\phi]$, $\phi_1+\phi_2+\phi_3+\phi_4=\phi$, $\tau\in [-5,-\phi-0.75]$.
\Le

\subsection{Proof of Proposition \ref{P02}}
Recall that our ultimate goal is to show that 
\[S_{\mathfrak{m}}\ll P^{n-6-\delta},\]
for some $\delta>0$, for every $n\geq 39$.  This is equivalent to having
\[\log_P(S_{\mathfrak{m}})< n-6.\]
We assume that $\rho$ is chosen sufficiently small to facilitate average van der Corput differencing bounds. We may now use all of the previous subsections to bound $\log_P(S_{\mathfrak{m}})$ by a continuous, piecewise linear function in three variables: By \eqref{E94}, we have
\[\log_P(S_{\mathfrak{m}})\leq \log_P(c_1) + \ve + \max_{\substack{\phi,\un{\phi},\tau \\ \eqref{E92}, \eqref{E93}, \, \tau> P^{-5}}} \{ B_P(\phi, \tau, \un{\phi}) , n-7\},\]
where $c_1$ is the implied constant. We clearly have that $\log_P(c_1) + \ve +n-7\leq n-6-\ve$ for sufficiently large $P$, so we will assume that this is the case. Hence by Lemmas \ref{T1101}-\ref{T918}, we have
\begin{align}
    \label{eq: minor arcs bound: final function}
    \log_P(S_{\mathfrak{m}})\leq \ve+\max\Co \min_{(\phi,\tau,\phi_3,\phi_4) \in D_1\cup D_2} \Co &B_{AV/P}(\phi,\tau,\phi_3,\phi_4), B_{PV/P}(\phi,\tau,\phi_3,\phi_4),\nonumber\\
    &B_{AV/W}(\phi,\tau,\phi_3,\phi_4),\, B_{PV/W}(\phi,\tau,\phi_3,\phi_4), \nonumber\\
    &B_{Weyl}(\phi,\tau,\phi_3,\phi_4)\Cc , n-6-2\ve\Cc,
\end{align}
where
\begin{align*}
    D_1&:=\{(\phi,\tau,\phi_3,\phi_4)\in \R^3 \: : \: \Delta\leq \phi \leq 3/2, \:\: 0\leq \phi_3 \leq \phi, \:\: -5 \leq \tau \leq -\phi-3/4 \}\\
    D_2&:=\{(\phi,\tau,\phi_3,\phi_4)\in \R^3 \: : \: 0\leq \phi \leq \Delta, \:\: 0\leq \phi_3 \leq \phi, \:\: -3+\Delta \leq \tau \leq -\phi-3/4 \}.
\end{align*}
Since $D_1$ and $D_2$ are convex polytopes and the function which we have bounded $\log_P(S_{\mathfrak{m}})$ is continuous and piecewise linear for every $n\in \N$. Each region on which this function is linear is a convex  polytope. It is well known that extremum value of such a function must be taken at a vertex of one of these polytopes. Therefore, one may numerically compute the exact maxima in \eqref{eq: minor arcs bound: final function}. We compute this maxima two different ways and check that both values coincide:

 The first way is to use an inbuilt Min-Max function in Mathematica that compares the two bounds. This algorithm can be found in Appendix A. An executable version of code can also be found in the first author's Github page \cite{Northey1}.
 
 We have also verified this using an open source python based algorithm (this can be found in \cite{Northey1}). 
 
After taking $\ve'=0.0001$ (see \eqref{eq: minor arcs bound: log H defn (1)}), $\Delta=1/7-0.001$, both numerical verifications proves that
\[\log_P(S_{\mathfrak{m}})\leq n-6.00185\]
for every $(\phi,\tau,\phi_3,\phi_4)\in D_1\cup D_2$, provided that $39\leq n \leq 48$. The limiting case is when $n=39$, $\phi=3/2$, $\tau=-2.25$, $\phi_2=\phi_3=\phi_4=0$. When $n\geq 49$, we may instead refer to Birch \cite{Birch61}. 

\section{Major Arcs}
\label{MA}
Finally, we will complete the proof of Theorems \ref{thm:main thm}-\ref{thm:main thm 1} by showing that
\[S_{\mathfrak{M}}=C_X P^{n-6}+O(P^{n-6-\delta}),\]
where
\[S_{\mathfrak{M}}=\sum_{q\leq P^{\Delta}}\starsum_{\ua}^q\int_{|\underline{z}|<P^{-3+\Delta}} S(\underline{a}/q+\underline{z}) d\underline{z},\]
and $C_X$ is a product of local densities. Let 
\[\mathfrak{S}(R):=\sum_{q=1}^R q^{-n}\starsum_{\ua}^q S_{\underline{a},q},\quad \mathfrak{J}(R):=\int_{|\underline{z}|<R}\int_{\mathbb{R}^n}\omega(\underline{x})e(z_1F(\underline{x})+z_2G(\underline{x}))\: d\underline{x}d\underline{z},\]
where
\[S_{\underline{a},q}:= \sum_{\x  \bmod{q}} e_q(a_1F(\x)+a_2G(\x)),\]
and
\[\mathfrak{S}:=\lim_{R\rightarrow\infty} \mathfrak{S}(R),\quad \mathfrak{J}=\lim_{R\rightarrow \infty} \mathfrak{J}(R),\]
if the limits exist. In the following let $\sigma$ denote the dimension of the singular locus of the complete intersection $X$. For our application here we only need to establish the $ \sigma=-1$ case. However, a general version is equally straight-forward.
We will start by showing the following:
\begin{lemma}
\label{T121}
Assume that $n-\sigma\geq 34$ and that $\mathfrak{S}$ is absolutely convergent, satisfying
\[\mathfrak{S}(R)=\mathfrak{S}+O_{\phi}(R^{-\phi}).\]
Then provided that we have $\Delta\in (0,1/7)$, 
\[S_{\mathfrak{M}}=\mathfrak{S}\mathfrak{J}P^{n-6}+O_{\phi}(P^{n-6-\delta}).\]
\end{lemma}
Following the proof found in \cite{Browning-Heath-Brown09}, the first step towards proving this lemma is to show that
\begin{equation}
    \label{E122}
    S(\underline{\alpha})=q^{-n}P^nS_{\underline{a},q}I(\underline{z}P^3)+O(P^{n-1+2\Delta})
\end{equation}
where 
\[I(\underline{t}):=\int_{\mathbb{R}^n}\omega(\underline{x})e(t_1F(\underline{x})+t_2G(\underline{x})) d\underline{x},\]
for $\underline{t}\in\mathbb{R}^2$. In order to achieve this, we need to be able to separate $S(\underline{\alpha})$'s dependence on $\underline{a}$ from its dependence on $\underline{z}$. Write $\underline{x}=\underline{u}+q\underline{v},$ where $\underline{u}$ runs over the complete set of residues modulo $q$ and recall that $\underline{\alpha}=\underline{a}/q+\underline{z}$. Then
\begin{equation}
    \label{E123}
    S(\underline{\alpha})=\sum_{\underline{u} \bmod{q}} e_q(a_1F(\underline{u})+a_2G(\underline{u}))\sum_{v\in\mathbb{Z}} \Phi_{\underline{u}}(\underline{v}),
\end{equation}
where
\[\Phi_{\underline{u}}(\underline{v})=\omega\Big{(}\frac{\underline{u}+q\underline{v}}{P}\Big{)}e(z_1F(\underline{u}+q\underline{v})+z_2G(\underline{u}+q\underline{v})).\]
In order to have it so that $\underline{a}$ and $\underline{z}$ are independent from each other, we will replace our $\underline{v}$ sum with a crude integral estimate which has no dependence on $\underline{u}$. In particular, we can use the fact that $\Phi_{\underline{u}}(\underline{v}+\underline{x})=\Phi_{\underline{u}}(\underline{v})+O(\max_{y\in[0,1]^n}|\nabla \Phi_{\underline{u}}(\underline{v}+\underline{y})|)$ for any $\underline{x}\in[0,1]^n$, to conclude the following:
\begin{align*}
\Big{|}\int_{\mathbb{R}^n}\Phi_{\underline{u}}(\underline{v}) d\underline{v}-\sum_{\underline{v}\in\mathbb{Z}^n}\Phi_{\underline{u}}(\underline{v})\Big{|}&\leq \meas(\mathcal{S}) \max_{\underline{\hat{v}}\in \mathcal{S}\cap \mathbb{Z}}\Big{|}\int_{ \underline{\hat{v}}+[0,1]^n}\Phi_{\underline{u}}(\underline{v}) d\underline{v}-\Phi_{\underline{u}}(\underline{\hat{v}})\Big{|}\\
&\ll \meas(\mathcal{S}) \max_{\underline{\hat{v}}\in\mathcal{S}\cap \mathbb{Z}} \max_{y\in[0,1]^n}|\nabla \Phi_{\underline{u}}(\underline{\hat{v}}+\underline{y})|.\\
\end{align*}
We note that 
\[\max_{\underline{\hat{v}}\in\mathcal{S}\cap \mathbb{Z}} \max_{y\in[0,1]^n}|\nabla \Phi_{\underline{u}}(\underline{\hat{v}}+\underline{y})|=\max_{\underline{\hat{v}}\in \mathcal{S}}|\nabla \Phi_{\underline{u}}(\underline{\hat{v}})|,\]
and so by the Liebniz rule we have
\begin{align*}
\Big{|}\int_{\mathbb{R}^n}\Phi_{\underline{u}}(\underline{v}) d\underline{v}-\sum_{\underline{v}\in\mathbb{Z}^n}\Phi_{\underline{u}}(\underline{v})\Big{|}&\ll P^n q^{-n}(q/P+q|\underline{z}|P^2)\\
&= P^{n-1}q^{1-n}+|\underline{z}|P^{n+2}q^{1-n},
\end{align*}
since $\mathcal{S}$ is an n-dimensional cube with sides of order $1+P/q\leq 2P/q.$ Hence, on setting\\
$P\underline{x}=\underline{u}+q\underline{v},$ we arrive at the following expression for $\sum_{\underline{v}} \Phi_{\underline{u}}(\underline{v})$:
\[\sum_{\underline{v}\in\mathbb{Z}^n} \Phi_{\underline{u}}(\underline{v})=\frac{P^n}{q^n}\int_{\mathbb{R}^n}\omega(\underline{x})e(z_1P^3F(\underline{x})+z_2P^3G(\underline{x})) d\underline{x}+O(P^{n-1}q^{1-n}+|\underline{z}|P^{n+2}q^{1-n}).\]
We can therefore conclude that
\begin{equation}
    \label{E63}
    S(\underline{\alpha})=P^n q^{-n} S_{\underline{a},q}I(\underline{z}P^3)+O(P^{n-1}q+|\underline{z}|P^{n+2}q)
\end{equation}

by \eqref{E123}. Since $|\underline{z}|\leq P^{-3+\Delta}$ and $q\leq P^{\Delta}$, we can now conclude that \eqref{E122} is indeed true. Furthermore, by substituting \eqref{E122} into $S_\mathfrak{M}$ and -- for the error term -- noting that the major arcs have measure $O(P^{-6+5\Delta})$ ($P^{-6+2\Delta}$ from the integrals, $P^{3\Delta}$ from the sums), we conclude that
\begin{equation}
    \label{E64}
    S_{\mathfrak{M}}=P^{n-6}\mathfrak{S}(P^{\Delta})\mathfrak{J}(P^{\Delta})+O(P^{n-7+7\Delta}).
\end{equation}
Since we have assumed $\mathfrak{S}(R)=\mathfrak{S}+O_\phi(R^{-\phi})$ for some $\phi>0$, we can replace $\mathfrak{S}(P^{\Delta})$ with $\mathfrak{S}$ leading us to
\begin{equation}
    \label{E65}
    S_{\mathfrak{M}}=P^{n-6}\mathfrak{S}\mathfrak{J}(P^{\Delta})+O_\phi(P^{n-7+7\Delta}+P^{n-6-\Delta\phi}).
\end{equation}
We will prove that this assumption is true in the next section. We now aim to show that we can replace $\mathfrak{J}(P^{\Delta})$ with $\mathfrak{J}$. In order to do this, we need $\mathfrak{J}$ to exist, and $|\mathfrak{J}-\mathfrak{J}(P^{\Delta})|$ to be sufficiently small. Now, it is easy to see that 
\[\mathfrak{J}-\mathfrak{J}(R)=\int_{t\geq R} I(\underline{t}) d\underline{t},\]
and so this motivates us to find a bound for the size of $I(\underline{t}).$ We will show the following:
\begin{lemma}
\label{T1240}
Let 
\[\sigma:=\dim \Sing_{\mathbb{C}}(X_F,X_G).\] Then
\[I(\underline{t})\ll \min\{1,|\underline{t}|^{\sigma+1-n/16+\ve}\}.\]
\end{lemma}

\begin{proof}
We will again follow the same procedure as in \cite{Browning-Heath-Brown09}. $I(\underline{t})\ll 1$ is trivial since $|I(\underline{t})|\leq \meas(\mathcal{S})$ for every $\underline{t}$. For the second estimate, we can assume $|\underline{t}|>1$. Then on taking $\underline{a}=0$, $q=1$ in \eqref{E63} we get
\[S(\underline{\alpha})=P^{n}O(|\underline{\alpha}|P^3)+O((|\underline{\alpha}|P^3+1)P^{n-1}),\]
for any $P\geq 1$. Likewise, for $|\underline{\alpha}|<P^{-1}$, we can also use Proposition \ref{T815} with $\underline{a}=\underline{0}$, $q=1$, to conclude that 
\[S(\underline{\alpha})\ll P^{n+\ve}(|\underline{\alpha}|P^3)^{(\sigma+1-n)/16}.\]
Hence for such $\alpha$, we may set $\underline{t}=\underline{\alpha}P^3$ and combine these estimates to get
\[I(\underline{t})\ll |\underline{t}|^{(\sigma+1-n)/16}P^{\ve}+|\underline{t}|P^{-1},\]
when $1<|\underline{t}|<P^{2}.$ Finally, we note that this is true for every $P\geq 1$ and $I(\underline{t})$ does not depend on $P$ at all. Hence we can choose $P=|\underline{t}|^{(16+n-\sigma-1)/16}$ to reach our second estimate of $I(\underline{t})$. 
\end{proof}
We can now use Lemma \ref{T1240} to conclude that
\begin{align*}
    \mathfrak{J}-\mathfrak{J}(R)=\int_{|\underline{t}|\geq R} I(\underline{t}) d\underline{t}&\ll \int_R^{\infty}\int_R^{\infty} \min\{1,|\underline{t}|^{(\sigma+1-n)/16+\ve}\} d\underline{t}
    \ll R^{(33+\sigma-n)/16+\ve}.
\end{align*}
For $n-\sigma\geq 34$, this shows that $\mathfrak{J}$ is absolutely convergent. Finally, replacing $\mathfrak{J}(P^{\Delta})$ by $\mathfrak{J}$ in \eqref{E65} gives us
\[S_\mathfrak{M}=\mathfrak{S}\mathfrak{J}P^{n-6}+O_\phi(P^{n-7+7\Delta}+P^{n-6-\Delta\phi}+P^{n-6-\Delta/16+\ve}),\]
which is permissible for Lemma \ref{T121} provided that $\Delta\in (0,1/7)$, $\phi>0$, and $\ve>0$ is taken to be sufficiently small.
\subsection{Convergence of the singular series}

Finally we turn to the issue of showing that the singular series
\[\sum_{q=1}^{\infty} q^{-n} \starsum_{\ua} S_{\underline{a},q}\]
 converges absolutely, and obeys the assumption made in Lemma \ref{T121}. In particular, we will show the following:
 \Tb
\label{T: Major Arcs: Singular series converges + asymptotic formula} 
 Assume $n-\sigma\geq 35$. Then $\mathfrak{S}$ is absolutely convergent. Furthermore, there is some $\phi\>0$ such that
 \[\mathfrak{S}(R)=\mathfrak{S} +O_{\phi}(R^{-\phi}).\]
 \Te
To see that $\mathfrak{S}$ converges for $n-\sigma\geq 35$, we will again adopt the approach of Browning and Heath Brown in \cite{Browning-Heath-Brown09}. We start by noting that 
\[\mathfrak{S}=q^{-n} \starsum_{\ua}^{q} S_{\underline{a},q}\]
is a multiplicative function of $q$, and so it follows that $\mathfrak{S}$ is absolutely convergent if and only if $\prod_p (1+\sum_{k=1}^{\infty} a_p(k))$ is, where
\[a_p(k):=p^{-kn} \starsum_{\ua}^{p^k} |S_{\underline{a},p^k}|.\]
But by taking logs, this is equivalent to $\sum_p\sum_{k=1}^{\infty} a_p(k)$ converging. Now by Proposition \ref{T815} with $\underline{a}=\underline{0}$, $q=p^k$, $|\underline{z}|<P^{-3+\Delta}$, $\omega=\chi$, we have that
\begin{equation}
	\label{E66}
a_p(k)\ll p^{k(2+(\sigma+1)/16-n/16)+\ve},
\end{equation}
for any $k\geq 1$, and so this enables us to establish that $\mathfrak{S}$ converges absolutely provided that $n-\sigma\geq 50$. We can use \eqref{E66} far more effectively than this if we are more careful: We will assume that $n-\sigma\geq 35$ from now on. Then by \eqref{E66}, we have
\[\sum_p\sum_{k\geq 16} a_p(k)\ll \sum_{p} p^{33+\sigma-n+\ve}< \sum_{m=1}^{\infty} m^{-2+\ve}\ll 1,\]
assuming $\ve>0$ is sufficiently small. We now need to show that $\sum_p \sum_{1\leq k \leq 15}$ also converges. For $2\leq k \leq 15$, we will use \cite[Lemma 25]{Browning-Heath-Brown09}. This shows that 
\[S_{\underline{a},p^k}\ll_k p^{(k-1)n+s_p(a_1F+a_2G)+1}.\]
Hence
\[\sum_p\sum_{k=2}^{15} a_p(k)\ll \sum_{p}\sum_{k=2}^{15}p^{k(2-n)}p^{(k-1)n+s_p(a_1F+a_2G)+1}=\sum_{p}\sum_{k=2}^{15}p^{2k+1-n+s_p(a_1F+a_2G)}.\]
But by Lemma \ref{P: background on a pair of quadrics: sigma'<= sigma+1}, we have $s_p(a_1F+a_2G)\leq s_p(F,G)+1$. Furthermore since $F$ and $G$ are fixed, $s_p(F,G)=\sigma$ for all but finitely many primes, and so by increasing the size of the implicit multiplicative constant if necessary, we have that
\[\sum_{p}\sum_{k=2}^{15}p^{2k+2-n+\sigma}\ll \sum_p p^{32-n+\sigma}\ll 1,\]
since we have assumed $n-\sigma\geq 35$.

All that is left to check is $k=1$. By Lemma 7 in \cite{Browning-Heath-Brown09}, we have
\[\sum_p a_p(1)\ll \sum_p p^{2-n/2+(s_p(a_1F+a_2G)+1)/2}\ll \sum_p p^{3-n/2+\sigma/2}\ll 1.\]
This enables us to establish Theorem \ref{T: Major Arcs: Singular series converges + asymptotic formula}. Finally, we will follow the approach used in \cite{Marmon_Vishe} to prove that there exists some $\phi>0$ such that
\[\mathfrak{S}(R)=\mathfrak{S}+O_\phi(R^{-\phi}).\]
We will continue to work under the assumption that $n-\sigma\geq 35$. Firstly let
\[S_q:= \starsum_{\ua}^q\sum_{\x}^q e_q(a_1F(\x)+a_2G(\x)).\]
Then, we have
\begin{equation}
\label{E67}
    |\mathfrak{S}-\mathfrak{S}(R)|\leq \sum_{q\geq R}q^{-n} |S_q|.
\end{equation}
We will split $q$ into several of its multiplicative components and bound each component separately. Let

\[b_i:=\prod_{p^i || q} p^i, \quad q_i:=\prod_{\substack{p^e || q\\ e\geq i}} p^e.\]
Then $q=q_k\prod_{i=1}^{k-1} b_i$ for every $k$ (e.g. $q=b_1b_2q_3$). Recall that by Lemma \ref{T95}, we have the following for any $R_1,\cdots,R_k>0$:
\begin{align}
    \label{eq: Major Arcs: Dyadic q_i bound}
    \sum_{\substack{b_1\sim R_1, \cdots, b_{k-1}\sim R_{k-1}\\ q_k\sim R_k }} 1 \ll \prod_{i=1}^k R_i^{1/i}.
\end{align}
We will use $k=16$. Now 
\[|S_q|\leq |S_{q_{16}}|\prod_{i=1}^{15} |S_{b_i}|.\]
We will bound each of these in turn: 
\begin{equation*}
  |S_{q_{16}}|\ll q_{16}^{(15n+\sigma+1)/16+\ve}  
\end{equation*}
by Proposition \ref{T815}. For $b_3,\cdots, b_{15}$, we split $b_k$ into prime powers and use Lemma 25 from \cite{Browning-Heath-Brown09}:
\begin{equation*}
  |S_{p^k}|\ll \starsum_{\ua}^{p^k} p^{(k-1)n+s_p(a_1F+a_2G)+1}\ll p^{(k-1)n+\sigma+2+2k}
\end{equation*}
for $p\gg 1$. Hence for $k\in\{3,\cdots,15\}$,
\begin{equation*}
  |S_{b_k}|\ll b_k^{2+((k-1)n+\sigma+2)/k}.
\end{equation*}
Finally for $b_1,b_2$, we use Lemma 7 from \cite{Browning-Heath-Brown09}. By following the same argument as for $S_{b_3},\cdots, S_{b_15}$, we get
\begin{equation*}
  |S_{b_k}|\ll b_k^{2+(n+\sigma+2)/2},
\end{equation*}
for $k\in\{1,2\}$. Hence
\begin{equation*}
  |S_{q}|\ll q^{2+\ve}(b_1b_2)^{(n+\sigma+2)/2}b_3^{(2n+\sigma+2)/3}\cdots b_{15}^{(14n+\sigma+2)/15},
\end{equation*}
or equivalently
\begin{equation*}
  |S_{q}|\ll \frac{q^{2+n+\ve}}{(b_1b_2)^{(m-1)/2}b_3^{(m-1)/3}\cdots b_{15}^{(m-1)/15}q_{16}^{m/16}},
\end{equation*}
where $m=n-\sigma-1$. Therefore, by \eqref{E67}, we have

\begin{align*}
  |\mathfrak{S}-\mathfrak{S}(R)|&\ll \sum_{b_1\cdots b_{15}q_{16}\geq R} (b_1b_2)^{2+\ve-(m-1)/2}b_3^{2+\ve-(m-1)/3}\cdots b_{15}^{2+\ve-(m-1)/15}q_{16}^{2+\ve-m/16}\\
                                &\ll \sum_{b_1\cdots b_{15}q_{16}\geq R} (b_1b_2)^{(5+\ve-m)/2}b_3^{(7+\ve-m)/3}\cdots b_{15}^{(31+\ve-m)/15}q_{16}^{(32+\ve-m)/16}.\\
\end{align*}
When $m\geq 34$, we clearly have 
\begin{align*}
  |\mathfrak{S}-\mathfrak{S}(R)|&\ll \sum_{b_1\cdots b_{15}q_{16}\geq R} (b_1b_2)^{-29/2+\ve}b_3^{-27/3+\ve}\cdots b_{15}^{-3/15+\ve}q_{16}^{-2/16+\ve}\\
                                &\ll R^{-1/16+2\ve} \sum_{b_1\cdots b_{15}q_{16}\geq R} (b_1b_2)^{-1-\ve}b_3^{-1/3-\ve}\cdots b_{15}^{-1/15-\ve}q_{16}^{-1/16-\ve}\\
                                & < R^{-1/16+2\ve} \sum_{b_1,\cdots, b_{15},q_{16}= 1}^{\infty} (b_1b_2)^{-1-\ve}b_3^{-1/3-\ve}\cdots b_{15}^{-1/15-\ve}q_{16}^{-1/16-\ve},
\end{align*}
and this sum converges by \eqref{eq: Major Arcs: Dyadic q_i bound}. Hence, we conclude that
\[\mathfrak{S}=\mathfrak{S}(R)+O(R^{-\phi}),\]
where $\phi=1/16-\ve$, provided that $n-\sigma\geq 35$.
\section*{Appendix A: Mathematica Code}
Here, we will include the Mathematica code that verifies our minor arcs bound. An executable version of this can be found at \cite{Northey1}.

\includegraphics[scale=0.6855]{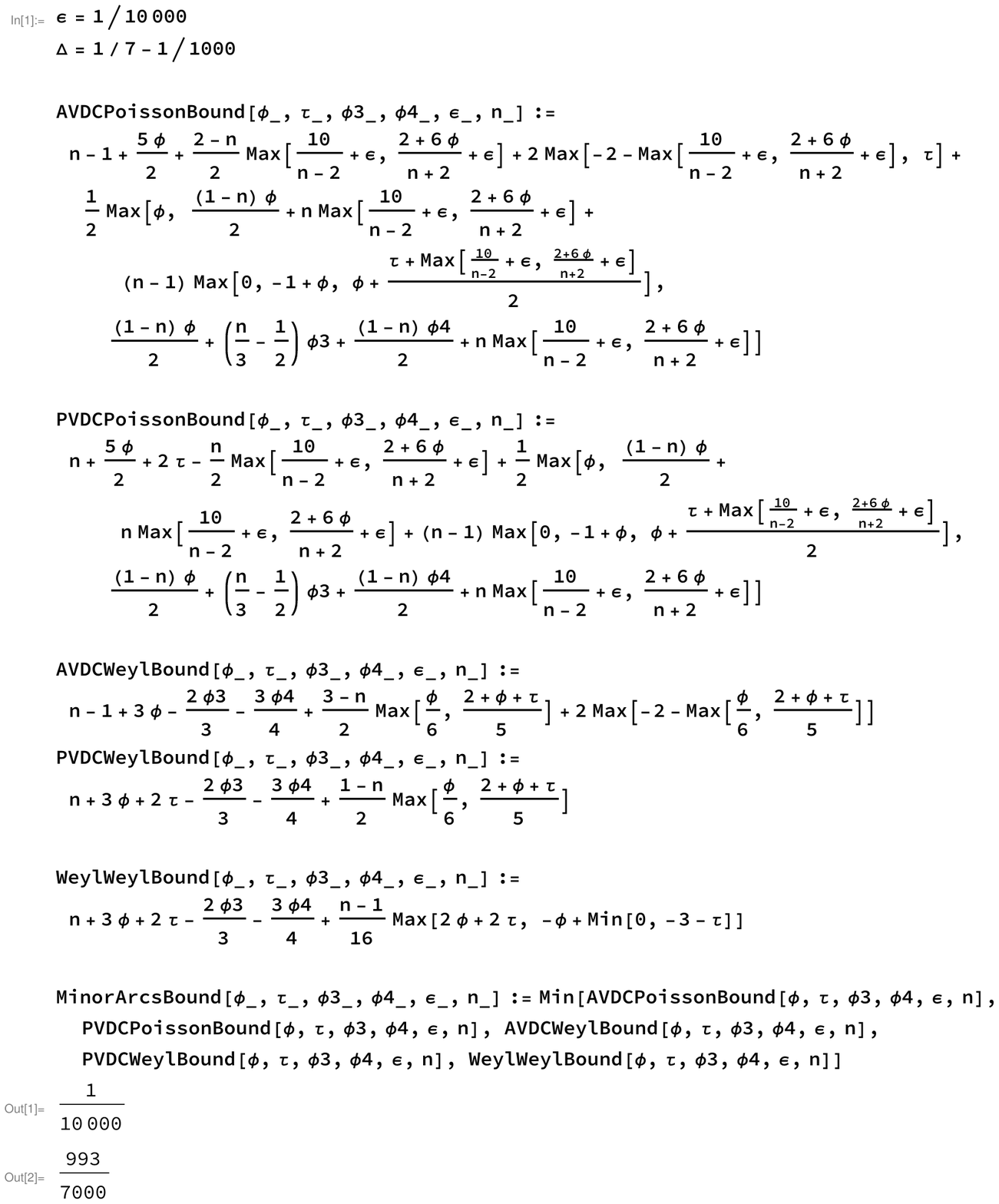}
\newpage
\includegraphics[scale=0.7]{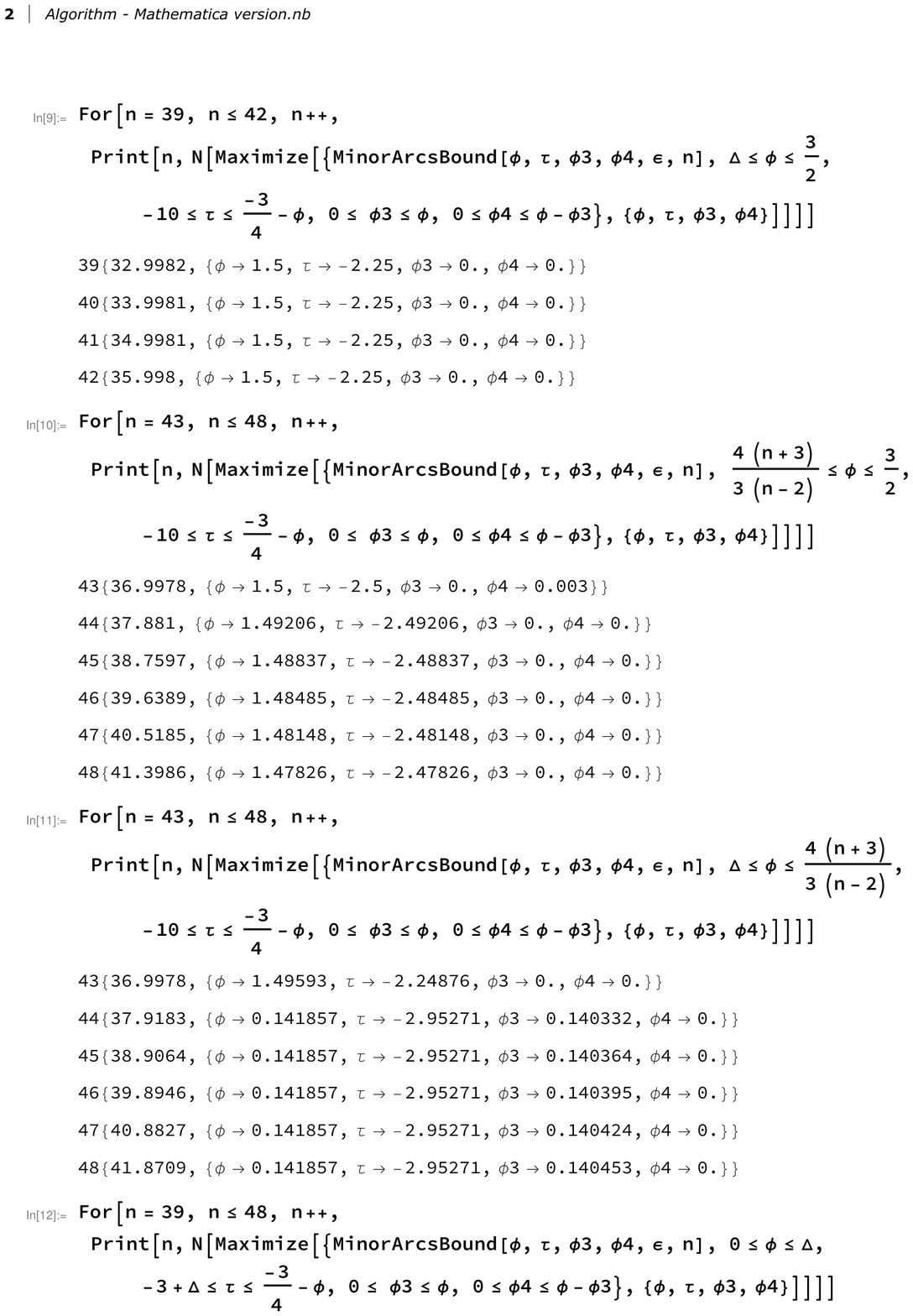}
\newpage
\includegraphics[scale=0.7]{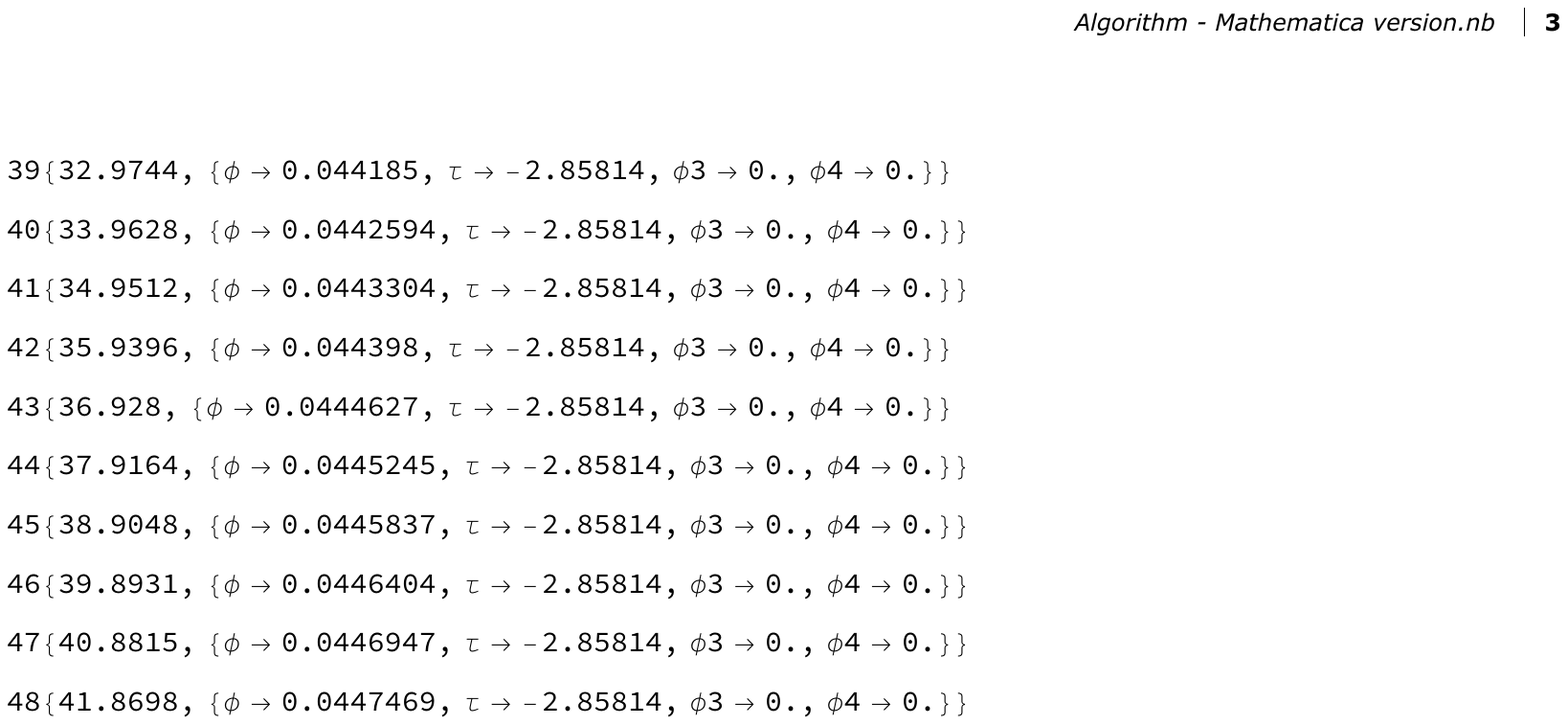}
\newpage

\bibliographystyle{plain}

\end{document}